\documentclass[12pt]{UOthesis}
\usepackage{geometry}                
\usepackage[parfill]{parskip}    
\usepackage{graphicx}
\usepackage{amssymb}
\usepackage{epstopdf}
\usepackage{amsmath}
\usepackage{amsthm}
\usepackage{dsfont}
\usepackage{csquotes}
\usepackage{float}
\usepackage[linesnumbered,vlined]{algorithm2e}

\newtheorem{thm}[theo]{Theorem}
\newtheorem{remark}[theo]{Remark}
\newtheorem{def2}[theo]{Definition}  
\newtheorem{algo}[theo]{Algorithm}
\newtheorem{question}[theo]{Question}

\newtheorem{problem}[theo]{Problem}
\newtheorem{conjecture}[theo]{Conjecture}

\AtEndEnvironment{defn}{\qed}

\newcommand{\F}{\mathcal{F}}
\newcommand{\G}{G_{\F}}
\newcommand\hidemath{$k$} 

\allowdisplaybreaks

\NoSignatures

\setUOname{Andrew Wagner}         
\setUOcpryear{2019}               
\setUOtitle{Eulerian Properties of\\ Design Hypergraphs and Hypergraphs with Small Edge Cuts}  

\phd                

{
  \newtheoremstyle{remarkstyle}{\topsep}{\topsep}{\rm}{}{\bfseries}{.}{.5em}{}
  \theoremstyle{remarkstyle}

}

\begin{document}

\baselineskip=21pt
\notachapter{Abstract}
An {\em Euler tour} of a hypergraph is a closed walk that traverses every edge exactly once; if a hypergraph admits such a walk, then it is called {\em eulerian}.  Although this notion is one of the progenitors of graph theory --- dating back to the eighteenth century --- treatment of this subject has only begun on hypergraphs in the last decade.  Other authors have produced results about {\em rank-2 universal cycles} and {\em 1-overlap cycles}, which are equivalent to our definition of Euler tours.  

In contrast, an {\em Euler family} is a collection of nontrivial closed walks that jointly traverse every edge of the hypergraph exactly once and cannot be concatenated simply.  Since an Euler tour is an Euler family comprising a single walk, having an Euler family is a weaker attribute than being eulerian; we call a hypergraph {\em quasi-eulerian} if it admits an Euler family.  Due to a result of Lov{\' a}sz, it can be much easier to determine that some classes of hypergraphs are quasi-eulerian, rather than eulerian; in this thesis, we present some techniques that allow us to make the leap from quasi-eulerian to eulerian.

A {\em triple system} of order $n$ and index $\lambda$ (denoted TS($n,\lambda$)) is a 3-uniform hypergraph in which every pair of vertices lies together in exactly $\lambda$ edges.  A {\em Steiner triple system} of order $n$ is a TS($n$,1).  We first give a proof that every TS($n,\lambda$) with $\lambda\geq 2$ is eulerian.  Other authors have already shown that every such triple system is quasi-eulerian, so we modify an Euler family in order to show that an Euler tour must exist.  We then give a proof that every Steiner triple system (barring the degenerate TS(3,1)) is eulerian.  We achieve this by first constructing a near-Hamilton cycle out of some of the edges, then demonstrating that the hypergraph consisting of the remaining edges has a decomposition into closed walks in which each edge is traversed exactly once.  

In order to extend these results on triple systems, we define a type of hypergraph called an {\em $\ell$-covering $k$-hypergraph}, a $k$-uniform hypergraph in which every $\ell$-subset of the vertices lie together in at least one edge.  We generalize the techniques used earlier on TS($n,\lambda$) with $\lambda\geq 2$ and define {\em interchanging cycles}.  Such cycles allow us to transform an Euler family into another Euler family, preferably of smaller cardinality.  We first prove that all
2-covering 3-hypergraphs are eulerian by starting with an Euler family that has the minimum cardinality possible, then demonstrating that if there are two or more walks in the Euler family, then we can rework two or more of them into a single walk.  We then use this result to prove by induction that, for $k\geq 3$, all $(k-1)$-covering $k$-hypergraphs are eulerian.

We attempt to extend these results further to all $\ell$-covering $k$-hypergraphs for $\ell\geq 2$ and $k\geq 3$.  Using the same induction technique as before, we only need to give a result for 2-covering $k$-hypergraphs.  We are able to use Lov{\' a}sz's condition and some counting techniques to show that these are all quasi-eulerian.

Finally, we give some constructive results on hypergraphs with small edge cuts.  There has been analogous work by other authors on hypergraphs with small vertex cuts.  We reduce the problem of finding an Euler tour in a hypergraph to finding an Euler tour in each of the connected components of the edge-deleted subhypergraph, then show how these individual Euler tours can be concatenated.
\cleardoublepage

\notachapter{Dedications}   
I dedicate this work to Irene Watt, a lifelong friend, teacher, coach, teammate, and advocate.  It is my utmost regret that this was not completed in time for her to see it.  
\cleardoublepage

\notachapter{Acknowledgement} 
Many thanks to the University of Ottawa and the Ontario Graduate Scholarship fund for their financial support.

Most of all, thank you to my thesis supervisor, Dr. Mateja \v{S}ajna, for your many contributions to this work.
\cleardoublepage

\tableofcontents
\cleardoublepage



%
%
%
\PrintNomenclature
\cleardoublepage

%


\pagenumbering{arabic}

\chapter{Introduction}


In this thesis, we explore some necessary and sufficient conditions for a hypergraph to admit an Euler tour.  We will focus particularly on hypergraphs that come from design theory, and as such, we borrow some nomenclature from hypergraphs and some from designs.  However, both of these kinds of objects --- from our perspective --- are solidly rooted in graph theory.

The problem of finding Euler tours in graphs is older than graph theory itself.  It begins, like so many disciplines in combinatorics, with a recreational puzzle that has come to be known as ``The Bridges of K{\"o}nigsberg''\cite{HW}.  As the (perhaps apocryphal) story goes, in the eighteenth century, residents of the Prussian city of K{\"o}nigsberg became fascinated with a simple game involving its seven bridges (see Figure \ref{fig:koenigsberg}).  It is said that they would spend an idle Sunday afternoon touring around the city, trying to devise a route that would lead them over each bridge exactly once.  However, despite their efforts, nobody could find a way to do so.

\begin{figure}[ht]
\centering
\includegraphics[scale=0.12]{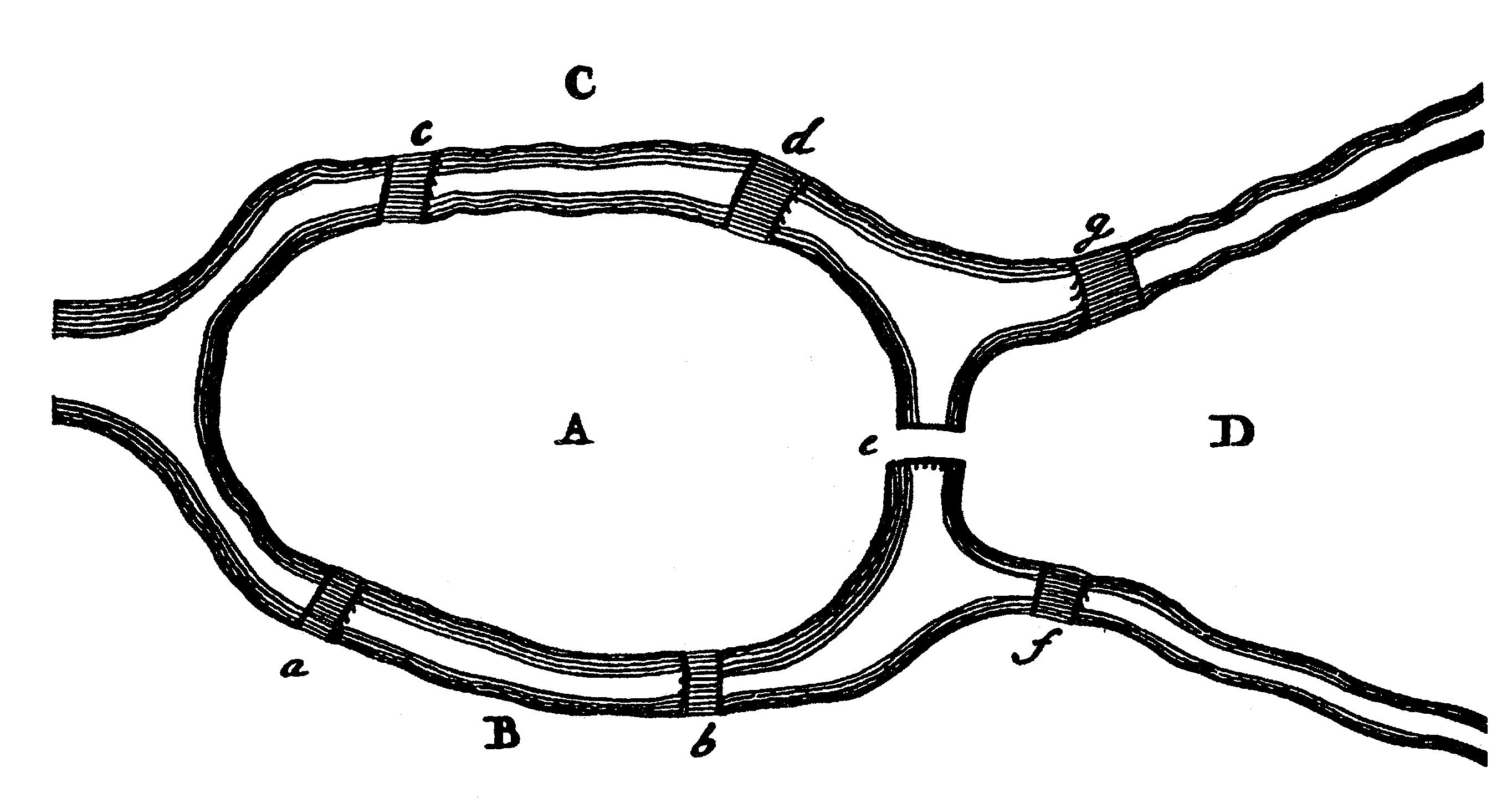}
\caption{Euler's diagram of K{\"o}nigsberg. \cite{E}\label{fig:koenigsberg}}
\end{figure}

As it happened, from 1735 to 1742, Leonhard Euler was in correspondence with Carl Leonhard Gottlieb Ehler, the mayor of what was then called Danzig (now the Polish city of Gda{\'n}sk) \cite{HW}.  From what physical evidence of this correspondence remains, we can surmise that Ehler may have introduced Euler to the K{\"o}nigsberg bridges problem, and implored him for a solution and complete proof.  Euler's response was an unequivocal rebuke:

\begin{displayquote}
{\em ``...Thus you see, most noble Sir, how this type of solution bears little relationship to mathematics, and I do not understand why you expect a mathematician to produce it, rather than anyone else, for the solution is based on reason alone, and its discovery does not depend on any mathematical principle. Because of this, I do not know why even questions which bear so little relationship to mathematics are solved more quickly by mathematicians than by others...''} \cite{HW}
\end{displayquote}

Euler gave the impression that this problem was beneath him.  However, the chronology of his other correspondences shows that even before he gave this icy rebuke to Ehler, Euler had confided in another mathematician that he had already solved it.

\begin{displayquote}
{\em ``This question is so banal, but seemed to me worthy of attention in that geometry, nor algebra, nor even the art of counting was sufficient to solve it. ... And so, after some deliberation, I obtained a simple, yet completely established, rule with whose help one can immediately decide for all examples of this kind, with any number of bridges in any arrangement, whether such a round trip is possible, or not...''} \cite{HW}
\end{displayquote}

Euler did indeed solve all problems of this kind with his method, and he invented a very primitive version of graph theory to do it, outlined in his 1736 paper \cite{E}.  His first idea was to focus on the land masses of K{\"o}nigsberg rather than the bridges directly.  He reasoned that, in a successfully constructed walk, each region would need to be visited a number of times equal to half of its bridges, rounded up.  However, with seven bridges, a walk can only visit eight regions (counting multiplicities), so one cannot round up for too many regions or a route becomes impossible.  Euler proved that, if such a walk was possible, each region had to have an even number of bridges incident with it, or else there could be exactly two with an odd number.  However, he evaded the proof of whether this is sufficient to guarantee a walk, concluding instead the following:

\begin{displayquote}
{\em ``When it has been determined that such a journey can be made, one still has to find how it should be arranged. For this I use the following rule: let those pairs of bridges which lead from one area to another be mentally removed, thereby considerably reducing the number of bridges; it is then an easy task to construct the required route across the remaining bridges, and the bridges which have been removed will not significantly alter the route found, as will become clear after a little thought. I do not therefore think it worthwhile to give any further details concerning the finding of the routes.''} \cite{HW}
\end{displayquote}

Though he was, of course, correct in that a simple greedy algorithm can solve the problem, a rigorous proof of his claims was not forthcoming until an 1871 paper, written and published posthumously by Carl Hierholzer \cite{H}.  Furthermore, it was not until 1878, in a paper of J. J. Sylvester's \cite{S}, that the word ``graph'' was uttered in any context resembling modern graph theory.  Sylvester was using graphs as a means of modeling molecules, so it did not even have anything to do with Euler's earlier work.

On the other hand, contrary to the story with graphs --- in which eulerian properties were discovered before graph theory was developed --- hypergraphs have been studied for the last few decades, yet very little work on Euler tours has cropped up.  It is our goal to make a move toward settling the matter, although it is not a simple problem and certainly not ``banal,'' as Euler put it!

This thesis is divided into broad chapters, as follows.  The proceeding chapter contains all the background information required to understand the graph theory and hypergraph theory used throughout.  We also include a (too-brief!) summary of results that pertain to Euler tours in hypergraphs, and it is there that we will outline exactly the results that we are advancing.  Each chapter thereafter represents a thorough investigation of one or more sufficient conditions to guarantee an Euler tour.  These results are, in some cases, related to one another, but we attempt to present them so that they may be read and appreciated independently of the others.
\cleardoublepage

\chapter{Preliminaries}

\section{Graphs}
We use Bondy and Murty's text \cite{BM} as a starting point for our graph theory definitions.  We will be careful to point out where our definitions diverge from theirs.\\

\begin{defn} {\rm {\bf (The Basics)} A {\em graph} $G=(V,E)$ is an ordered pair in which $V$ is a non-empty finite set of objects called {\em vertices}, and $E$ is a finite set of objects called {\em edges,} with $V\cap E=\emptyset$.  In the event that these sets are not named, we can also refer to the vertex set of $G$ by $V(G)$ and the edge set of $G$ by $E(G)$.

$G$ is also equipped with an {\em incidence function} $\psi$ that associates with each edge an unordered pair of (not necessarily distinct) vertices.  Let $e\in E(G)$ and suppose $\psi(e)=\{u,v\}$.  It is common to omit reference to the incidence function and instead say that $e = \{u,v\}$ if this does not cause ambiguity --- usually, this is when $\psi$ is injective.  If $\psi$ is not injective, then referring to $\psi(E)$ simply by $E$ results in $E$ being a multiset.  We generally will also omit the set braces in edges and write $e=uv$ instead of $e=\{u,v\}.$

If $e=uv$, then we say that $e$ is {\em incident with} $u$ and $v$ and that $e$ {\em joins} the vertices $u$ and $v$, which are also called the {\em ends} of $e$.  We call $u$ and $v$ {\em adjacent} ({\em via} edge $e$), or {\em neighbours}, since there is an edge joining them.  We may also call two edges {\em adjacent} if they have a vertex in common.
}
\end{defn}

Euler certainly did not invent the concept of graphs in his proof of the K{\"o}nigsberg bridge problem, but he had the same idea.  In his paper \cite{E}, the landmasses can be represented by vertices and the bridges can be represented by the edges of a graph.  Two vertices are adjacent if and only if the corresponding landmasses have a bridge from one to the other.  This can be represented visually as in Figure~\ref{fig:graph}.

\begin{figure}[ht]
\centering
\includegraphics[scale=0.12]{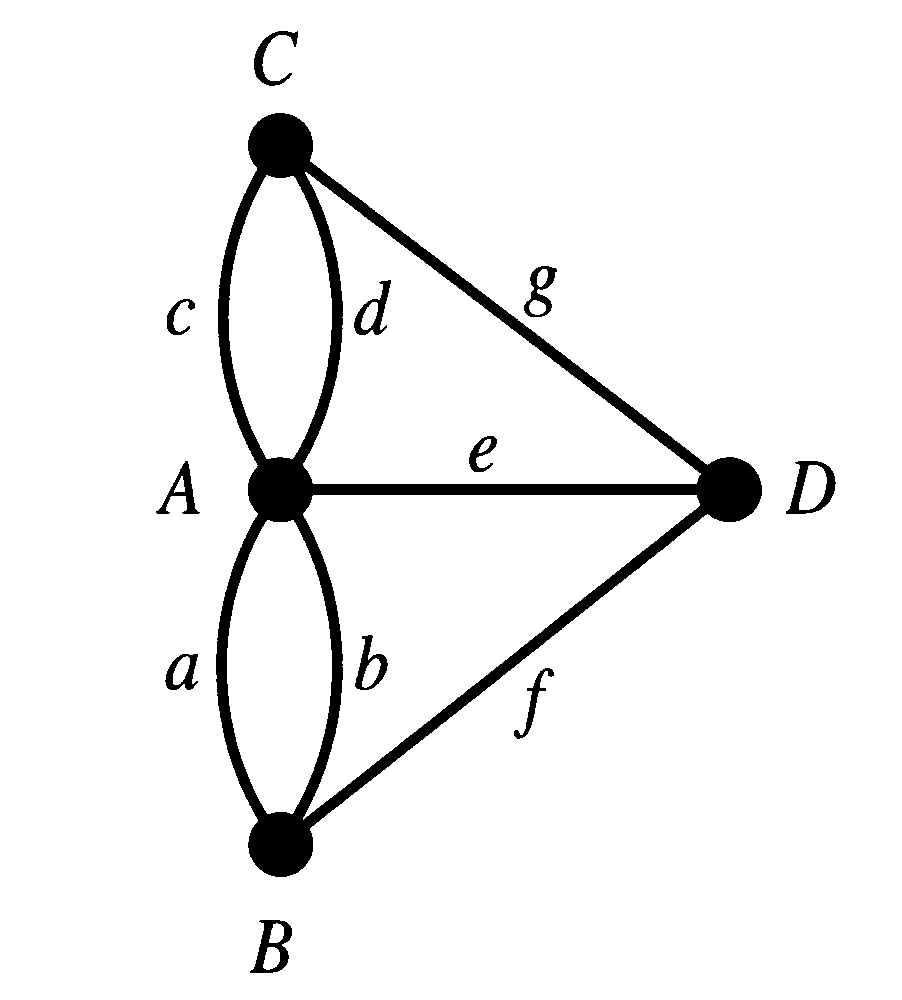}
\caption{A graph representation of K{\"o}nigsberg. \cite{HW} \label{fig:graph}}
\end{figure}

The {\em order} of a graph is its number of vertices, and the {\em size} of a graph is its number of edges.  A graph that has no edges is called {\em empty} or {\em edgeless}, while a graph with just a single vertex (but any number of edges) is called {\em trivial}.

An edge is called a {\em loop} if both of its ends are the same vertex, while two distinct edges are called {\em parallel} if they have the same ends.  Note that a vertex is not adjacent to itself unless it is in a loop.  A graph that has no loops and no parallel edges is called {\em simple}, whereas a graph that explicitly allows loops and parallel edges is called a {\em multigraph}.

The {\em degree} of a vertex $v$ in $G$, denoted $\deg_G(v)$ (or simply $\deg(v)$ if it is clear that we are talking about $G$), is the number of edges incident with $v$, although each loop counts as two incident edges, not one.  If every vertex of $G$ has the same degree $k$, then we call $G$ {\em $k$-regular}.  If every vertex of $G$ has even degree, then we call $G$ {\em even}.

One of the most rudimentary results in graph theory is called the Handshaking Lemma.  Euler implicitly discovered this result when he made his counting argument about a possible tour around K{\"o}nigsberg.\\

\begin{lem}\label{lem:handshaking}{\bf Handshaking Lemma} Let $G=(V,E)$ be a graph.  Then $$\sum\limits_{v\in V}\deg(v) = 2|E|.$$
\end{lem}

\begin{proof} We sum the degrees of each vertex of $G$ in two ways.  Summing over all $v\in V$ yields the quantity on the left-hand side.  On the other hand, each edge contributes 2 to $\sum\deg(v)$ as it is incident with exactly two vertices; therefore, summing the degrees of the ends of an edge $e$ over all $e\in E$ yields $2|E|$.
\end{proof}

\begin{defn}{\rm {\bf (Subgraphs)} Let $G=(V,E)$ be a graph.  A {\em subgraph of $G$} is a graph $G'=(V',E')$ with $V'\subseteq V$ and $E'\subseteq E$.  For any $V'\subseteq V$, we can define the subgraph of $G$ {\em induced by $V'$} as $G[V']=(V',\{e\in E: e\subseteq V'\})$.
}
\end{defn}

\begin{defn}{\rm {\bf (Bipartite Graphs)} $G$ is called {\em bipartite} if $V$ partitions into two sets $X$ and $Y$ such that every $e\in E$ has one end in $X$ and the other end in $Y$.  If this is the case, then $\{X,Y\}$ is called a {\em bipartition} of $G$.
}
\end{defn}

\begin{defn}{\rm {\bf (Operations on Graphs)} Let $G=(V,E)$ be a graph.  For $V'\subsetneq V$, the {\em vertex-deleted subgraph}, obtained by {\em deleting} $V'$ from $G$, is defined as $G-V'= G[V\setminus V']$.  If $V'$ consists of a single vertex $v$, then we may also write $G-v$ to denote $G-\{v\}$.  For $E'\subseteq E$, the {\em edge-deleted subgraph}, obtained by {\em deleting} $E'$ from $G$, is defined as $G\setminus E' = (V, E\setminus E')$.  If $E'$ comprises a single edge $e$, then we may also write $G\setminus e$ to denote $G\setminus\{e\}$.

Let $G_1=(V_1,E_1)$ and $G_2=(V_2,E_2)$ be graphs, neither of which has parallel edges.  The {\em symmetric difference} of $G_1$ and $G_2$, denoted $G_1\Delta G_2$, is the graph $(V_1\cup V_2, E_1\Delta E_2)$, where $E_1\Delta E_2$ is the symmetric difference of $E_1$ and $E_2$.
}
\end{defn}

We remark that our definition of symmetric difference of graphs $G_1$ and $G_2$ differs from that of Bondy and Murty \cite{BM}.  Their definition assumes that the vertex sets of $G_1$ and $G_2$ are the same; however, it will be more convenient for us if we allow $G_1$ and $G_2$ to have different vertex sets.\\

\begin{defn}\label{def:graphwalks}{\rm {\bf (Walks) } A {\em walk (of length $k\geq 0$),} in a graph $G$ is a sequence $W=v_0e_1v_1e_2\dotso v_{k-1}e_kv_k$ alternating between vertices $v_0,\dotso , v_k$ and edges $e_1,\dotso , e_k$, in which every pair of consecutive vertices of $W$ are adjacent via the interposing edge.  The first vertex $v_0$ and the last vertex $v_k$ are called the {\em initial} and {\em terminal} vertices of $W$, respectively, and the vertices $v_1,\dotso , v_{k-1}$ are called {\em internal} vertices of $W$.  We may be more specific by calling $W$ a $v_0v_k$-walk.  A {\em shortest} $v_0v_k$-walk is one of minimum length.  If $k=0$, then the walk consists of a single vertex with no edges and is called a {\em trivial} walk.  We say that $W$ {\em traverses} (or {\em visits}) the vertices $v_0,\dotso , v_k$ and edges $e_1,\dotso , e_k$ once for each time they appear in the sequence of $W$.\footnotemark

\footnotetext{If $v_0=v_k$, then the first and last traversal of $v_0$ count as just one.}

Note that if $G$ is a simple graph, it is not necessary to explicitly list the edges in the sequence of $W$.  Since each pair of consecutive vertices of a walk must be adjacent, the sequence of vertices alone uniquely determines the edges that are traversed.

There are a few different types of walks that $W=v_0e_1v_1e_2\dotso v_{k-1}e_kv_k$ can be:

\begin{description}
\item[(1)] If $v_0=v_k$, then $W$ is called {\em closed}; otherwise, it is {\em open}.  (In this exceptional case, the vertex $v_0$ is traversed a number of times equal to one less than the number of times it appears in the sequence of $W$.)
\item[(2)] If $e_1,\dotso , e_k$ are pairwise distinct, then $W$ is a {\em trail}.  (We may have $k=0$; if that is the case, then $W$ is a {\em trivial} trail.)
\item[(3)] If $v_0,v_1,\dotso , v_k$ are pairwise distinct, then $W$ is a {\em path}.
\item[(4)] If $e_1,\dotso , e_k$ and $v_0,v_1,\dotso , v_{k-1}$ are pairwise distinct with $k\geq 1,$ and $W$ is closed, then $W$ is a {\em cycle}.  
\end{description}

A cycle of length $k$ is often called a {\em $k$-cycle}, and a cycle that traverses every vertex of the graph is called a {\em Hamilton cycle}.

Any (contiguous) subsequence of a walk $W$ that is itself a walk can be called a {\em subwalk} of $W$.  Likewise, if $W$ is a path or trail, a subwalk of $W$ can be called a {\em subpath} or {\em subtrail} of $W$, respectively.

We can associate with any trail $T$ a graph $G(T)$ whose vertex and edge sets are equal to the sets of vertices and edges it traverses, respectively.  Note that, though vertices may be traversed more than once by $T$, they nevertheless appear only once in $G(T)$. 
}
\end{defn}




\begin{defn}{\rm {\bf (Connectivity)} A graph $G=(V,E)$ is called {\em connected} if, for all pairs of vertices $u,v\in V$, there exists a $uv$-walk in $G$.  Similarly, we may call two vertices $u$ and $v$ {\em connected (in $G$)} if $G$ admits a $uv$-walk.  
A {\em connected component} of $G$ is a maximal connected subgraph of $G$.  It is simple to see that connectedness is an equivalence relation on $V$, and hence the vertex sets of all the connected components of $G$ partition $V$.  The number of connected components of $G$ is denoted by $c(G)$.

A vertex $v\in V(G)$ is a {\em cut vertex} of $G$ if $c(G-v) > c(G)$; that is, if the deletion of $v$ increases the number of connected components.  Likewise, an edge $e\in E(G)$ is a {\em cut edge} of $G$ if $c(G\setminus e) > c(G)$.

A graph $G$ is called {\em $2$-edge-connected} if it is connected, non-trivial, and has no cut edge.

If $\emptyset\neq S\subsetneq V(G)$, then $[S,\overline{S}] = \{e\in E(G): e\cap S\neq\emptyset, e\cap\overline{S}\neq\emptyset\}$ is called an {\em edge cut} of $G$, where $\overline{S}=V(G)\setminus S$ denotes the complement of $S$ in $V(G)$.  
}
\end{defn}

\begin{remark}{\rm
It is easy to see that, for any two vertices $u,v$ in a graph $G$, that a $uv$-walk exists if and only if a $uv$-path exists.  To construct a $uv$-path from a $uv$-walk $W$, we need only delete any closed subwalks of $W$, so that the path does not traverse any vertex twice.  On the other hand, a $uv$-path is already a $uv$-walk.

Using this fact, we can say that $G$ is connected if and only if it admits a $uv$-{\em path}, for any $u,v\in V(G)$.
}
\end{remark}

We can now make a simple observation about connectivity in edge-deleted subgraphs that will come in handy later.\\

\begin{prop}\label{prop:G-uv} Let $G$ be a connected graph and $uv\in E(G)$.  If $G\setminus uv$ admits a $uv$-walk, then $G\setminus uv$ is connected.
\end{prop}

\begin{proof} Let $x, y\in V(G\setminus uv).$  Since $G$ is connected, there exists an $xy$-path $P$ in $G$.  If $P$ does not traverse the edge $uv$, then $P$ is an $xy$-walk in $G\setminus uv$.  Otherwise, without loss of generality, write $P=xP_1uvP_2y$, for some subpaths $P_1, P_2$ of $W$ that do not traverse $uv$.  Since we assume $G\setminus uv$ admits a $uv$-walk, let $W$ be a $uv$-walk in $G\setminus uv$.  Then $xP_1uWvP_2y$ is an $xy$-walk in $G\setminus uv$.  

Therefore, we conclude that $G\setminus uv$ is connected.
\end{proof}
Euler had a particular kind of walk in mind when he was considering the bridges of K{\"o}nigsberg.  He was looking for a trail that traversed every edge exactly once, as this corresponds to crossing every bridge exactly once.\\

\begin{defn}{\rm {\bf (Euler Trails and Euler Tours)} Let $G$ be a graph.  A trail that traverses every edge of $G$ exactly once is called an {\em Euler trail} of $G$ if it is open, or an {\em Euler tour} of $G$ if it is closed.  If $G$ has an Euler tour, then $G$ is called {\em eulerian.}
}
\end{defn}
We note that other sources may refer to an Euler trail as an ``Eulerian trail'' while an Euler tour is often called an ``Eulerian circuit.''  We will reserve the adjectival form {\em eulerian} (with a lower-case ``e'') to describe to a graph that admits an Euler tour, but never to refer to the trail itself.

Finally, we conclude this section with a rudimentary result about eulerian graphs, which we will make repeated use of.\\

\begin{prop}\label{prop:2-edge-connected}
Let $G=(V,E)$ be a non-trivial connected eulerian graph.  Then $G$ is 2-edge-connected.
\end{prop}

\begin{proof}
Since $G$ is non-trivial and connected, its Euler tour contains an edge, so $|E|\geq 1$.  Let $e=uv\in E$ be arbitrary, and let $T$ be an Euler tour of $G$.  Let $T'$ be a shortest closed subtrail of $T$ traversing $e$, and without loss of generality, write $T'=ueve_1v_1\dotso v_{k-1}e_ku$.  Since $T'$ is a trail, it does not traverse $e$ twice.

Then $W=ue_kv_{k-1}\dotso v_1e_1v$ does not traverse $e$ at all, so $W$ is a $uv$-walk in $G\setminus uv$.  Then we conclude that $G\setminus uv$ is connected by Proposition~\ref{prop:G-uv}.  Then $uv$ is not a cut edge, and since $uv$ was arbitrary, there are no cut edges in $G$.

Hence $G$ is 2-edge-connected.
\end{proof}

\begin{remark}\label{remark:closedtrails}{\rm
Let $T$ be a non-trivial closed trail and $G(T)$ be the graph associated with $T$.  Then $G(T)$ is connected since $T$ traverses each of its vertices, and $G(T)$ is eulerian, evidenced by the fact that $T$ is a closed trail in $G(T)$ that traverses each of its edges exactly once.  Once we start talking about hypergraphs, we will regularly use Proposition~\ref{prop:2-edge-connected} alongside this fact to show that $G(T)$ is 2-edge-connected.
}
\end{remark}

\section{Hypergraphs}

Hypergraphs, as we use them, are meant to be generalisations of graphs.  Therefore, we would like it if any definition that applies to graphs can also apply to hypergraphs.  Sometimes, there are multiple ways to extend a concept or we need to define something that has no analogy in graphs; we will turn our attention to these differences.  

There is no definitive authority on hypergraph concepts, but there are commonalities in the different texts available.  We follow the definitions put forth by Bahmanian and {\v S}ajna in \cite{BS1}, but these are by no means universal.

\subsection{The Basics}
\begin{defn}{\rm \cite{BS1} {\bf (Basic Terminology)} A {\em hypergraph} $H$ is a pair $(V, E)$ with $V\cap E=\emptyset$, where $V=V(H)$ is a non-empty finite set of {\em vertices} and $E=E(H)$ is a finite set of {\em edges}.  Each edge is associated with a subset of $V$ via the {\em incidence function} $\psi: E\rightarrow 2^V$.  As in graphs, we can often omit the incidence function and simply say that each edge $e$ is itself a subset of $V$ (given by $\psi(e)$). However, referring to $\psi(E)$ as $E$ results in $E$ being a multiset if $\psi$ is not injective, so some care is required. We may further represent each edge as a set with set braces and commas omitted: for example, we may write $e=uvw$ instead of $e=\{u,v,w\}$, so long as this does not cause confusion.  Since we regard each edge as a set, certain set-theoretic terms ({\em e.g. singleton, pair, triple, empty}) and operations (generally union and intersection) apply to them as well.

For any pair of distinct vertices $u,v\in V(H)$, we say that $u$ and $v$ are {\em adjacent} to each other ({\em via} an edge $e\in E(H)$) --- and neighbours of each other --- if they are contained in an edge $e$ together.  In that case, we also say that $e$ {\em joins} $u$ and $v$, and that $e$ is {\em incident} with each of the vertices it contains.  Two edges can also be called {\em adjacent} to each other if they are incident with a common vertex.  The {\em degree} of a vertex $v\in V(H)$, denoted $\deg_H(v)$ (or $\deg(v)$ if the choice of hypergraph is clear) is the number of edges in $H$ that are incident with $v$.

Two distinct edges $e,f\in E(H)$ are {\em parallel} if $\psi(e)=\psi(f)$.  A hypergraph is called {\em simple} if it has no parallel edges.

The set of {\em flags} of $H$, denoted $F=F(H)$, is the collection of all pairs $(v,e)$ with $v\in V, e\in E,$ and $v\in e$.  Note that if $E$ is considered a multiset, then $F$, too, is a multiset.  In that case, the multiplicity of $(v,e)$ in $F$, for each $v\in e$, is inherited from the multiplicity of $e$ in $E$.
}
\end{defn}

\begin{defn}{\rm \cite{BS1} {\bf (Properties of Hypergraphs)} Let $H=(V,E)$ be a hypergraph.  The {\em order} of $H$ is $|V|$ and is usually denoted by $n$.  The {\em size} of $H$ is the number of edges of $H$ (counting multiplicities if $E$ is considered a multiset), and is usually denoted by $|E|=m$.  If $|V|=1$, then $H$ is {\em trivial}; if $|E|=0$, then $H$ is called {\em empty} or {\em edgeless}.

If every edge of $H$ has cardinality $k$, then $H$ is called {\em $k$-uniform}.  If every vertex of $H$ has degree $r$, then $H$ is called {\em $r$-regular}.  $H$ is called {\em linear} if no two edges of $H$ have more than one vertex in common.
}
\end{defn}

\begin{defn}{\rm {\bf (Design Hypergraphs)} A {\em triple system of order $n$ and index $\lambda$}, denoted TS($n,\lambda$), is a non-empty, 3-uniform hypergraph $H$ in which every pair of vertices lie together in exactly $\lambda$ edges.  A TS($n$,1) is also called a {\em Steiner triple system}, and can be denoted STS($n$).  A TS($n$,2) is also called a {\em twofold triple system.}

More generally, an $\ell$-covering $k$-hypergraph is a non-empty, $k$-uniform hypergraph $H$ in which every $\ell$-subset of the vertex set lie together in at least one edge.  If $\ell=k-1$, then we may refer to this simply as a {\em covering $k$-hypergraph}.
}
\end{defn}

We refer to the kinds of hypergraphs defined above as ``design hypergraphs,'' since these concepts come from design theory.  We also call the set of all $\ell$-covering $k$-hypergraphs as ``covering hypergraphs.''  Observe that triple systems are instances of covering hypergraphs.  Eulerian properties of triple systems, and particularly of Steiner triple systems, have been well studied by other authors, as we will see in the following chapter.

\subsection{Subhypergraphs}

Though the notion of subgraphs is quite simple, there are many ways in which we can construct a hypergraph that is some substructure of another hypergraph.  The canonical substructure that we can construct is called a {\em subhypergraph}, but this is not always the most useful possibility.\\

\begin{defn}{\rm \cite{BS1,DPP} {\bf (Subhypergraphs and Strong Subhypergraphs)} Let $H=(V,E)$ be a hypergraph.  A {\em subhypergraph} of $H$ is a hypergraph $H'=(V', E')$ with $V'\subseteq V$ and $E'=\{e\cap V': e\in E''\}$ for some $E''\subseteq E$.  In other words, the vertex set of $H$ is restricted to $V'$, we remove vertices of $V\setminus V'$ from each edge, and then discard whichever edges originated from $E\setminus E''$.    

A subhypergraph of $H$ {\em induced} by $V'\subseteq V$ is the subhypergraph $H[V']=(V',E')$ where $E'=\{e\cap V': e\in E, e\cap V'\neq\emptyset\}$.

For any $V'\subsetneq V$, we can {\em delete} $V'$ from $H$ to obtain the {\em vertex-deleted subhypergraph} $H-V'=H[V\setminus V']$.  If $V'=\{v\}$ is singleton, then we may simply write $H-v=H[V\setminus\{v\}]$.


A subhypergraph $H'=(V', E')$ of $H$ is called {\em strong} if $V'\subseteq V$ and $E'\subseteq E$.

If $e\in E$, then $H\setminus e$ denotes the {\em edge-deleted subhypergraph} $(V, E\setminus \{e\})$.  More generally, if $E'\subseteq E$, then $H\setminus E'$ denotes the hypergraph $(V,E\setminus E')$, and such a subhypergraph is also known as a {\em partial} hypergraph of $H$. \cite{D}

A {\em spanning} subhypergraph of $H=(V,E)$ is one whose vertex set is $V$.

We next define hypergraphs obtained by adding edges to $H$.  If $e\in 2^V$, then $H+e$ is obtained from $H$ by adjoining a new copy of the edge $e$ to $E$.  Similarly, if $E'$ is a collection of subsets of $2^V$, then $H+E'$ is the hypergraph obtained from $H$ by adjoining a new copy of each edge in $E'$ to $E$.

If $H_1=(V_1,E_1)$ and $H_2=(V_2,E_2)$ are hypergraphs, then we may define their union $H_1\cup H_2=(V_1\cup V_2, E_1\cup E_2)$, in which $V_1\cup V_2$ is set union and $E_1\cup E_2$ is multiset union.
}
\end{defn}

\begin{remark}{\rm Note that definition of the vertex-induced subhypergraph $H[V']$ does {\em not} extend from the same definition on graphs.  To obtain $H[V']$ from $H$, we delete the vertices of $V(H)\setminus V'$ from $H$ and from each edge of $H$, then clean up the hypergraph by removing any resultant empty edges.  In a graph, such an operation could result in edges of cardinality 1, which is not allowed.
}
\end{remark}

\subsection{Graphs Associated with Hypergraphs}
\begin{defn}{\rm \cite{BS1} {\bf (Incidence Graph)} The {\em incidence graph} $\mathcal{G}(H)=(V_G, E_G)$ of a hypergraph $H$ is defined as follows:
\begin{itemize}
\item $V_G = V\cup E$;
\item $E_G = \{ve: (v,e)\in F(H)\}$.
\end{itemize}

In $V_G$, the vertices from $V$ are called {\em v-vertices} and the vertices from $E$ are called {\em e-vertices}. 
}
\end{defn}

\begin{remark}{\rm
For a hypergraph $H=(V,E)$, we can observe that $\mathcal{G}(H)$ is simple and bipartite with bipartition $(V,E)$.  We will later see that some subgraphs of $\mathcal{G}(H)$ correspond to subhypergraphs of or walks in $H$.
}
\end{remark}

The following definition originates from \cite{D}, but we have modified it (see below).\\

\begin{defn}\label{def:2-section}{\rm \cite{D} {\bf (2-section)} Let $H=(V,E)$ be a hypergraph with incidence function $\psi$.  The {\em 2-section} of $H$ is a graph $G$, not necessarily simple, constructed as follows:
\begin{itemize}
\item $V(G)=V(H)$; and
\item For all distinct $u,v\in V$, the multiplicity of $uv$ in $E(G)$ is $|\{e\in E(H): \{u,v\}\subseteq \psi(e)\}|$.
\end{itemize}
}
That is, we have exactly one edge $uv$ in $G$ for each edge $e\in E(H)$ that is incident with both $u$ and $v$.
\end{defn}

We will generally be interested in the definition of a 2-section that allows for multiple edges.  However, many sources ({\em e.g.} \cite{D}, \cite{V}) consider the underlying simple graph instead.  We shall call such a graph the {\em simple 2-section} (of a hypergraph $H$), and consider the non-simple version to be the default.\\

\begin{defn}\label{defn:linegraph}{\rm \cite{BS1} {\bf (Line Graph)} Let $H=(V,E)$ be a hypergraph.  The {\em line graph} (or {\em intersection graph}) of $H$, denoted $\mathcal{L}(H)$, is the (simple) graph with vertex set $E$ and edge set $\{ef: e, f\in E,\text{ $e$ and $f$ are adjacent to each other in $H$}\}.$

More generally, for any positive integer $\ell$, the {\em level-$\ell$ line graph} of $H$, denoted $\mathcal{L}_{\ell}(H)$, is the (simple) graph with vertex set $E$ and edge set $\{ef: e,f\in E,e\neq f, |e\cap f|= \ell\}$.  The {\em level-$\ell^*$ line graph} of $H$ is denoted $\mathcal{L}^*_{\ell}(H)$ and is the (simple) graph with vertex set $E$ and edge set $\{ef: e,f\in E,e\neq f, |e\cap f|\geq \ell\}$.
}
\end{defn}

\subsection{Walks}
\begin{defn}\label{def:euler}{\rm \cite{BS1} {\bf (Walks, Paths, Trails, and Cycles)} Let $H=(V,E)$ be a hypergraph.  A {\em walk (of length $k\geq 0$),} in $H$ is defined as a sequence $W=v_0e_1v_1\dotso v_{k-1}e_kv_k$ such that $v_0,\dotso ,v_k\in V$, and $e_1,\dotso e_k\in E$; for $i=1,\dotso ,k$, we have that $v_i$ is adjacent to $v_{i-1}$ via edge $e_i$.  In this case, we may call $W$ specifically a $v_0v_k$-walk.  

We say that $W$ {\em traverses} (or {\em visits}) the vertices $v_0,\dotso , v_k$ and {\em traverses} edge $e_i$ ({\em via vertices} $v_{i-1}$ and $v_i$), for all $i=1,\dotso,k$, each time they appear in the sequence of $W$.  (If $v_0=v_k$, then we count the first and last traversals of $v_0$ as just one.)  The vertices that are traversed by $W$ form the set of {\em anchors} $V(W)$ of $W$.  Any vertex that is in $e_1\cup \dotso\cup e_k$ that is not an anchor is called a {\em floater} vertex of $W$. In $W$, an {\em anchor flag} is a flag $(v_i,e_i)$ or $(v_{i-1},e_i)$ for $i=1,\dotso ,k$, and each of these flags $(v,e)$ is traversed once for each time $ve$ or $ev$ appears in the sequence of $W$.  The edges $e_1,\dotso ,e_k$ traversed by $W$ form the set $E(W)$ of edges of $W$.  

If $\mathcal{W}$ is a collection of walks, then we may say that $\mathcal{W}$ {\em traverses} a vertex $v$ (or {\em traverses} an edge $e$ {\em via vertices} $u$ and $v$) if there exists a walk $W\in\mathcal{W}$ that traverses $v$ (or traverses $e$ via $u$ and $v$, respectively).  The number of times that a vertex or edge is traversed by $\mathcal{W}$ is equal to the sum of the number of times each $W\in\mathcal{W}$ traverses that vertex or edge.

As is the case for graphs, there are a few different types that a walk $W=v_0e_1v_1\dotso v_{k-1}$ $e_kv_k$ can be, as follows:

\begin{description}
\item[(1)] If $v_0=v_k$, then $W$ is called {\em closed}; otherwise, it is {\em open}.
\item[(2a)] If $W$ does not traverse any anchor flags more than once, then it is called a {\em trail}.
\item[(2b)] If $e_1,\dotso , e_k$ are pairwise distinct, then $W$ is called a {\em strict trail}.
\item[(3a)] If $v_0,\dotso , v_k$ are pairwise distinct and $W$ does not traverse any anchor flags more than once, then $W$ is called a {\em pseudo-path}.
\item[(3b)] If $W$ is a pseudo-path and we further have that $e_1,\dotso , e_k$ are distinct, then $W$ is called a {\em path}.
\item[(4a)] If $k\geq 2$, and $v_0,\dotso , v_{k-1}$ are pairwise distinct, and $v_0=v_k$, then $W$ is called a {\em pseudo-cycle}.
\item[(4b)] If $W$ is a pseudo-cycle and we further have that $e_1,\dotso , e_k$ are distinct, then $W$ is called a {\em cycle}.
\end{description}
}
\end{defn}

\begin{remark}\label{remark:traversals}{\rm
As the reader will note, there are many more categories of walks in hypergraphs than there are for graphs.  These divergent definitions represent alternative ways of extending notions from graphs to hypergraphs.  We will generally disregard pseudo-paths, pseudo-cycles, and trails; and we will be much more interested in paths, cycles, and strict trails.  Each of these traversals corresponds to a traversal in the incidence graph.

Let $W=v_0e_1v_1\dotso v_{k-1}e_kv_k$ be a walk in a hypergraph $H$. Then $W'=v_0e_1v_1\dotso$ $v_{k-1}e_kv_k$ is a walk in the incidence graph $G$ of $H$.  (Note that, since $G$ is a simple graph, it suffices to list the vertices of $W'$.)
\begin{description}
\item[(1)] If $W$ is open (closed), and it corresponds to a traversal in $G$ as described below, then $W'$ is open (closed).
\item[(2a)] $W$ is a trail in $H$ if and only if $W'$ is a trail in $G$. \cite[Lemma 3.6]{BS1}
\item[(2b)] $W$ is a strict trail in $H$ if and only if $W'$ is a trail in $G$ that traverses each e-vertex at most once. \cite[Lemma 3.6]{BS1}
\item[(3)] $W$ is a path in $H$ if and only if $W'$ is a path in $G$. \cite[Lemma 3.6]{BS1}
\item[(4)] $W$ is a cycle in $H$ if and only if $W'$ is a cycle in $G$. \cite[Lemma 3.6]{BS1}\\
\end{description}
Note that if $W'$ is a trail, then it corresponds to a subgraph $G(W')$ of $G$.  We can also skip this intermediate step involving $W'$ and define $G_W = G(W')$, the subgraph of $G$ associated with $W$.  This will get used quite often to prove rudimentary relationships between traversals in $H$ and subgraphs of $G$ in Chapter~\ref{chapter:previousresults}.\\
}
\end{remark}

\begin{defn}{\rm {\bf (Concatenation of (Sub)walks)} If $W_1=v_0e_1v_1\dotso e_kv_k$ is a $v_0v_k$-walk in a hypergraph $H$ and $W_2=v_ke_{k+1}v_{k+1}\dotso e_{\ell-1}v_{\ell}$ is a $v_kv_{\ell}$-walk in $H$, then the {\em concatenation} of $W_1$ and $W_2$ is the $v_0v_{\ell}$-walk $v_0e_1v_1\dotso v_{k-1}e_k$ $v_ke_{k+1}v_{k+1}\dotso v_{\ell-1}e_{\ell}v_{\ell}$, denoted $W_1W_2$.
}
\end{defn}

\begin{remark}{\rm If $T_1=v_0e_1 \dotso e_kv_0$ and $T_2=u_0f_1\dotso f_{\ell}u_0$ are edge-disjoint closed strict trails of $H$ such that $v_i=u_0\in V(T_1)\cap V(T_2)$, then we may concatenate $T_1$ and $T_2$.  The concatenation of $T_1$ and $T_2$ is a strict closed trail $T=v_0e_1\dotso e_iv_iT_2v_ie_{i+1}\dotso e_k$ $v_0$.
}
\end{remark}

\subsection{Connectivity}

\begin{defn}{\rm \cite{BS1} {\bf (Connectedness)} A hypergraph $H=(V,E)$ is said to be {\em connected} if, for every $u, v\in V$, there exists a $uv$-walk.  Two vertices $u,v\in V$ are called {\em connected} if there exists a $uv$-walk in $H$.  A {\em connected component} of $H$ is a maximal connected strong subhypergraph of $H$ that does not contain any empty edges. Observe that connectedness is an equivalence relation on $V$, and so the vertex sets of each connected component of $H$ partition $V$.

The number of connected components of $H$ is denoted by $c(H)$.}
\end{defn}

\begin{defn}{\rm \cite{BS1} {\bf (Cut Edges and Cut Vertices)} Let $H$ be a hypergraph.  A {\em cut edge} of $H$ is an edge $e\in E(H)$ such that $c(H\setminus e)>c(H)$.  If $c(H\setminus e) = c(H) + |e| - 1$, then $e$ is called a {\em strong cut edge}.  A cut edge is {\em trivial} if the number of non-trivial connected components in $H\setminus e$ is the same as in $H$.

A {\em cut vertex} of $H$ is a vertex $v\in V(H)$ such that $c(H - v)>c(H)$.
}
\end{defn}

\begin{remark}\label{remark:cutedge}{\rm \cite[Lemma 3.15, Theorem 3.17]{BS1} Note that for any edge $e$ in a hypergraph $H$, we have $c(H\setminus e)\leq c(H) + |e| - 1$. If $e$ is a strong cut edge, then not only does this hold with equality, but we have that the vertices of $e$ lie in distinct connected components of $H\setminus e$.
}
\end{remark}

There is a nice relationship between the cut vertices or edges of a hypergraph and the cut vertices of its incidence graph, as follows.\\

\begin{thm}{\rm \cite[Theorem 3.23]{BS1}}\label{thm:cutvvertex} Let $H$ be a non-trivial hypergraph whose edges all have cardinality at least 2, and let $G$ be the incidence graph of $H$.  
\begin{description}
\item[(1)]For any $v\in V(H)$, we have that $v$ is a cut vertex of $H$ if and only if $v$ is a cut vertex of $G$.
\item[(2)]For any $e\in E(H)$, we have that $e$ is a cut edge of $H$ if and only if $e$ is a cut vertex of $G$.\qed
\end{description} 
\end{thm}

\vspace{-28px}We also define {\em edge cuts} analogously to how we define them for graphs.  We will discuss edge cuts more in Chapter~\ref{chapter:edgecuts}.\\

\begin{defn}{\rm {\bf (Edge Cut)} Let $H$ be a hypergraph and let $S\subsetneq V(H)$ be non-empty.  Then $[S,\overline{S}] = \{e\in E(H): e\cap S\neq\emptyset, e\cap\overline{S}\neq\emptyset\}$ is called an {\em edge cut} of $H$, where $\overline{S}$ denotes $V(H)\setminus S$, the complement of $S$ in $V(H)$.
}
\end{defn}

\subsection{Duality}
One advantage of hypergraphs is that the set of non-empty hypergraphs is closed under duals, something that is not true of the set of graphs.\\

\begin{defn}{\rm {\bf (Hypergraph Dual)}\label{def:dual} Let $H=(V,E)$ be a hypergraph, where $V=\{v_1,\dotso ,v_n\}$ and $E=\{e_1,\dotso ,e_m\}$ with $m\geq 1$.  Then the {\em dual} hypergraph of $H$ is the hypergraph $D=(V(D),E(D))$ with incidence function $\psi$, defined as follows.

\begin{itemize}
\item $V(D)=E(H)$;
\item $E(D)=\{e_v: v\in V(H)\}$;
\item $\psi(e_v) = \{e\in V(D): e\text{ is incident with $v$ in $H$}\}$ for all $e_v\in E(D)$.
\end{itemize}

When the dual of $H$ is defined, we also refer to $H$ as the {\em primal} hypergraph.
}
\end{defn}
There are many facts about dual hypergraphs that are immediate.  First, any 2-uniform hypergraph ({\em i.e.} loopless graph) has a 2-regular dual.  In general, the dual of any $k$-uniform hypergraph is $k$-regular.  A second immediate observation is that the incidence graphs of a hypergraph and its dual are isomorphic (and obtained from each other by exchanging the roles of v-vertices and e-vertices).  

\subsection{Euler Tours and Families}

\begin{defn}{\rm \cite{BS2} {\bf (Euler Trails, Euler Tours, Euler Families)} Let $H$ be a hypergraph.  A strict trail $T$ of $H$ that traverses every edge of $H$ exactly once is called an {\em Euler trail} if $T$ is open, and an {\em Euler tour} if $T$ is closed.  A hypergraph that admits an Euler tour is called {\em eulerian}.  

An {\em Euler family} $\F$ is a collection of pairwise edge-disjoint, anchor-disjoint non-trivial strict closed trails of $H$ such that $\F$ traverses every edge of $H$ (exactly once, necessarily).  The trails of $\F$ are called the {\em components} of $\F$.  A hypergraph that admits an Euler family is called {\em quasi-eulerian}.  $\F$ is called {\em minimum} if there does not exist an Euler family of $H$ with fewer closed trails than $\F$ has.  If $|\F|\leq 1$, then we associate $\F$ with an Euler tour of $H$: the single trail of $\F$ is an Euler tour if $|\F|=1$, and any trivial walk of $H$ is an Euler tour if $|\F|=0$.

An Euler trail, Euler tour, or Euler family is called {\em spanning} if it traverses every vertex of $H$.
}
\end{defn}

To familiarise ourselves with some of these new definitions, we present a basic result to check when certain small (in size) hypergraphs can have Euler tours or families.\\

\begin{lem}\label{lem:trivialcases}
Let $H$ be a connected hypergraph.
\begin{description}
\item[(1)] If $|E(H)|=0$, then any trivial walk of $H$ is an Euler tour of $H$, and $\emptyset$ is an Euler family of $H$.
\item[(2)] If $|E(H)|=1$, then $H$ does not have an Euler family.
\end{description}
\end{lem}

\begin{proof}
{\bf (1)} Assume $|E(H)|=0$.  Observe that a trivial walk of $H$ satisfies the definition of an Euler tour of $H$, and $\emptyset$ is a collection of anchor-disjoint, edge-disjoint closed strict trails that traverses every edge of $H$.

{\bf (2)} Assume $|E(H)|=1$, and suppose $\F$ is an Euler family of $H$.  Then $\F$ is non-empty since $H$ is non-empty, and there must be a single closed strict trail of $\F$ that traverses one edge, so it must be of the form $T=vev$.  However, this is not a trail, since consecutive anchor vertices must be different --- a contradiction.  Hence $H$ does not have an Euler family.
\end{proof}

Analogous to the explosion of definitions for traversals when moving from graphs to hypergraphs, there are multiple extensions of the term ``Euler tour'' from graphs.  In addition to the Euler tour and Euler family definitions given above, we can also define a {\em flag-traversing tour} to be a trail (not necessarily strict) that traverses every flag of the hypergraph exactly once \cite{BS2}.   Applying the hypergraph definitions of Euler tour, flag-traversing tour, and Euler family to a connected loopless graph yields essentially the same object, which we will see is not at all the case for hypergraphs.  We can demonstrate the coincidence of Euler families and Euler tours for loopless connected graphs with the following lemma.\\

\begin{lem}\label{lem:eulerfamilyingraphs} Let $G$ be a 2-uniform connected hypergraph, and let $\F$ be an Euler family of $G$.  Then $|\F|\leq 1,$ so $\F$ corresponds to an Euler tour of $G$.
\end{lem}

\begin{proof} Let $\F=\{T_1,\dotso , T_k\}$ be an Euler family of $G$.  By definition, we have that $T_1,\dotso , T_k$ are mutually anchor-disjoint (and hence vertex-disjoint), yet jointly traverse every edge of $G$.  Hence $k\geq 2$ implies that $G$ is not connected, a contradiction.  Then $k=0$ or $k=1$, so $\F$ corresponds to an Euler tour.
\end{proof}

On the other hand, Euler tours, Euler families, and flag-traversing tours are not only quite different concepts from each other for hypergraphs, but they give rise to fundamentally different problems.  We will not be interested in flag-traversing tours at all, and likewise we will find no use for (non-strict) trails.  In subsequent chapters, we shall say {\em trails} to exclusively mean {\em strict trails}.
\cleardoublepage

\chapter{Previous Results}\label{chapter:previousresults}

\section{Euler Tours in Graphs}
We begin with one of many characterizations of eulerian graphs, first proved by Veblen in 1912.  A {\em cycle decomposition} of a graph $G$ is a collection $\{C_1,\dotso , C_k\}$ of cycles in $G$ such that every edge of $G$ is traversed in exactly one cycle.

Although Veblen's Theorem was proven after Carl Hierholzer proved Euler's Theorem about the Bridges of K\"{o}nigsberg, it can be used to prove Euler's Theorem quite simply.  We present original proofs of these theorems using the modern notation we have cultivated so far.\\

\begin{thm}\label{thm:veblen}{\rm (Veblen's Theorem, 1912 \cite{Veblen}) Let $G$ be a connected graph.  Then $G$ is eulerian if and only if $G$ admits a cycle decomposition.
}
\end{thm}

\begin{proof}
$\Rightarrow$: Let $T = v_0e_1v_1\dotso v_{m-1}e_mv_0$ be an Euler tour of $G$.  For any closed trail $W=u_0f_1u_1\dotso u_{r-1}e_ru_0$, define a parameter $\ell_W = r - |\{u_0,\dotso , u_{r-1}\}|$; that is, $\ell_W$ is the number of indices $i\in\{0,\dotso , r-1\}$ such that $u_i=u_j$ for some $j<i$.  Let $G(T)$ be the graph associated with $T$.  We will prove by induction on $\ell_T$ that $G(T)$ has a cycle decomposition.

If $\ell_T=0$, then $T$ does not traverse the same vertex twice, so it is itself either a cycle, or trivial.  In the former case, we can see that $\{T\}$ is a cycle decomposition of $G(T)$; in the latter, the empty set is a cycle decomposition of $G(T)$.

Now assume that for some $k\geq 0$, if $\ell_T\leq k$, then $G(T)$ has a cycle decomposition.

Suppose that $\ell_T=k+1$.  Then there is a vertex of $T$ that is traversed twice.  Let $i,j\in\{0,\dotso, m-1\}$ be indices with $i<j$ and $v_i=v_j$, such that $j-i$ is as small as possible.

Define $C=v_ie_{i+1}v_{i+1}\dotso e_jv_j$. If two vertices among $v_i,v_{i+1},\dotso , v_{j-1}$ are the same, then the minimality of $j-i$ is violated, contradicting our stipulation on $i$ and $j$.  Hence we may assume that $C$ is a cycle.

Let $T'=v_0e_1\dotso v_ie_{j+1}v_{j+1}\dotso v_{m-1}e_mv_0$ be a closed trail of $G(T)\setminus E(C)$ obtained by deleting the cycle $C$ from $T$.  Then $T'$ has at least one less repetition of vertices in its sequence than $T$ does, so $\ell_{T'}\leq k$.  By the induction hypothesis, we can decompose $G(T)\setminus E(C)$ into cycles $C_1,\dotso , C_r.$  Then $\{C,C_1,\dotso , C_r\}$ is a decomposition of $G(T)$ into cycles.  Since $T$ traverses every edge of $G$ exactly once, this gives a cycle decomposition of $G$.

By strong induction on $\ell_T$, we have shown that any Euler tour $T$ gives rise to a cycle decomposition of $G$.

$\Leftarrow$: Let $\mathcal{C}=\{C_1,\dotso , C_k\}$ be a cycle decomposition of $G$.  Since every edge of $G$ is traversed in exactly one cycle, and all cycles are closed trails, we have that either $\mathcal{C}$ is an Euler family of $G$; or else two of the cycles have a vertex in common.

{\sc Case 1: $\mathcal{C}$ is an Euler family of $G$.}  If $G$ is loopless, then Lemma~\ref{lem:eulerfamilyingraphs} implies that $G$ is eulerian, which completes the proof.  Hence we assume $G$ has some loop $C_i\in\mathcal{C}$.  If $k=1$, then $C_i=C_1$ is an Euler tour of $G$.  Hence we may assume $k\geq 2$.  Since $G$ is connected, we cannot have that $C_i$ traverses an isolated vertex, so there exists a different cycle $C_j\in\mathcal{C}$ that traverses the same vertex that $C_i$ does.  But then $\mathcal{C}$ is not vertex-disjoint, so it cannot be an Euler family, contradicting our assumption.

{\sc Case 2: there exist $C_i, C_j\in\mathcal{C}$, with $i\neq j$, such that $C_i$ and $C_j$ traverse a common vertex.} We can concatenate $C_i$ and $C_j$ as trails and obtain a smaller collection of closed trails.  Repeatedly applying this argument to the new collection of closed trails will eventually yield an Euler family of $G$, which demonstrates that $G$ is eulerian.
\end{proof}

Unfortunately, a cycle decomposition in a hypergraph is not enough to guarantee that the hypergraph is eulerian, only quasi-eulerian --- we will explore more on this subject later.  For graphs, Veblen's result is quite useful, and we will use it to easily prove the most well-known characterization of eulerian graphs: the one conjectured by Euler himself.  A modern statement of the theorem is presented in a tome by Fleischner \cite{F1}, and we present our proof that leans on Veblen's Theorem.\\

\begin{thm}\label{thm:eulertour}{\rm \cite[Theorem IV.1]{F1}} Let $G$ be a connected graph.  Then $G$ has an Euler tour if and only if $G$ is even.
\end{thm}

\begin{proof} $\Rightarrow:$ Let $T=v_0e_0v_1\dotso e_{i-1}v_ie_i\dotso v_{m-1}e_{m-1}v_0$ be an Euler tour of $G$.  Then, for each $i\in\mathds{Z}_m$, note that $e_{i-1}$ and $e_i$ each contribute 1 to the degree of $v_i$; this total contribution is even.  Since every edge of $G$ appears exactly once in $T$ and the degree of each of its ends is counted in this fashion, every vertex of $G$ must have even degree.

$\Leftarrow:$ We will first prove that every even graph of size $m$ has a cycle decomposition by strong induction on $m$.

Let $G'$ be an even graph of size $m$.  If $m=0$, then $G'$ has an empty cycle decomposition.

Now assume that there exists some $k\geq 0$ such that if $m\leq k$, then $G'$ has a cycle decomposition.

Suppose $m=k+1$.  Fix some $u\in V(G')$ that is not isolated, and let $T$ be a longest trail in $G'$ whose initial vertex is $u$.

Suppose the terminal vertex of $T$ is some $v\neq u$.  Then the number of edges incident with $v$ traversed by $T$ must be odd: one edge that is the last edge traversed by $T$, and two edges for each time $T$ visits $v$ earlier in the sequence.

Since $v$ has even degree in $G'$, it has another incident edge that is not traversed by $T$, so we can form a longer trail by appending this edge to $T$, contradicting the assumption that $T$ is as long as possible.  Hence $T$ is closed.

If $T$ traverses every edge of $G'$, then $T$ is an Euler tour of $G'$, and so $G'$ has just one non-trivial connected component.  This connected component has a cycle decomposition by Theorem~\ref{thm:veblen}, which serves as a cycle decomposition for $G'$ since they have the same edge set.  

Otherwise, let $G(T)$ denote the graph associated with $T$.  We know that $T$ has length no greater than $k$, so $G(T)$ and $G'\setminus E(G(T))$ are both graphs of size at most $k$.  Since we have already proven that all connected eulerian graphs are even (in the forward direction of this proof), and $G(T)$ is connected and eulerian, we conclude that $G(T)$ is even.  Then, since $G'$ is also even, we have that $G'\setminus E(G(T))$ is even as well.  Hence $G(T)$ and $G'\setminus E(G(T))$ each satisfy the conditions of the induction hypothesis.  Applying the induction hypothesis yields cycle decompositions $\mathcal{C}_1$ of $G(T)$ and $\mathcal{C}_2$ of $G'\setminus E(G(T))$.

Since $G(T)$ and $G'\setminus E(G(T))$ are edge-disjoint, so are their cycle decompositions.  Hence $\mathcal{C}_1\cup\mathcal{C}_2$ is a cycle decomposition of $G'$.

In both cases, we have concluded that $G'$ has a cycle decomposition, so by induction, every even graph has a cycle decomposition.  In particular, we have that $G$ has a cycle decomposition.  By Veblen's Theorem~\ref{thm:veblen}, since $G$ is connected and has a cycle decomposition, we conclude that $G$ is eulerian.
\end{proof}

Interestingly, Fleischner defines ``eulerian'' for a connected graph to mean ``even,'' which is of course, by this theorem, equivalent to admitting an Euler tour.  

There are also a number of parity results on eulerian graphs that can be found in \cite{F3}, of which we present just one.\\

\begin{thm}\label{thm:oddnumberofcycles}{\rm \cite{M}} Let $G$ be a graph.  Then every vertex of $G$ has even degree if and only if every edge of $G$ is contained in an odd number of cycles (as subgraphs).\qed
\end{thm}

\begin{remark}{\rm Note that when we count cycles, we are treating all sequences that are reversals or cyclic rotations of each other as the same cycle; that is, we are counting subgraphs isomorphic to a cycle.

The proof of Theorem~\ref{thm:oddnumberofcycles} relies heavily on the fact that an eulerian graph is even; that is, all its vertices have even degree.  This result does not extend in any meaningful way to hypergraphs, as we will show with a couple of counterexamples.

Consider the hypergraph $H_1=(V_1,E_1)$ with $V_1=\{1,2,3,4\}$ and $E_1=\{a,b\}$.  Define $a=b=1234.$  The incidence graph of $H_1$ is shown in Figure \ref{fig:oddcycles1} below.  Note that $H_1$ is eulerian with an Euler tour $1a2b1$ (among others), which corresponds to a closed trail in the incidence graph, per Remark~\ref{remark:traversals}.  However, each edge is in $\binom{4}{2}=6$ cycles, which is not an odd number.

\vspace{15px}
\begin{figure}[h]
\centerline{\includegraphics[scale=0.6]{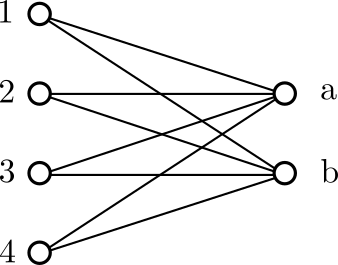}}
\caption{Incidence graph of $H_1$. \label{fig:oddcycles1}}
\end{figure}

However, perhaps our reckoning for the number of cycles needs to be different in hypergraphs.  After all, all the cycles above were, in some sense, the same: they merely traversed different anchor vertices but the edges traversed were the same.  There are a number of ways we could attempt to count cycles, so perhaps we were too na\"{i}ve with the first attempt.  We could regard two cycles as different only if the cyclic sequences of edges they traverse are different; or perhaps we could construct a hypergraph associated with a cycle that includes all the edges traversed, along with all the anchors and floaters, and count these hypergraphs instead.  Maybe then, Theorem~\ref{thm:oddnumberofcycles} will extend nicely!

Unfortunately, this is not the case either: we can dash both of these ideas with one counterexample.  Let $H_2=(V_2,E_2)$ with $V_2=\{1,2,3,4,5\}$ and $E_2=\{a,b,c,d\}.$  Define $a=123, b=234, c=345,$ and $d=15$.  The incidence graph of $H_2$ is shown in Figure \ref{fig:oddcycles2} below.  It can be observed that any cycle traversing $d$ must also traverse $a$ and $c$, and so we can count five cycles whose edge traversals are different.  They are $2a3b2, 3b4c3, 2a3c4b2, 1a3c5d1,$ and $1a2b3c5d1$, the last of which is itself an Euler tour.  Then every edge is in fact in an even number of these cycles, yet $H$ is eulerian.  Furthermore, every hypergraph that we could associate with these cycles will be different, so once again, each edge of $H$ is in an even number of those, too.

\vspace{15px}
\begin{figure}[h]
\centerline{\includegraphics[scale=0.6]{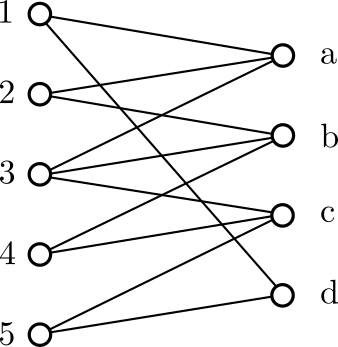}}
\caption{Incidence graph of $H_2$. \label{fig:oddcycles2}}
\end{figure}

We are left to conclude that Theorem~\ref{thm:oddnumberofcycles} does not extend to hypergraphs in any obvious way.
}
\end{remark}





\section{Euler Tours in Hypergraphs}

\subsection{Strongly Connected Hypergraphs}

One of the most advanced (and, in our estimation, the first) papers on eulerian hypergraphs is due to Lonc and Naroski \cite{LN}.  They prove a necessary and sufficient condition for a hypergraph to be eulerian on a special class of hypergraphs, defined below.\\

\begin{defn}{\rm \cite{LN} Let $H$ be a $k$-uniform hypergraph without isolated vertices.  If the level-($k-1$) line graph $\mathcal{L}_{k-1}(H)$ is connected, then we call $H$ {\em strongly connected}. }
\end{defn}

Lonc and Naroski prove the following theorem about strongly connected hypergraphs:\\

\begin{thm}\label{thm:LN2}{\rm \cite[Theorem 2]{LN}} Let $H$ be a strongly connected $k$-uniform hypergraph and let $V_{\text{odd}}(H)$ be the set of vertices of odd degree in $H$.  Then $H$ is eulerian if and only if $|V_{odd}(H)|\leq (k-2)|E(H)|. \qed$
\end{thm}

It is perhaps not difficult to see, once we apply the upcoming Theorem~\ref{thm:incidencegraph}, that for any $k$-uniform hypergraph $H$, the inequality $|V_{odd}(H)|\leq (k-2)|E(H)|$ is necessary for $H$ to be eulerian.  We can regard the construction of an Euler tour as a deletion of some edges of the incidence graph: for each e-vertex, we must choose $k-2$ incident edges to delete, and when we are done, we need all the v-vertices to have even degree.  If there are more odd-degree v-vertices from the outset than there are edges to be deleted, then it will be impossible to end up with all v-vertices of even degree.

Furthermore, if $k=2$, then $H$ is a connected graph, and then Theorem~\ref{thm:LN2} states, as usual, that being eulerian is equivalent to being even.

Lonc and Naroski, in the process of proving Theorem~\ref{thm:LN2}, also make use of the following general numerical result.  We will have tools later in this chapter that help us prove this more simply, but for now we must make do without.\\

\begin{prop}{\rm \cite[Proposition 1]{LN}} Let $H$ be an eulerian hypergraph.  Then $$\sum_{v\in V(H)}\left\lfloor\frac{\deg(v)}{2}\right\rfloor \geq |E(H)|.$$
\end{prop}

\begin{proof} Let $G$ be the incidence graph of $H$, let $T$ be an Euler tour of $H$, and let $G_T$ be the subgraph of $G$ corresponding to $T$ as described at the end of Remark~\ref{remark:traversals}.  Note, since $T$ is an Euler tour of $H$, that $G_T$ contains all the e-vertices of $G$ and is itself eulerian (the latter observed in Remark~\ref{remark:closedtrails}).  Since $G_T$ is eulerian, it is even by Theorem~\ref{thm:eulertour}.  Observe that each e-vertex $e$ of $G_T$ has exactly two neighbours: namely, the two anchors via which $e$ is traversed in $T$.

Now, there exists a hypergraph $H'$ that has $G_T$ as its incidence graph.  Now, since all the edges of $H'$ have cardinality 2 and $H'$ has no loops, it can be considered to be a graph.  Since $G_T$ is even, all the vertices of $H$ have even degree as well.

Then, by the Handshaking Lemma~\ref{lem:handshaking}, we have $$\sum_{v\in V(H')}\left\lfloor\frac{\deg_{H'}(v)}{2}\right\rfloor = \sum_{v\in V(H')}\frac{\deg_{H'}(v)}{2} = |E(H')|.$$

Now, we have $|E(H')|=|E(H)|$ because these quantities are both equal to the number of e-vertices of $G_T$.  Furthermore, for each $v\in V(H)$, we have that either $v'\not\in V(H')$ or $\deg_H(v)\geq\deg_{H'}(v)$.  Therefore, we have $$\sum_{v\in V(H)}\left\lfloor\frac{\deg_H(v)}{2}\right\rfloor \geq \sum_{v\in V(H')}\left\lfloor\frac{\deg_{H'}(v)}{2}\right\rfloor = |E(H)|.$$

\end{proof}

To prove Theorem~\ref{thm:LN2}, Lonc and Naroski first proved the easier case, when $k>3$.  (The case $k=2$ is self-evident, as our earlier observations showed.)\\

\begin{thm}{\rm \cite[Theorem 3]{LN}} Let $k>3$.  A strongly connected $k$-uniform hypergraph $H$ has an Euler tour if and only if $|E(H)|\neq 1$. \qed
\end{thm}

The remaining case, for $k=3$, was significantly more difficult and required a few successive results to pin down.  Some of these results refer to a class of hypergraphs which, roughly speaking, have a minimum number of edges among strongly connected hypergraphs of the same order.\\

\begin{def2}{\rm \cite{LN} Define $\mathcal{H}$ to be the smallest class of 3-uniform hypergraphs such that the following both hold:

\begin{description}
\item[(1)] The 3-uniform hypergraph of order 3 with a single edge belongs to $\mathcal{H}$.
\item[(2)] Let $H=(V,E)$.  If $H\in\mathcal{H}$, then the hypergraph $(V\cup \{v\}, E\cup \{e\})$ belongs to $\mathcal{H}$, for all $v,e$ such that $v\not\in V, v\in e, |e|=3,$ and $e\setminus\{v\}\subseteq e'$ for some $e'\in E$.\qed
\end{description}
}
\end{def2}

\begin{thm}{\rm \cite[Theorem 4]{LN}} Let $H$ be a strongly connected 3-uniform hypergraph.  The following are equivalent:

\begin{description}
\item[(1)] $H$ is eulerian;
\item[(2)] $|V_{\text{odd}}(H)|\leq |E(H)|$;
\item[(3)] $H\not\in \mathcal{H}$ or $H$ has a vertex of even degree. \qed
\end{description}
\end{thm}

For the remaining results, we need to define the {\em skeleton} of a hypergraph and a $wheel$ hypergraph.\\

\begin{defn}{\rm \cite{LN} Let $H=(V,E)$ be a 3-uniform hypergraph.  The {\em skeleton} of $H$, denoted $S(H)$, is a graph induced by the edge set $\{e\cap f: e,f\in E, |e\cap f|=2\}$. \phantom{aaa}
}
\end{defn}

\begin{defn}{\rm \cite{LN} Let $H$ be a 3-uniform hypergraph.  If $\mathcal{L}_{k-1}(H)$ is a cycle of length $\ell$, for $\ell\geq 3$, and there is a vertex of $H$ that belongs to all the edges of $H$, then $H$ is called a wheel, denoted $W_{\ell}$.
}
\end{defn}

\begin{thm}{\rm \cite[Theorem 6]{LN}} Let $H$ be a 3-uniform hypergraph whose skeleton is connected.  Then $H$ is eulerian if both of the following hold:

\begin{description}
\item[(1)] $H\not\in\mathcal{H}$ or $H$ has a vertex of even degree; and
\item[(2)] $H$ is not a wheel. \qed
\end{description}
\end{thm}

\subsection{Euler Tours in Design Hypergraphs}
In Chapters~\ref{chapter:STS}, \ref{chapter:covering}, and~\ref{chapter:quasi-eulerian}, we will be investigating so-called ``design hypergraphs,'' that is, hypergraphs that are uniform, regular, and satisfy some balancing property; triple systems are an example of such a hypergraph.  Such hypergraphs are implicitly studied in design theory, and there are some results about Euler tours in these designs.  Dewar and Stevens proved in 2012 \cite{DS} that certain kinds of triple systems are eulerian, while Horan and Hurlbert proved in 2013 \cite{HH,HH2} the existence of certain eulerian designs of every possible order.  We will discuss all of these results in this section, particularly as they apply to our new results.\\

\begin{defn}{\rm Let $H_1$ and $H_2$ be hypergraphs with incidence functions $\psi_1$ and $\psi_2$, respectively.  We call $H_1$ and $H_2$ {\em isomorphic} if there exist bijections $\alpha: V(H_1)\rightarrow V(H_2)$ and $\beta: E(H_1)\rightarrow E(H_2)$ such that, for all $e\in E(H_1)$, we have $\psi_1(e) = \{v_1,v_2,\dotso , v_k\}$ if and only if $\psi_2(\beta(e)) = \{\alpha(v_1),\alpha(v_2),\dotso , \alpha(v_k)\}$.  The pair $(\alpha,\beta)$ is called an {\em isomorphism} from $H_1$ to $H_2$.

In the event that $H_1$ and $H_2$ are simple, then $\beta$ can be induced from a bijection $\alpha: V(H_1)\rightarrow V(H_2)$ that maps edges of $H_1$ to edges of $H_2$.

An {\em automorphism} is an isomorphism from a hypergraph to itself, and the collection of all automorphisms admitted by a hypergraph forms a group under composition, called the {\em automorphism group.}
}
\end{defn}

\begin{defn}{\rm \cite{DS} Let $H$ be a hypergraph of order $n\geq 2$.  We call $H$ {\em cyclic} if its automorphism group contains a cyclic subgroup of order $n$.  
}
\end{defn}

\begin{thm}{\rm \cite[Corollary 5.10]{DS}}\label{thm:ds1} Let $H$ be a cyclic Steiner triple system of order $n>3.$  Then $H$ is eulerian.\qed
\end{thm}

\begin{thm}{\rm \cite[Corollary 5.11]{DS}}\label{thm:ds2} Let $H$ be a cyclic twofold triple system.  Then $H$ is eulerian.\qed
\end{thm}

The theorems above were originally stated in terms of {\em rank-2 universal cycles}, for which the definition follows.  For the purposes of our discussion, a {\em rank function} is an integer-valued function on a set of combinatorial objects, and the {\em rank} of such an object is the value such a rank function assigns to that object.\\

\begin{defn}{\rm \cite{DS} Let $V$ be a ground set, and $\F_n=\{F_0,\dotso , F_{m-1}\}$ be a set of combinatorial objects of rank $n$.

For each $F\in\F_n$, let $\mathcal{R}(F)$ be a set of sequences of length $n$ with elements from $V$; these are the sequences representing the combinatorial object $F$.

A cyclic sequence $\mathcal{U}=(x_0,x_1,\dotso , x_{m-1})$ is called a {\em universal cycle} of rank $n$ for $\F_n$ if there exists a function $$f:\mathds{Z}_m\rightarrow\bigcup_{F\in\F_n}\mathcal{R}(F)$$ such that $f(i)=(x_i,x_{i+1},\dotso , x_{i+n-1})\in\mathcal{R}(F_i)$, for all $i\in\mathds{Z}_m$.
}
\end{defn}

\begin{remark}{\rm Let $H=(V,E)$.  We will see how a rank-2 universal cycle of $H$ is equivalent to an Euler tour of $H$.  The ``combinatorial objects'' that we are considering are the edges of $H$, so $\F_n=E$, and the ``ground set'' is the set of vertices $V$.  Then we designate the set of rank-2 representatives of each $e\in E$ by $\mathcal{R}(e)=\{(u,v):u,v\in e, u\neq v\}$.

Suppose $T=v_0e_1v_1e_1\dotso v_{m-1}e_mv_0$ is an Euler tour of $H$.  Then $\mathcal{U}=(v_0,v_1,\dotso , v_{m-1})$ is a universal cycle corresponding to $T$, accompanied by the function $f(e_i)=(v_{i-1},v_i)$ for all $i\in\mathds{Z}_m$.

Conversely, if $\mathcal{U}=(v_0,v_1,\dotso , v_{m-1})$ is a universal cycle for $E$ (with $\mathcal{R}(e)$ as defined above), then there is an ordering $e_1,\dotso , e_m$ of the edges of $H$ such that for all $i\in\mathds{Z}_m$, we have $v_{i-1}v_i\in e_i$.  Since $v_{i-1}\neq v_i$ for all $i$ by the definition of $\mathcal{R}(e)$, we see that $v_0e_1v_1\dotso v_{m-1}e_mv_0$ is an Euler tour of $H$.
}
\end{remark}

In later chapters, we will extend Dewar and Stevens's results by demonstrating that all triple systems are eulerian.

Horan and Hurlbert, in slight contrast, stated their results in terms of 1-overlap cycles, which we define as follows.  Let $H$ be an $r$-uniform\footnotemark~hypergraph.  A cyclic sequence $\mathcal{U}=(v_0,v_1,\dotso , v_{m-1})$ of vertices of $H$ is called a {\em 1-overlap cycle} for $H$ if each subsequence $(v_{i(r-1)},v_{i(r-1)+1},\dotso , v_{(i+1)(r-1)})$ is an ordering of the vertices of an edge of $H$, and the list of all such subsequences for $i=0,\dotso , \frac{m}{r-1} - 1$ contains exactly one representative of each edge of $H$.  (Note that the first entry of each subsequence ``overlaps'' the last entry of the previous.) Horan and Hurlbert immediately remarked in \cite{HH2} that a 1-overlap cycle is equivalent to a rank-2 universal cycle.
 
\footnotetext{We note that Horan and Hurlbert define 1-overlap cycles more generally, not just for uniform hypergraphs; however, we present only a simple definition.}

In their work, Horan and Hurlbert gave explicit constructions of Euler tours in an infinite family of both Steiner triple systems \cite{HH} and Steiner quadruple systems \cite{HH2}, the latter of which is defined as follows.\\

\begin{defn}{\rm
A {\em Steiner quadruple system} is a 4-uniform hypergraph $H$ in which every triple of vertices lie together in exactly one edge.
}
\end{defn}

Their main results are the following.\\

\begin{thm}\label{thm:hh1}{\rm \cite[Corollary 23]{HH}} Let $n\equiv 1$ or 3 (mod 6) with $n\geq 7$.  Then there exists an eulerian STS($n$).\qed
\end{thm}

\begin{thm}\label{thm:hh2}{\rm \cite[Theorem 1.2]{HH2}} Let $n\equiv 2$ or 4 (mod 6) with $n > 4$.  Then there exists an eulerian Steiner quadruple system of order $n$.\qed
\end{thm}

Since a Steiner triple system of order $n$ exists if and only if $n\equiv 1$ or 3 (mod 6), Theorem~\ref{thm:hh1} covers every possible order, excepting the degenerate case $n=3$.  Likewise, a Steiner quadruple system of order $n$ exists if and only if $n\equiv 2$ or 4 (mod 6), so their result on quadruple systems is comprehensive as well.  Once again, our results will complete Horan and Hurlbert's above results by proving that all such designs are eulerian.

\subsection{Complexity of EULER TOUR}

There are a few results about the complexity of recognizing an eulerian hypergraph.  We begin by defining the EULER TOUR problem.\\

\begin{problem}EULER TOUR

{\rm \textsc{Given:}} A hypergraph $H$.

{\rm \textsc{Decide:}} Does $H$ have an Euler tour?
\end{problem}

If we consider only the set of 2-uniform hypergraphs ({\em i.e.} graphs), then this problem reduces to checking that the degree of every vertex is even, hence it is in the class P.  However, when we look at $k$-uniform hypergraphs for $k\geq 3$, this is no longer the case.  The proof \cite{LN} makes use of a known NP-complete problem that we adapt for our purpose.\\

\begin{problem}HAMILTON CYCLE

{\rm \textsc{Given:}} A graph $G$.

{\rm \textsc{Decide:}} Does $G$ have a Hamilton cycle?
\end{problem}

Famously, the HAMILTON CYCLE problem is NP-complete over several different classes of graphs.  Of interest to us is the following result by Garey, Johnson, and Tarjan from 1976:

\begin{thm}{\rm \cite{GJT}}\label{thm:Hamcycle}
The HAMILTON CYCLE problem is NP-complete over the set of planar, 3-regular, 3-connected graphs.
\end{thm}

We are not very interested in what it means for a graph to be planar, so we will relax that qualification.  A {3-connected} graph is a connected graph of order at least 3 that is connected after deletion of any set of up to two vertices.  Observe that a 3-regular, 3-connected graph must also be simple, so we can reframe Theorem~\ref{thm:Hamcycle} as follows.\\

\begin{cor}\label{cor:Hamcycle2}
The HAMILTON CYCLE problem is NP-complete over the set of simple 3-regular graphs.
\end{cor}

Now we will look at and prove some results about the complexity of EULER TOUR in hypergraphs.\\

\begin{thm}\label{thm:complexity}\hfill
\begin{description}
\item[(1)] {\rm \cite[Theorem 2.29]{BS2}} EULER TOUR is NP-complete on the set of 3-uniform, 2-regular, linear hypergraphs.
\item[(2)] {\rm \cite[Theorem 1]{LN}} Let $k\geq 3$.  Then EULER TOUR is NP-complete on the set of $k$-uniform hypergraphs.
\end{description}
\end{thm}

\begin{proof}
The proofs of these theorems are similar, as we will now show.

First of all, verification that a given sequence is an Euler tour is a matter of determining that it is a closed walk of the appropriate length that does not repeat any edges, which is polynomial in the order and size of the hypergraph.  Therefore, EULER TOUR is in NP.

To show that EULER TOUR is also NP-hard, we reduce the HAMILTON CYCLE problem to EULER TOUR.  Let $G=(V(G),E(G))$ be a simple 3-regular graph.  Since $G$ is loopless, we may regard it as a 2-uniform hypergraph.  Let $H=(V,E)$ be the dual hypergraph of $G$ as given in Definition~\ref{def:dual}.  Note that $H$ is 3-uniform because $G$ is 3-regular; 2-regular because $G$ is 2-uniform; and no two edges of $H$ share more than one common vertex because $G$ has no parallel edges (and, therefore, no two vertices of $G$ are the ends of two distinct edges).

To complete the proof of (1), we show that $G$ has a Hamilton cycle if and only if $H$ has an Euler tour.  First, suppose $C=v_0e_1v_1e_2\dotso v_{n-1}e_nv_0$ is a Hamilton cycle of $G$.  Then $T=e_1e_{v_1}e_2e_{v_2}\dotso e_{n-1}e_{v_{n-1}}e_ne_{v_0}e_1$ is a closed walk of $H$.  It traverses every edge $e_{v_0},\dotso e_{v_{n-1}}$ of $H$ exactly once because $C$ traverses every vertex $v_0,\dotso , v_{n-1}$ of $G$ exactly once.  Hence $T$ is an Euler tour of $H$.

Conversely, suppose $T=e_1e_{v_1}e_2e_{v_2}\dotso e_{n-1}e_{v_{n-1}}e_ne_{v_0}e_1$ is an Euler tour of $H$.  We claim that $C=v_0e_1v_1e_2\dotso v_{n-1}e_nv_0$ is a Hamilton cycle of $G$. It is clear from the duality of $G$ and $H$ that $C$ is a closed walk of $G$ that traverses every vertex exactly once.  It necessarily traverses no edge more than once because that would imply that it traverses some vertex more than once.  Hence $C$ is indeed a Hamilton cycle of $G$.

Since HAMILTON CYCLE is NP-complete on the set of simple 3-regular graphs, this shows that EULER TOUR is NP-hard (and, therefore, NP-complete) on the set of 3-uniform, 2-regular, linear hypergraphs.

To complete the proof of (2), we must modify the dual hypergraph $H$.  Of course, if $k=3$, then we can complete the proof as we did for (1), but if $k\geq 4$ then we will need to increase the cardinality of the edges somehow.  We accomplish this by introducing ``dummy'' vertices that will fill out the cardinalities of the edges, but which only have degree 1, so they cannot be traversed by an Euler tour.

Fix $k\geq 3$.  For each $v\in V(G)$, define a set of new vertices $\{u_v^1,\dotso , u_v^{k-3}\}$, such that these sets are pairwise disjoint and disjoint from $V(H)$.  Define a new hypergraph $H'$ from $H$ using these extra vertices in the following way:
\begin{itemize}
\item $V(H') = V(H)\cup\big(\bigcup_{v\in V(G)}\{u_v^1,\dotso , u_v^{k-3}\}\big)$;
\item $E(H') = \{e'_v: v\in V(G)\}$;
\item $\psi(e'_v) = e_v\cup \{u_v^1,\dotso , u_v^{k-3}\}$, for each $e'_v\in E(H')$.
\end{itemize}

Now, observe that $H'$ is $k$-uniform and its vertices have degree 2 if they come from $V(H)$, and degree 1 otherwise.  None of the degree 1 vertices can be in an Euler tour, so an Euler tour of $H'$ truly does correspond to a Hamilton cycle of $G$ since it traverses only vertices and edges of $H'$ that correspond to edges and vertices of $G$.  Much as in the proof of (1), we can also see that a Hamilton cycle of $G$ corresponds to an Euler tour of $H'$.  Hence EULER TOUR is NP-hard (and, therefore, NP-complete) on the set of $k$-regular hypergraphs.  (Note that we have also shown that EULER TOUR is NP-complete on the set of $k$-regular {\em linear} hypergraphs.)
\end{proof}

Lonc and Naroski also proved a complexity result about a more specific class of 3-uniform hypergraphs.  Though it is more technical and of less interest to us, we present it here for completeness.\\

\begin{thm}\label{thm:complexity2}{\rm \cite[Theorem 7]{LN}} The EULER TOUR problem on the set of 3-uniform hypergraphs whose skeleton is connected is NP-complete.
\end{thm}

Lonc and Naroski prove Theorem~\ref{thm:complexity2} by reducing the problem to the EULER TOUR problem on the set of connected 3-uniform hypergraphs, which we have seen is NP-complete.

\subsection{Conditions for a Hypergraph to be (Quasi)-eulerian}
The paper by Bahmanian and \v{S}ajna \cite{BS2} was meant to be the definitive paper on eulerian hypergraphs.  A lot of focus was also given to Euler families, as these are weaker versions of Euler tours, and results about eulerian and quasi-eulerian hypergraphs are closely related.  They presented several necessary and sufficient conditions for a hypergraph to be eulerian or quasi-eulerian.  We begin with the block characterization, for which additional terminology is required.\\



\begin{defn}{\rm \cite{BS2} Let $H$ be a hypergraph with no empty edges.  A vertex $v\in V$ is a {\em separating vertex} for $H$ if $H$ decomposes into two non-empty connected hypersubgraphs with just $v$ in common.  A {\em block} of $H$ is a connected strong subhypergraph of $H$ with no separating vertices, maximal with respect to this property.
}
\end{defn}

\begin{remark}{\rm A separating vertex of a hypergraph $H$ is either a cut vertex of $H$ or a vertex in a singleton edge \cite{BS1}.   However, it is self-evident that if $H$ has a singleton edge, then it cannot be eulerian or quasi-eulerian, as such an edge cannot be part of a walk.  Therefore, a block can usually be regarded as a maximally connected strong subhypergraph with no cut vertices --- in particular, this interpretation does not change the result that follows.\\
}
\end{remark}

\begin{thm}\label{thm:blocks}{\rm \cite[Theorem 2.24]{BS2}} Let $H$ be a hypergraph.  Then

\begin{description}
\item[(1)] $H$ is quasi-eulerian if and only if each block of $H$ is quasi-eulerian.
\item[(2)] $H$ is eulerian if and only if each block $B$ of $H$ has an Euler tour that traverses every cut vertex of $H$ contained in $B$.\qed
\end{description}
\end{thm}

Theorem~\ref{thm:blocks} illustrates that we need to consider cut vertices (and similarly, cut edges) carefully when constructing Euler tours and families.  We will explore this concept more in Chapter~\ref{chapter:edgecuts}.

Now we turn our attention to a necessary condition that has an impact on other results on the hamiltonicity of line graphs.\\

\begin{thm}\label{thm:hamiltonianlinegraph}{\rm \cite[Theorem 2.35]{BS2}} Let $H$ be a hypergraph and $\mathcal{L}(H)$ be its line graph.  If $H$ is eulerian, then $\mathcal{L}(H)$ is hamiltonian.  \qed
\end{thm}

\begin{remark}{\rm
There has been much work done on the hamiltonicity of line graphs for certain designs (see, for example, \cite{AHM, HR}), and the matter has largely been settled for designs such as triple systems.  However, more recent work in the area is concerned with variations of the line graph rather than variations on these designs ({\em e.g.} \cite{PVW, EP}).

The reader must take caution, however: the converse of Theorem~\ref{thm:hamiltonianlinegraph} does not hold in general.  It is sufficient that the {\em dual} of a hypergraph $H$ be hamiltonian in order for $H$ to be eulerian; however, the line graph is isomorphic to the simple 2-section of the dual, and we know that information about $H$ is lost when examining the 2-section.  The next result gives conditions under which the converse does hold.\\
}
\end{remark}


\begin{thm}{\rm \cite[Theorem 2.17]{BS2}} Let $H$ be a hypergraph.  Let $L$ be a graph and assume one of the following holds:

\begin{description}
\item[(1)] $L=\mathcal{L}(H)$ and $\text{deg}_H(v)\leq 2$ for all $v\in V$; or
\item[(2)] $L=\mathcal{L}_2(H)$ and $L$ is bipartite; or
\item[(3)] $L=\mathcal{L}_3(H)$.
\end{description}

If $L$ is hamiltonian, then $H$ is eulerian.  If $L$ admits a 2-factor, then $H$ is quasi-eulerian.\qed
\end{thm}

We next present a weakened statement of a result that will nonetheless provide a few interesting corollaries.  The full version of the theorem will be presented, along with its proof, in Section~\ref{chapter:tools}.\\

\begin{thm}\label{thm:incidencegraph1}{\rm \cite[Theorem 2.18]{BS2}} Let $H=(V,E)$ be a hypergraph and $G$ be its incidence graph.  If there exists a spanning subgraph $G'$ of $G$ such that $\text{deg}_{G'}(e) = 2$ for all $e\in E(H)$ and $\text{deg}_{G'}(v)$ is even for all $v\in V(H),$ then $H$ is quasi-eulerian.  If $G'$ additionally has at most one non-trivial connected component, then $H$ is eulerian.
\end{thm}

We can almost immediately conclude the following three results by using Theorem~\ref{thm:incidencegraph1}.\\

\begin{cor}{\rm \cite[Corollary 2.19]{BS2}} Let $H$ be a hypergraph and let $G$ be its incidence graph.  If $G$ has a 2-factor, then $H$ is quasi-eulerian.  If $G$ is hamiltonian, then $H$ is eulerian.\qed
\end{cor}

\begin{cor}{\rm \cite[Corollary 2.20]{BS2}} Let $H$ be an $r$-regular, $r$-uniform hypergraph for $r\geq 2$.  Then $H$ is quasi-eulerian.\qed
\end{cor}

\begin{cor}{\rm \cite[Corollary 2.21]{BS2}} Let $H$ be a $2k$-uniform hypergraph with no vertices of odd degree.  Then $H$ is quasi-eulerian.  \qed
\end{cor}




Bahmanian and \v{S}ajna also settled the question of complexity for the problem of recognizing a quasi-eulerian hypergraph, as follows.\\

\begin{problem}\label{problem:eulerfamily}{\rm \cite{BS2}} {\rm EULER FAMILY}

{\rm \textsc{Given:}} A hypergraph $H$.

{\rm \textsc{Decide:}} Does $H$ have an Euler family?
\end{problem}

They show that EULER FAMILY on the class of all hypergraphs is in P by reducing it to the problem of finding a 1-factor in a graph.

We will present more of their results in Section~\ref{chapter:tools}, for they will be quite pertinent to the problems we wish to tackle.

\section{Spanning Euler Families and Vertex Cuts}
A recent paper by Steimle and \v{S}ajna \cite{SS} has characterized the necessary and sufficient conditions for a spanning Euler family (or tour) to exist by examining {\em vertex cuts} and some related subhypergraphs.  We first give a definition of vertex cuts, and explain what subhypergraphs will be used.\\

\begin{defn}{\rm {\bf (Vertex Cut)} \cite{SS}
Let $H=(V,E)$ be a hypergraph and let $S\subsetneq V$ be a proper subset of its vertices.  We call $S$ a {\em vertex cut} of $H$ if $H- S$ is disconnected.  If $S$ has the property that no proper subset of $S$ is a vertex cut of $H$, then we call $S$ a {\em minimal} vertex cut of $H$.
}
\end{defn}

Generally, we are most interested in vertex cuts of connected hypergraphs.  Note that the empty set is a vertex cut of a disconnected hypergraph, which is usually not very interesting.  However, if $H$ is connected, then we can usefully define the following subhypergraphs of $H$.\\

\begin{defn}{\rm {\bf (Derived Hypergraph)} \cite{SS}
Let $H$ be a hypergraph with no empty edges, let $S$ be a vertex cut of $H$, and let $H_i$, for $i\in I$, be the connected components of $H-S$.  Then the {\em $S$-component of $H$ corresponding to $H_i$} is the subhypergraph $H_i'$ with $V(H_i')=V(H_i)\cup S$ and $E(H_i')=\{e\in E(H): e\subseteq V(H_i')\}$.

If $|S|$ is even, then we additionally define the $S^*$- and $S^{**}$-components of $H$ corresponding to $H_i$ as the hypergraphs $H_i^*$ and $H_i^{**}$ obtained by adjoining $\frac{|S|}{2}$ or $|S|$ copies, respectively, of the edge $S$ to $H_i'$.

The hypergraphs $H_i', H_i^*,$ and $H_i^{**}$, for all $i\in I$, are collectively referred to as the {\em derived hypergraphs of $H$ (with respect to $S$).}
}
\end{defn}

The authors of \cite{SS} used the derived hypergraphs to give necessary and sufficient conditions for the existence of spanning Euler families and tours in two kinds of situations: first, if there exists a vertex cut of small cardinality (1 or 2); second, if there exists a vertex cut whose vertices all have degree 2.

These results all came in pairs: one for spanning Euler families and one for spanning Euler tours.  

First, the conditions when $H$ has a cut vertex:\\

\begin{thm}{\rm \cite[Theorem 3.11]{SS}} Let $H$ be a connected hypergraph with a cut vertex $v$, and let $H_i$, for $i\in I$, be the connected components of $H-v$.  Let $H_i'$, for $i\in I$, denote the $\{v\}$-components corresponding to $H_i$.  Then $H$ has a spanning Euler family if and only if the following hold.
\begin{description}
\item[(1)] For some $i\in I$, we have that $H_i'$ admits a spanning Euler family.
\item[(2)] For all $i\in I$, at least one of $H_i$ and $H_i'$ admits a spanning Euler family.\qed
\end{description}
\end{thm}

\begin{thm}{\rm \cite[Corollary 3.12]{SS}} Let $H$ be a connected hypergraph with a cut vertex $v$.  Then $H$ has a spanning Euler tour if and only if every $\{v\}$-component of $H$ admits a spanning Euler tour.\qed
\end{thm}

There are two possibilities when $H$ has a vertex cut of cardinality 2, given by Cases (1) and (2) in each of the following two results:\\

\begin{thm}{\rm \cite[Theorem 3.18, 3.20]{SS}}
Let $H$ be a hypergraph with a vertex cut $S=\{u,v\}$ of cardinality 2.  For each $i\in I$, let $H_i, H_i'$, and $H_i^*$ denote the connected components of $H-S$, the $S$-components of $H$, and the $S^*$-components of $H$, respectively.  Let $E_S=\{e\in E: e = S\}$.  Then $H$ has a spanning Euler family if and only if one of (1) and (2) hold.
\begin{description}
\item[(1)] $|E_S|$ is even and there exists $J\subseteq I$ of even cardinality such that the following hold.
	\begin{description}
	\item[(a)] For all $i\in J$, we have that $H_i^*$ admits a spanning Euler family.
	\item[(b)] For each $i\in I\setminus J$, at least one of $H_i, H_i',H_i'-u,$ and $H_i'-v$ admits a spanning Euler family.
	\item[(c)] If $J=\emptyset$, then at least one of the following holds:
		\begin{description}
		\item[(i)] There exists some $i\in I$ such that $H_i'$ admits a spanning Euler family; or
		\item[(ii)] There exist distinct $i,j\in I$ such that $H_i'-u$ and $H_j'-v$ admit spanning Euler families.
		\end{description}
	\end{description}
\item[(2)] $|E_S|$ is odd and there exists $J\subseteq I$ of odd cardinality such that the following hold.
	\begin{description}
	\item[(a)] For each $i\in J$, we have that $H_i'$ admits a spanning Euler family.
	\item[(b)] For each $i\in I\setminus J$, at least one of $H_i$ and $H_i^*$ has a spanning Euler family.\\\phantom{aaaaa}\qed
	\end{description}
\end{description}
\end{thm}

\begin{thm}{\rm \cite[Theorem 3.19, 3.21]{SS}}
Let $H$ be a hypergraph with a vertex cut $S=\{u,v\}$ of cardinality 2.  For each $i\in I$, let $H_i, H_i', H_i^*$, and $H_i^{**}$ denote the connected components of $H-S$, the $S$-components of $H$, the $S^*$-components of $H$, and the $S^{**}$-components of $H$, respectively.  Let $E_S=\{e\in E: e = S\}$.  Then $H$ has a spanning Euler tour if and only if one of (1) and (2) hold.
\begin{description}
\item[(1)] $|E_S|$ is even and there exists $J\subseteq I$ of even cardinality such that the following hold.
	\begin{description}
	\item[(a)] For all $i\in J$, we have that $H_i^*$ admits a spanning Euler tour.
	\item[(b)] For each $i\in I\setminus J$, at least one of $H_i'$ and $H_i^{**}$ admits a spanning Euler tour.
	\item[(c)] If $J=\emptyset$, then there exists $i\in I$ such that $H_i'$ admits a spanning Euler tour.
	\end{description}
\item[(2)] $|E_S|$ is odd and there exists $J\subseteq I$ of odd cardinality such that the following hold.
	\begin{description}
	\item[(a)] For each $i\in J$, we have that $H_i'$ admits a spanning Euler tour.
	\item[(b)] For each $i\in I\setminus J$, we have that $H_i^*$ has a spanning Euler tour.\qed
	\end{description}
\end{description}
\end{thm}

The situation is a bit simpler when $H$ has a minimal vertex cut $S$ whose vertices all have degree 2.  It can be shown that, in this case, we have just two connected components in $H-S$.\\

\begin{thm}{\rm \cite[Theorem 3.7]{SS}}
Let $H$ be a connected hypergraph with a minimal vertex cut $S$ such that $\deg(v)=2$ for all $v\in S$.  Then the following hold.
\begin{description}
\item[(1)] If $|S|$ is odd, then $H$ does not have a spanning Euler family.
\item[(2)] If $|S|$ is even, then $H$ has a spanning Euler family if and only if both $S^*$-components of $H$ have spanning Euler families.\qed
\end{description}
\end{thm}

Steimle and \v{S}ajna were only able to prove a necessary condition for spanning Euler tours, and showed that the converse does not hold.\\

\begin{thm}{\rm \cite[Corollary 3.8]{SS}}
Let $H$ be a connected hypergraph with a minimal vertex cut $S$ such that $\deg(v)=2$ for all $v\in S$.  If $H$ admits a spanning Euler tour, then $|S|$ is even and both $S^*$-components of $H$ admit a spanning Euler tour.\qed
\end{thm}

We will be proving analogous results to these, but for edge cuts, in Chapter~\ref{chapter:edgecuts}.  Although we will not be exclusively looking for {\em spanning} Euler families and tours, our conditions will look similar in appearance to the ones shown above.  Instead of investigating derived hypergraphs, we will look at some analogous hypergraphs associated to $H$ and the edge cut.

\section{Tools}\label{chapter:tools}
In this section, we will give results that have either previously been proven by other authors, or are general enough to be applied in multiple upcoming chapters.  Most of the previously existing results are due to Bahmanian and \v{S}ajna, and have to do with finding Euler tours or Euler families via the incidence graph.  Some of these results lead to very rudimentary results that are used to begin our investigation; we also include these here.  

We must first explain how the incidence graph can be used to find eulerian traversals by presenting the full version of Theorem~\ref{thm:incidencegraph1}.\\

\begin{thm}\label{thm:incidencegraph}{\rm \cite[Theorem 2.18]{BS2}} Let $H=(V,E)$ be a hypergraph and $G$ be its incidence graph.  Then the following hold:
\begin{description}
\item[(1)] $H$ is quasi-eulerian if and only if $G$ has a spanning subgraph $G'$ such that $\text{deg}_{G'}(e) = 2$ for all $e\in E(H)$ and $\text{deg}_{G'}(v)$ is even for all $v\in V(H)$.  
\item[(2)] $H$ is eulerian if and only if $G$ has a spanning subgraph $G'$ such that $\text{deg}_{G'}(e) = 2$ for all $e\in E(H)$ and $\text{deg}_{G'}(v)$ is even for all $v\in V(H)$, and $G'$ has at most one non-trivial connected component.
\end{description} 
\end{thm}

\begin{proof}
$\Rightarrow:$ Let $\F$ be an Euler family for $H$.  Fix some $T\in\F$.  Then $T$ has a corresponding trail $T'$ in $G$, per Remark~\ref{remark:traversals}, that traverses each e-vertex at most once.  Let $G_T=G(T')$ be the graph associated with $T'$.  Observe that $G_T$ is a connected subgraph of $G$ that has $T'$ as an Euler tour.  Then Theorem~\ref{thm:eulertour} indicates that $G_T$ is even.

Now, since the $G_T$, for $T\in\F$, are pairwise edge-disjoint, we get that $G'=\bigcup\limits_{T\in\F}G_T$ is an even subgraph of $G$.  Furthermore, since $\F$ is an Euler family of $H$, we have that the collection $\{T':T\in\F\}$ traverses every e-vertex of $G$ exactly once; hence each e-vertex $e$ has two neighbours in $G'$, namely the two v-vertices that correspond to the vertices via which $e$ is traversed in $\F$.  We conclude that $G'$ is an even subgraph of $G$ that contains each e-vertex of $G$, and in which each e-vertex has degree 2.  Then $\G=(V(G), E(G'))$ is a spanning subgraph of $G$ in which every v-vertex has even degree and every e-vertex has degree 2.

Now, if $H$ is eulerian, then we may assume $\F=\{T\}$.  Hence $G'=G_T$, which is connected.  Then $\G=(V(G), E(G'))$ has just one non-trivial connected component.

$\Leftarrow:$ Suppose that $G'$ is a spanning subgraph of $G$ with $\deg_{G'}(e)=2$ for all $e\in E$ and $\deg_{G'}(v)$ even for all $v\in V$.  Let $G_i$, for $i\in I$, be the non-trivial connected components of $G'$.  Then each $G_i$ admits an Euler tour $T_i'$.  Since $\deg_{G'}(e)=2$ for each e-vertex $e$, we have that $T_i'$ traverses each e-vertex at most once.  Then Remark~\ref{remark:traversals} says that $T_i'$ corresponds to a (strict) trail $T_i$ of $H$ such that $T_i$ traverses the vertices and edges of $H$ that correspond to the v-vertices and e-vertices, respectively, that are traversed by $T_i'$ in $G'$.

Then we claim $\F=\{T_i:i\in I\}$ is an Euler family of $H$: first of all, the collection of $T_i'$ is pairwise vertex- and edge-disjoint since each $T_i'$ is a trail of a distinct connected component of $G'$.  Secondly, each edge of $H$ is traversed exactly once by $\F$ because each e-vertex of $G$ is traversed exactly once by some $T_i'$, for $i\in I$.  Finally, since each $T_i'$ is an Euler tour of a non-trivial (and hence non-empty) connected component of $G'$, the corresponding trail $T_i$ is non-trivial as well.

Furthermore, if $G'$ has at most one non-trivial connected component, then we have $|\F|=1$, and so it yields an Euler tour of $H$.
\end{proof}

We will be making great use of Theorem~\ref{thm:incidencegraph}, especially in Chapters~\ref{chapter:STS}, \ref{chapter:covering}, and~\ref{chapter:quasi-eulerian}, so we would do well to equip ourselves with further terminology.\\

\begin{def2}{\rm Let $H$ be a hypergraph with incidence graph $G$.  If $\F$ is an Euler family of $H$, then we formally define the {\em subgraph of $G$ corresponding to $\F$}, usually denoted $G_{\F}$, as described in the proof of Theorem~\ref{thm:incidencegraph}:
\begin{itemize}
\item $G_{\F}$ is a spanning subgraph of $G$;
\item For v-vertex $v$ and e-vertex $e$ in $G_{\F}$, we have that $v$ and $e$ are adjacent if and only if $(v,e)$ is an anchor flag of a component of $\F$.\qed
\end{itemize}
}
\end{def2}

The justification for calling the trails of $\F$ ``components'' is given in the following result, which we treat as a corollary of Theorem~\ref{thm:incidencegraph}.\\

\begin{cor}\label{cor:correspondence} Let $H$ be a hypergraph with incidence graph $G$.  Let $\F$ be an Euler family of $H$ and $\G$ be the corresponding subgraph of $G$.  Let $G_T$, for each $T\in\F$, be the subgraph of $G$ corresponding to $T$.

Then $\{G_T: T\in\F\}$ is the set of non-trivial connected components of $\G$.
\end{cor}

\begin{proof} Since $\G$ is uniquely defined from $\F$ (although not vice-versa), we can construct $\G$ (and a graph $G'$) in the way described in the proof of the forward direction of Theorem~\ref{thm:incidencegraph}.  Note that $G'$ can be obtained from $\G$ by deleting isolated v-vertices.  Then $G'$ is a subgraph of $G$ containing all the edges of $\G$ and it is evident that $G'$ is a vertex-disjoint union of connected subgraphs $G_T$ of $G$.  Hence each $G_T$, for $T\in\F$, is a connected component of $G'$, and since each trail $T$ is non-trivial, so is its corresponding connected component $G_T$.  Then $\G$ has the same non-trivial connected components as $G'$, so $\{G_T: T\in\F\}$ is the set of non-trivial connected components of $\G$.
\end{proof}

We also have a basic new result about the nature of some of the connected components of $\G$ if $H$ is a covering 3-hypergraph.\\

\begin{lem}\label{lem:onecc} Let $H$ be a covering 3-hypergraph with incidence graph $G$.  Let $\F$ be an Euler family of $H$ and $\G$ be the corresponding subgraph of $G$.  Then $\G$ has at most one trivial connected component.
\end{lem}

\begin{proof}
Suppose $x$ and $y$ are distinct isolated vertices of $\G$.  Since the e-vertices of $\G$ have degree 2, we have that $x$ and $y$ must both be v-vertices.  Additionally, since $H$ is a covering 3-hypergraph, there exists some $e\in E(H)$ containing both $x$ and $y$, along with exactly one other vertex $z$.  Then the e-vertex $e\in V(\G)$ must be adjacent to two of $x,y,$ and $z$, contradicting the fact that neither $x$ nor $y$ has neighbours.

Therefore, there can be at most one trivial connected component in $\G$.
\end{proof}

Bahmanian and \v{S}ajna also began to investigate the eulerian properties of certain designs.  One important tool that these results are based on is due to Fleischner, as follows.\\

\begin{thm}\label{thm:fleischner}{\rm \cite{F3}} Let $G$ be a graph with no cut edges and no vertices of degree less than or equal to 2.  Then $G$ has a spanning even subgraph in which every vertex has degree at least 2.\qed
\end{thm}

Fleischner's Theorem can be used to show that a hypergraph is quasi-eulerian, as long as we can prove that the incidence graph satisfies the hypothesis of the theorem.  In fact, we can obtain an auxiliary graph from the incidence graph, apply Fleischner's Theorem to the auxiliary graph, and still conclude that the incidence graph gives rise to an Euler family.  We will see one such example in the following theorem.\\

\begin{thm}\label{thm:eulerfamily}{\rm \cite[Theorem 2.39]{BS2}} Let $H=(V,E)$ be a 3-uniform hypergraph without cut edges.  Then $H$ is quasi-eulerian.
\end{thm}

\begin{proof}
Let $G$ be the incidence graph of $H$.  We first show that $G$ has no cut edges.  Suppose that $ve\in E(G)$ is a cut edge of $G$, where $v\in V(H)$ and $e\in E(H)$.  Since $H$ is 3-uniform, we know that $e$ is not isolated in $G\setminus ve$, hence $G-e$ is a subgraph of $G\setminus ve$, and $G-e$ has at least as many connected components as $G\setminus ve$ has.   But now $G-e$ is the incidence graph of $H\setminus e$ \cite[Lemma 2.19]{BS1}, which is connected since $H$ has no cut edges.  This contradicts the fact that $G-e$ is disconnected.  Hence $G$ has no cut edges, as claimed.

Now, if $G$ has any vertices of degree 2 or less, they must be v-vertices, for all of the e-vertices of $G$ have degree 3.  If $G$ has a v-vertex $v$ of degree 1, then its incident edge is a cut edge of $G$, contradicting our earlier finding.  Hence there are no vertices of degree 1.

Now we may obtain an auxiliary graph $G'$ from $G$ by the following process:
\begin{description}
\item[(1)] Delete all isolated vertices.
\item[(2)] For each v-vertex $v$ of degree 2, replace the 2-path $e_1ve_2$ with a single edge $e_1e_2$, and then delete $v$.
\end{description}  
Then $G'$ is a graph that has no vertices of degree 2 or less, and --- as $G$ has --- no cut edges.

We may apply Theorem~\ref{thm:fleischner} to obtain a spanning even subgraph $G'_1$ of $G'$ with no isolated vertices.  Then we construct a graph $G_1$ by modifying $G'_1$ as follows:
\begin{description}
\item[(1)] Restore the isolated vertices of $G$ that were deleted previously.
\item[(2)] For every v-vertex $v\in V(G)$ that has degree 2, if its ends $e_1$ and $e_2$ are adjacent to each other in $G'_1$, then replace the edge $e_1e_2\in E(G'_1)$ with a 2-path $e_1ve_2$ in $G_1$; otherwise, add an isolated v-vertex $v$ in $G_1$.
\end{description}   
Then $G_1$ is a spanning even subgraph of $G$.  By Theorem~\ref{thm:incidencegraph}, this subgraph gives rise to an Euler family for $H$, so $H$ is quasi-eulerian.
\end{proof}

We have an immediate corollary of Theorem~\ref{thm:eulerfamily} that applies to covering hypergraphs.\\

\begin{cor}\label{cor:quasieuleriancovering}{\rm \cite[Corollary 2.40]{BS2}} Let $H$ be a covering 3-hypergraph with at least two edges.  Then $H$ is quasi-eulerian.
\end{cor}

\begin{proof} It suffices to show that $H$ has no cut edges.  Suppose, for the sake of obtaining a contradiction, that $e\in E(H)$ is a cut edge of $H$.  

If $|V(H)|=3$, then $e=V(H)$, and since $H$ has at least two edges, there is another edge parallel to $e$.  In this case, we obtain a contradiction because $H\setminus e$ is connected.  

Hence assume $|V(H)|\geq 4$, and let $w\in V(H)\setminus e$.  Suppose $u,v\in V(H)$ lie in different connected components of $H\setminus e$.  Then there is an edge of $H\setminus e$ containing both $u$ and $w$ and an edge of $H\setminus e$ containing both $v$ and $w$.  Hence $u$ and $v$ are, in fact, in the same connected component of $H\setminus e$, a contradiction.

Therefore, there are no cut edges of $H$.  We may apply Theorem~\ref{thm:eulerfamily} to see that $H$ is quasi-eulerian.
\end{proof}

We now present a useful new tool that has broad enough use to feature in Chapters~\ref{chapter:covering},~\ref{chapter:quasi-eulerian}, and~\ref{chapter:edgecuts}.  It demonstrates how to convert an Euler family from one hypergraph to another if the incidence graph of one of them is a subgraph (spanning on the e-vertices) of the incidence graph of the other.\\

\begin{lem}\label{lem:truncatedhypergraph}
Let $H_1$ and $H_2$ be hypergraphs.  Assume $V(H_1)\subseteq V(H_2)$ and let $\varphi:E(H_1)\rightarrow E(H_2)$ be any bijection satisfying $e\subseteq\varphi(e)$ for all $e\in E(H_1)$.  Then the following hold:
\begin{description}
\item[(1)] If $H_1$ has an Euler family $\F_1$, then $H_2$ has an Euler family $\F_2$ obtained from $\F_1$ by replacing each edge $e$ with $\varphi(e)$.
\item[(2)] If $H_2$ has an Euler family $\F_2$, and for all $e\in E(H_2)$, we have that $e$ is traversed in $\F_2$ via vertices in $\varphi^{-1}(e)$, then $H_1$ has an Euler family $\F_1$ obtained from $\F_2$ by replacing each edge $e$ with $\varphi^{-1}(e)$.
\end{description}
\end{lem}

\begin{proof}
{\bf (1)} Let $\F_1$ be an Euler family of $H_1$.  Obtain $\F_2$ from $\F_1$ by replacing every edge $e$ in a trail of $\F_1$ with $\varphi(e)$.  Then $\F_2$ is a collection of anchor-disjoint closed trails of $H_2$: they are trails because each edge is used at most once and is traversed via two vertices of $H_2$ that lie in that edge; they are anchor-disjoint and closed because the anchors are the same as in the trails of $\F_1$.  Hence $\F_2$ is an Euler family of $H_2$ with the claimed properties.

{\bf (2)} Let $\F_2$ be an Euler family of $H_2$ such that, for all $e\in E(H_2)$, we have that $e$ is traversed in $\F_2$ via vertices in $\varphi^{-1}(e)$.  Obtain $\F_1$ from $\F_2$ by replacing every edge $e$ in a trail of $\F_2$ with $\varphi^{-1}(e)$.  As in (1), we have that $\F_1$ is an Euler family of $H_1$, since our extra assumption ensures that every edge $\varphi^{-1}(e)$ is traversed in $\F_1$ only by anchors that it contains in $H_1$.
\end{proof}

A complete characterization of quasi-eulerian hypergraphs comes by way of a theorem of Lov\'{a}sz \cite{LL}.  There are a few ways we can use this theorem, and we will see how it can be directly applied in Chapter~\ref{chapter:quasi-eulerian}.  First, we must define a $(g,f)$-factor:\\

\begin{defn}{\rm \cite{AK} Let $G=(V,E)$ be a graph, and $f,g:V\rightarrow\mathds{N}$ be functions.  A {\em $(g,f)$-factor} of $G$ is a spanning subgraph $F$ of $G$ such that $g(v)\leq\deg_F(v)\leq f(v)$ for all $v\in V$.  An $(f,f)$-factor is also called an {\em $f$-factor}.
}
\end{defn}
And now, Lov\'{a}sz's theorem:\\

\begin{thm}\label{thm:lovasz2} {\em (The $(g,f)$-factor Theorem, Lov\'{a}sz \cite{LL,AK})} Let $G=(V,E)$ be a graph and $f,g:V\rightarrow\mathds{N}$ be functions such that $g(x)\leq f(x)$ and $g(x)\equiv f(x)$ (mod 2) for all $x\in V$.  Then $G$ has a $(g,f)$-factor $F$ such that $\deg_F(x)\equiv f(x)$ (mod 2) for all $x\in V$ if and only if, for all disjoint $S,T\subseteq V$, we have
\begin{equation}\label{eqn:lovasz}
\sum_{x\in S}f(x) + \sum_{x\in T}(\deg_G(x) - g(x)) - \varepsilon_G(S,T) - q(S,T)\geq 0,
\end{equation}
where $\varepsilon_G(S,T)$ denotes the number of edges with one end in $S$ and the other in $T$, and $q(S,T)$ is the number of connected components $C$ of $G-(S\cup T)$ such that $\sum\limits_{x\in V(C)}f(x) + \varepsilon_G(V(C),T)\equiv 1\text{ (mod 2).}\qed$
\end{thm}

Bahmanian and \v{S}ajna use Theorem~\ref{thm:lovasz2} to give a characterization of quasi-eulerian hypergraphs, by regarding an Euler family as a certain $(g,f)$-factor of the incidence graph.\\

\begin{thm}{\rm \cite[Corollary 2.35]{BS2}} Let $H=(V,E)$ be a hypergraph and $G$ be its incidence graph.  For $X,Y\subseteq V(G)$, let $\varepsilon_G(X,Y)$ denote the number of edges of $G$ with one end in $X$ and the other in $Y$.  Then $H$ is quasi-eulerian if and only if, for all disjoint sets $S\subseteq E$ and $T\subseteq V\cup E$, we have 
\begin{equation}\label{eqn:lovasz2}
2|S| + \sum_{x\in T}\deg_G(x) - 2|T\cap E| - \varepsilon_G(S,T\cap V) -q'(S,T)\geq 0,
\end{equation}
where $q'(S,T)$ is the number of connected components $C$ of $G-(S\cup T)$ such that $\varepsilon_G(V(C),T)$ is odd.
\end{thm}

\begin{proof} We begin by defining functions $g,f:V(G)\rightarrow\mathds{N}$ as follows:
\[  g(x)= \left\{
\begin{array}{lcl}
      0 & \text{if} & x\in V \\
      2 & \text{if} & x\in E\\
\end{array} 
\right. \hspace{15px}\text{ and }\hspace{15px} f(x)=\left\{
\begin{array}{lcl}
      2K & \text{if} & x\in V\\
      2 & \text{if} & x\in E\\
\end{array} 
\right.
\]
where $K$ is a sufficiently large integer.

Now, we can see that a $(g,f)$-factor $F$ of $G$, with $\deg_F(x)$ even for all $x\in V(F)$, corresponds to an Euler family for $H$ by Theorem~\ref{thm:incidencegraph}.

We will next show that, for any disjoint sets $S\subseteq E$ and $T\subseteq V\cup E$, we have that Inequalities~\ref{eqn:lovasz} and~\ref{eqn:lovasz2} are equivalent.

From the definitions of $g$ and $f$ we have 
\begin{align*}
&\sum_{x\in T}(\deg_G(x)-g(x))\\
=&\sum_{x\in T\cap V}(\deg_G(x) - 0) + \sum_{x\in T\cap E}\deg_G(x) - 2|T\cap E|\\
=&\sum_{x\in T}\deg_G(x) - 2|T\cap E|, \text{ and}
\end{align*}
$$\sum_{x\in S}f(x) = 2|S|.$$

Since $G$ is bipartite with $S$ and $T\cap E$ being subsets of the same part, we have $\varepsilon_G(S,T\cap E)=0$.  Hence $\varepsilon_G(S,T) = \varepsilon_G(S,T\cap V)$.

Finally, since $f(x)$ is even for all $x\in V(G)$, we observe that $q(S,T)=q'(S,T)$.

Therefore, Inequalities~\ref{eqn:lovasz} and~\ref{eqn:lovasz2} are equivalent in the case where $S\subseteq E$ and $T\subseteq V\cup E$.

Now, when $S\cap V\neq\emptyset$, it suffices to show that Inequality~\ref{eqn:lovasz} holds.  But since the left-hand side of the inequality contains a summand of $f(x)$ for some $x\in V$, and this quantity is arbitrarily large, it is evident that the inequality holds.

Therefore, if $H$ is quasi-eulerian, then there exists a $(g,f)$-factor $F$ of $G$ such that, for any $x\in V(G)$, we have that $\deg_F(x)$ is even if $x$ is a v-vertex, or $\deg_F(x)=2$ if $x$ is an e-vertex.  Then Theorem~\ref{thm:lovasz2} states that Inequality~\ref{eqn:lovasz} holds for all disjoint sets $S,T\subseteq V\cup E$, and so Inequality~\ref{eqn:lovasz2} holds for all disjoint sets $S\subseteq E$ and $T\subseteq V\cup E$.  Conversely, if Inequality~\ref{eqn:lovasz2} holds for all disjoint sets $S\subseteq E, T\subseteq V\cup E$, then Inequality~\ref{eqn:lovasz} holds for all disjoint sets $S,T\subseteq V\cup E$.  Then Theorem~\ref{thm:lovasz2} says that $G$ admits a $(g,f)$-factor, and this corresponds to an Euler family of $H$ by Theorem~\ref{thm:incidencegraph}. \phantom{aa}
\end{proof}

Finally, we present a condition of necessity that will be taken into consideration in Chapter~\ref{chapter:edgecuts}.\\

\begin{thm}{\rm \cite[Theorem 2.7]{BS2}}\label{thm:nocutedges} Let $H$ be a hypergraph.  If $H$ is eulerian, then $H$ has no non-trivial cut edges.  If $H$ is quasi-eulerian, then $H$ has no strong cut edges.
\end{thm}

\begin{proof} Let $H$ be a hypergraph.  We first assume that $H$ is eulerian.  If $H$ is edgeless, then it is evident that it has no non-trivial cut edges, so assume $H$ is not edgeless.  Let $e\in E(H)$, and without loss of generality, let $T=v_0ev_1e_2v_2\dotso v_{m-1}e_mv_0$ be an Euler tour for $H$.  Then $W=v_1e_2v_2\dotso v_{m-1}e_mv_0$ is an open Euler trail for $H\setminus e$.

Now, every edge of $H\setminus e$ is in the same connected component of $H\setminus e$, so $H\setminus e$ has no more non-trivial connected components than $H$ has.  Hence $e$ is not a non-trivial cut edge of $H$, and we conclude that $H$ does not have any non-trivial cut edges.

Now, assume instead that $H$ is quasi-eulerian.  Let $\F=\{T_1,\dotso , T_k\}$ be an Euler family for $H$ and suppose, for the sake of obtaining a contradiction, that $e\in E(H)$ is a strong cut edge of $H$.  Without loss of generality, let $T_1=v_0ev_1e_2v_2\dotso v_{\ell-1}e_{\ell}v_0$ be the component of $\F$ that traverses $e$.

Since $e$ is a strong cut edge of $H$, its vertices lie in distinct connected components of $H\setminus e$ (see Remark~\ref{remark:cutedge}).  However, we can see that $v_1e_2v_2\dotso v_{\ell-1}e_{\ell}v_0$ is a walk in $H\setminus e$ from one of the vertices of $e$ to another, a contradiction.  Hence $H$ contains no strong cut edges.
\end{proof}
\cleardoublepage

\chapter{Eulerian Properties of Triple Systems}\label{chapter:STS}

As we have seen, there are some existing results about Euler families and Euler tours for 3-uniform hypergraphs, and particularly triple systems.  It seems natural to begin our exploration with triple systems since they are highly structured hypergraphs.  In this chapter, we first find that it is not difficult to show that a TS($n,\lambda$) is eulerian if $\lambda\geq 2$.  We present the result as it was originally proven in a paper by \v{S}ajna and the author \cite{SW}, although more recent developments have greatly simplified the proof. These updated tools will be discussed in Chapter~\ref{chapter:covering}.

After giving the proof for $\lambda\geq 2$, we introduce some new notation in order to prove that TS($n,1$) is eulerian for all $n\geq 13$, and we use techniques that are specific to Steiner triple systems.  There are three other Steiner triple systems for $n<13$; in two cases we construct an Euler tour, while the remaining case TS($3,1$) is degenerate and does not admit an Euler tour.

\section{Basic Tools and Lemmas}

The following Theorem~\ref{thm:inc1} and Lemma~\ref{thm:inc2} are simple observations about a hypergraph and its incidence graph, but are nevertheless important and will be used freely in our upcoming proof of Theorem~\ref{thm:eulerianTS}, regarding triple systems of index 2 or greater.\\

\begin{thm}\label{thm:inc1}{\rm \cite[Theorem 3.11]{BS1}} Let $H$ be a hypergraph without empty edges, and let $G$ be its incidence graph.  Then $H$ is connected if and only if $G$ is connected.\\
\end{thm}

\begin{lem}\label{thm:inc2}{\rm \cite[Lemma 2.8]{BS1}} Let $H=(V,E)$ be a hypergraph, and suppose $|V|\geq 2$ and that $H$ has no empty edges.  If $v\in V$ is such that $\{v\}\not\in E$, then $\mathcal{G}(H- v)=\mathcal{G}(H)- v$.
\end{lem}





\section{Main Results}
We use a different proof strategy depending on the index of the triple system, so we divide the results into two main sections.

\subsection{Results on Triple Systems of Index 2 or Greater}
Here we present the main theorem for triple systems that are not Steiner triple systems.  We develop new tools later, in Chapter~\ref{chapter:covering}, that enable us to prove this very simply.  However, we present a proof without those tools in order to properly motivate their development.

\begin{thm}\label{thm:eulerianTS} Let $H$ be a 3-uniform hypergraph in which every pair of vertices lie together in at least two edges.  Then $H$ is eulerian.
\end{thm}

\begin{proof} Since $H$ is a covering 3-hypergraph, Corollary~\ref{cor:quasieuleriancovering} implies that $H$ has an Euler family. Let $\F$ be a minimum Euler family of $H$, with $|\F| = k$. Suppose $k\geq 2$.  Let $G$ be the incidence graph of $H$ and $\G$ the spanning subgraph of $G$ corresponding to $\F$. By Corollary~\ref{cor:correspondence},  $\G$ has exactly $k$ non-trivial connected components, say $G_1, \ldots, G_k$. We show that two of these components can be re-routed to yield an Euler family with fewer than $k$ closed strict trails.

Fix $i \in \{ 1,2\}$. We first wish to find a  v-vertex $u_i \in V(G_i)$ that is not a cut vertex. Let $H_i$ be a hypergraph whose incidence graph is $G_i$.  Note that $H_i$ is a graph since each e-vertex of $G_i$ has degree 2. Moreover, Theorem~\ref{thm:inc1} implies that $H_i$ is connected since $G_i$ is. Hence $H_i$ has a vertex $u_i$ that is not a cut vertex \cite[Corollary 2.7]{BM2}. It follows that $H_i - u_i$ is connected, and since $\mathcal{G}(H_i- u_i) = \mathcal{G}(H_i)- u_i=G_i- u_i$ by Theorem~\ref{thm:inc2}, the graph $G_i- u_i$ is connected by Theorem~\ref{thm:inc1}. Therefore, we have that $u_i$ is a v-vertex of $G_i$ that is not a cut vertex.

By our assumption on $H$, v-vertices $u_1$ and $u_2$ are adjacent in $G$ to two common e-vertices, say $e$ and $f$. In $G$, $e$ is adjacent to $u_1$, $u_2$, and exactly one other vertex, because $|e|=3$.  Since $u_1$ and $u_2$ are in distinct connected components of $\G$ and $\deg_{\G}(e)=2$, exactly one of $u_1$ and $u_2$ is a neighbour of $e$  in $\G$, and the same is true of $f$.

{\sc Case 1: $e$ and $f$ are  adjacent in $\G$ to the same v-vertex in $\{u_1,u_2\}$,} say to $u_1$.  Since $u_1$ is not a cut vertex of $G_1$, we have that $G_1- u_1$ is connected.  Hence  $G_1\setminus\{u_1e,u_1f\}$ is either connected, or has $u_1$ as an isolated vertex and $G_1- u_1$ as the unique non-trivial connected component.  In either case, $G_1\setminus\{u_1e,u_1f\}$ has a unique non-trivial connected component; call it $G_1^*$.  Then, since $u_2e$ joins a vertex of $G_1^*$ to a vertex of $G_2$, we have that $(G_1^*\cup G_2)+\{u_2e,u_2f\}$ is connected.  Obtain a graph $G^*$ from $\G$ by replacing connected components $G_1$ and $G_2$ with the connected component $(G_1^*\cup G_2)+\{u_2e,u_2f\}$.

Now, $G^*$ is a  subgraph of $G$ that either does not contain $u_1$, or else $\deg_{G^*}(u_1)=\deg_{\G}(u_1) - 2$. Furthermore, $\deg_{G^*}(u_2)=\deg_{\G}(u_2) + 2$, and $\deg_{G^*}(x)=\deg_{\G}(x)$ for all $x\in V(G)\setminus\{u_1,u_2\}$.  Therefore, Theorem~\ref{thm:incidencegraph} states that $G^*$ corresponds to an Euler family $\F^*$ of $H$.  Since $G^*$ has fewer non-trivial connected components than $\G$, we conclude that $\F^*$ contains fewer closed strict  trails than $\F$, a contradiction.

{\sc Case 2: $e$ and $f$ are adjacent in $\G$ to distinct v-vertices in $\{u_1,u_2\}$;}  say $e$ is adjacent to $u_1$, and $f$ is adjacent to $u_2$.  Since $G_1$ and $G_2$ are even graphs, they have no cut edges \cite[Theorem 2.10]{BM}, and so  $G_1^*=G_1\setminus u_1e$ and $G_2^*=G_2\setminus u_2f$ are connected. It follows that $(G_1^*\cup G_2^*)+ \{u_1f,u_2e\}$ is connected.  Obtain a graph $G^*$ from $\G$ by replacing connected components $G_1$ and $G_2$ with the connected component $(G_1^*\cup G_2^*)+ \{u_1f,u_2e\}$.  Then the degree of each vertex in $G^*$ is the same as in $\G$, so Theorem~\ref{thm:incidencegraph} implies that $G^*$ corresponds to an Euler family $\F^*$ of $H$.  Since $G^*$ has fewer non-trivial connected components than $\G$, as in the previous case, we see that $\F^*$ contradicts the minimality of $\F$.

We conclude that $k=1$, and the unique member of $\F$ is an Euler tour of $H$.
\end{proof}









The following corollary settles the matter for triple systems of order 3.\\

\begin{cor}\label{cor:TS3} Let $H=(V,E)$ be a TS-($3,\lambda$).  Then $H$ is eulerian if and only if $\lambda\geq 2$.
\end{cor}

\begin{proof} If $\lambda = 1$, then it is evident that $H$ has only a single edge.  Clearly, a hypergraph with a single edge cannot be eulerian.

On the other hand, if $\lambda\geq 2$, then Theorem~\ref{thm:eulerianTS} implies that $H$ is eulerian.
\end{proof}

\subsection{Results on Steiner Triple Systems}
The proof that Steiner triple systems are eulerian is more difficult than for triple systems of higher index.  We will need some extra tools to accomplish this.  We define an {\em edge-corresponding function} and the {\em labeled 2-section} so that we can recreate the hypergraph from its 2-section.  Following these definitions, Lemma~\ref{lem:cyclecorrespondence} will tell us when cycles in the 2-section correspond to cycles in the hypergraph.\\

\begin{defn}{\rm Let $H$ be a hypergraph with incidence function $\psi_H$ and let $G$ be the 2-section of $H$ with incidence function $\psi_G$.  An {\em edge-corresponding function} from $G$ to $H$ is any function $\varphi: E(G)\rightarrow E(H)$ satisfying the following:
\begin{itemize}
\item for all $e\in E(G)$, we have $\varphi(e)=e'$ for some $e'\in E(H)$ such that $\psi_G(e)\subseteq \psi_H(e')$; and
\item for all $e_1, e_2\in E(G)$, $e_1\neq e_2$, with $\psi_G(e_1)=\psi_G(e_2)$, we have $\varphi(e_1)\neq \varphi(e_2)$.
\end{itemize}
}
\end{defn}

\begin{defn}{\rm Let $H=(V,E)$ be a 3-uniform hypergraph, $G$ be the 2-section of $H$, and $\varphi$ be an edge-corresponding function from $G$ to $H$.  Define a labeling function $\ell:E(G) \rightarrow V(G)$, for all $uv\in E(G)$, by $\ell(uv)=x$ where $\varphi(uv)=\{u,v,x\}$.  Then we call $(G, \ell)$ the {\em labeled 2-section} of $H$.
}
\end{defn}

\begin{lem}\label{lem:cyclecorrespondence} Let $H=(V,E)$ be a simple 3-uniform hypergraph with 2-section $G$.  Let $\varphi$ be an edge-corresponding function from $G$ to $H$, and let $(G,\ell)$ be the labeled 2-section of $H$.

Let $C=v_0e_1v_1\dotso v_{k-1}e_kv_0$ be a cycle in $G$, let $C'=v_0\varphi(e_1)v_1\dotso v_{k-1}\varphi(e_k)v_0$, and assume that for all $i\in\mathds{Z}_k$, we have $(\ell(v_{i-1}v_i), \ell(v_iv_{i+1}))\neq (v_{i+1}, v_{i-1})$.  Then $C'$ is a cycle in $H$.
\end{lem}

\begin{proof}
Let $C$ and $C'$ be be defined as above.  Since each edge $\varphi(e_i)$ of $H$ is incident with both $v_{i-1}$ and $v_i$ (because $e_i$ is, and $e_i\subseteq \varphi(e_i)$), and $v_{i-1}\neq v_i$ for all $i\in\mathds{Z}_k$, we have that $C'$ is a walk.  Clearly, $C'$ is closed.  We must show that no vertices and no edges are repeated in the sequence.

Since $C$ is a cycle, we have that $v_0,\dotso ,v_{k-1}$ are pairwise distinct, so $C'$ does not repeat any vertices.  Suppose, however, that $\varphi(e_i)=\varphi(e_j)$ for some $i,j\in\{1,\dotso , k\}, i\neq j$.

If $|i-j|\not\equiv 1$ (mod $k$), then $\varphi(e_i)$ and $\varphi(e_j)$ are nonconsecutive in $C'$, and consequently $e_i$ and $e_j$ are disjoint.  Since $e_i\subseteq \varphi(e_i)$ and $e_j\subseteq \varphi(e_j)=\varphi(e_i)$, we conclude that $|\varphi(e_i)|\geq 4$, a contradiction.

On the other hand, if $|i-j|\equiv 1$ (mod $k$), then without loss of generality, we have $j\equiv i+1$ (mod $k$).  Then we must have $\varphi(e_i)=\varphi(e_{i+1})=\{v_{i-1},v_i,v_{i+1}\}$.  However, $e_i=v_{i-1}v_i$ and $e_{i+1}=v_iv_{i+1}$.  Therefore, we must have ($\ell(e_i), \ell(e_{j}))=(v_{i+1},v_{i-1})$; that is, $(\ell(v_{i-1}v_i), \ell(v_iv_{i+1}))= (v_{i+1}, v_{i-1})$.  This contradicts our assumption on the labels.

Therefore, the edges of $C'$ are pairwise distinct, so $C'$ is a cycle in $H$.
\end{proof}

We need just one more small definition that will only be used in Theorem~\ref{thm:eulerianSTS}.\\

\begin{defn}{\rm {\bf (Cycle Exchange)} \cite{BM}
Let $G$ be a graph, and $C=v_0v_1v_2\dotso v_{k-1}v_0$ be a cycle of $G$.  Fix some $i,j\in\mathds{Z}_k$ such that $|j-i|\geq 2$.

If $v_iv_j,v_{i+1}v_{j+1}\in E(G)$, then we can perform a {\em cycle exchange} on $v_iv_{i+1}$ and $v_jv_{j+1}$ by deleting $v_iv_{i+1}$ and $v_jv_{j+1}$ from the cycle and adding $v_iv_j$ and $v_{i+1}v_{j+1}$ to the cycle and rerouting appropriately.  The cycle we obtain from this cycle exchange is $C'=v_0v_1\dotso v_{i-1}v_iv_jv_{j-1}v_{j-2}\dotso v_{i+1}v_{j+1}v_{j+2}\dotso v_{k-1}v_0$.
}
\end{defn}

We are now ready to present the main theorem for Steiner triple systems.  Our strategy is to obtain a subgraph of the incidence graph $G$ that corresponds to an Euler tour.  Recall from Theorem~\ref{thm:incidencegraph} that we must find a spanning subgraph of $G$ that is even and connected, except that it may have up to one isolated v-vertex.  We take advantage of this exception by first finding a cycle $C$ of $H$ that traverses all of the vertices except one.  We then show that $H\setminus E(C)$ is quasi-eulerian using Theorem~\ref{thm:eulerfamily}.  We translate both $C$ and the trails of the Euler family into $G$, and this gives a subgraph of $G$ that is connected and satisfies the degree requirements.\\

\begin{thm}\label{thm:eulerianSTS}Let $H$ be an STS($n$) with $n\geq 13$.  Then $H$ is eulerian.
\end{thm}

\begin{proof}
Let $H=(V,E)$. We shall construct an Euler tour of $H$ by concatenating an $(n-1)$-cycle $C_H$ with an Euler family of $H\setminus E(C_H)$.

Let $G=(V,E_G)$ be the complete graph on vertex set $V$. Let $\varphi$ be an edge-corresponding function from $G$ to $H$, and let $\ell$ be the corresponding labeling function. Observe that, for every $v\in V$, the set $\{ e\in E_G: \ell(e)=v\}$ is a perfect matching of $G-v$, and keep in mind that $n$ is odd.

Fix a vertex $u_0\in V$, and let $F=\{ e\in E_G: \ell(e)=u_0\}$. Since  $G - u_0$ is a complete graph, we can extend $F$ to a Hamilton cycle $C=v_0v_1\ldots v_{n-2}v_0$ of $G - u_0$.  We may assume without loss of generality that $\ell(v_{2i}v_{2i+1})=u_0$ for all $0 \leq i \leq  \frac{n-3}{2}$,  while all other edges of $C$ have labels distinct from $u_0$.

For each $v \in V$, define $m_v=|\{ e \in E(C): \ell(e)=v\}|$, and let $M_C=\max \{m_v: v\in V\setminus\{u_0\}\}$.  Moreover, we define the {\em labeling profile}  of $C$ as the sequence ${\bf L}_C=(m_{u_1}, m_{u_2}, \ldots )$ containing all positive terms $m_{u_i}$ such that $m_{u_1} \ge m_{u_2} \ge \ldots$, for $u_1,u_2,\ldots , u_{n-1} \in V\setminus \{u_0\}$. Observe that $\sum\limits_{i=1}^{n-1} m_{u_i}=\frac{n-1}{2}$.

Let $C$ be a Hamilton cycle of $G - u_0$ that contains $F$ and minimizes  $M_C$. We shall now show that $M_C \le\frac{n-9}{2}$.

Suppose $M_C>\frac{n-9}{2}$.  Since $n \ge 13$, we have $M_C \ge \frac{n-7}{2} \ge 3$. Let $v\in V\setminus\{u_0\}$ be such that $m_v=M_C$.  Observe that at most three edges of $C$ have their label in  $V\setminus\{u_0,v\}$, and the following cases arise.

{\sc Case 1: ${\bf L}_C = (M_C)$.}  Let  $v_{2i-1}v_{2i}$ and $v_{2j-1}v_{2j}$ be two edges of $C$ with label $v$.  Perform a cycle exchange on these two edges to obtain a Hamilton cycle $C'$ of $G - u_0$. Since adjacent edges in $G$ cannot have the same label, we have exchanged edges labeled $v$ for edges that are not labeled $v$, so we can see that $M_{C'} \le \max\{M_C-2, 2\} < M_C$, a contradiction.

{\sc Case 2: ${\bf L}_C = (M_C, i)$, where $i \in \{1,2,3\}$.}  Let $u\in V\setminus \{u_0\}$ be such that $m_u=i$, and let $e$ and $e'$ be edges of $C$ with label $u$ and $v$, respectively. Performing a cycle exchange on these two edges, we obtain a Hamilton cycle $C'$ of $G - u_0$ that contains $F$ and satisfies $M_{C'}= \max \{M_C-1, 2\}<M_C$, a contradiction.

{\sc Case 3: ${\bf L}_C = (M_C, 2, 1)$ or ${\bf L}_C = (M_C, 1, 1)$.} Let $u$ and $x$ be distinct vertices in $V\setminus \{u_0\}$ such that $m_u=1$ and $m_x \in \{ 1,2\}$.  In $C$, there are two edges incident to $u$: one is labeled $u_0$ and the other edge, $e$, has a label in $\{v,x\}$.  If $\ell(e)=v$, then let $e'$ be an edge of $C$ labeled $x$; otherwise, let $e'$ be an edge labeled $v$.  Performing a cycle exchange on  $e$ and $e'$ we obtain a Hamilton cycle $C'$ of $G - u_0$ that contains $F$.  Now  $M_{C'}=\max\{M_C-1, 2\}<M_C$, a contradiction.

{\sc Case 4: ${\bf L}_C = (M_C, 1, 1, 1)$.}  Let $u,x,y\in V\setminus\{u_0\}$ be such that $m_u=m_x=m_y=1$.  If $n \ge 15$, then $M_C \ge \frac{n-7}{2} \ge 4$. We perform a cycle exchange to replace two of the edges labeled $v$ in $C$ to obtain a Hamilton cycle $C'$ of $G - u_0$ that contains $F$. Now $M_{C'} \le \max \{M_C-2, 3\} < M_C$, a contradiction. 

We may thus assume $n=13$, and since ${\bf L}_C = (M_C, 1, 1, 1)$, we have $M_C=3$.  Obtain $C'$ by performing a cycle exchange on an edge labeled $v$ and the edge labeled $u$.  Unless both of the new edges are labeled $x$ or both are labeled $y$, we have that $M_{C'} < M_C$, a contradiction.  Hence assume, without loss of generality, that both of the new edges are labeled $x$.  Then ${\bf L}_{C'}=(3, 2, 1)$, and we proceed as in Case 3 to obtain a contradiction.

We conclude that $M_C \le \frac{n-9}{2}$.

Next, we show that our cycle $C=v_0 v_1 \dotso v_{n-2} v_0$ corresponds to a cycle of $H$. Let $C_H=v_0 \varphi(v_0 v_1) v_1 \dotso v_{n-2} \varphi(v_{n-2} v_0) v_0$. In order to apply Lemma~\ref{lem:cyclecorrespondence}, we must show that, for all $i\in\mathds{Z}_k$, we have $(\ell(v_{i-1}v_i),\ell(v_iv_{i+1}))\neq (v_{i+1},v_{i-1})$.  However, this follows immediately from the fact that every second label of an edge in $C$ is $u_0$, and $u_0$ is not equal to $v_i$ for any $i\in\mathds{Z}_k$.  Hence we may apply Lemma~\ref{lem:cyclecorrespondence} to find that $C_H$ is a cycle in $H$.  Since $C_H$ has length $n-1$ and it does not traverse $u_0$, it must traverse all the vertices of $V\setminus\{u_0\}$.

Let $E(C_H)$ denote the set of edges of $H$ traversed by $C_H$, and let  $H'=(H\setminus E(C_H))- u_0$.  We wish to show that $H'$ has no cut edges. Since $H'$ is 3-uniform, Theorem~\ref{thm:eulerfamily} will then imply that $H'$ has an Euler family.

Take any $u \in V(H')=V\setminus\{u_0\}$.  We have that $\deg_H(u) = \frac{n-1}{2}$.  In $C$, there are exactly two edges incident with $u$, and $u$ appears as the label of at most $M_C$ other edges. Therefore,  $u$ is contained in at most $M_C+2$ edges traversed by $C_H$, and $\deg_{H'}(u) \ge \frac{n-1}{2}-(M_C+2)$. In particular, since  $M_C \le \frac{n-9}{2}$, we have $\deg_{H'}(u) \ge 2$.

Suppose $H'$ has a cut edge $e=\{x_1,x_2,x_3\}$.  Let $X_1$ be the vertex set of a connected component of $H'\setminus e$, and $X_2=V(H')\setminus X_1$.  For $i=1,2$, we may assume without loss of generality that $x_i \in X_i$, and since $\deg_{H'\setminus e}(x_i) \ge 1$,  we must have that $|X_i| \ge 3$.

Consider the set $S=\{ \{y_1,y_2\}: y_i \in X_i, i=1,2 \}$. Observe that every pair in $S$ must be contained in an edge in $E(C_H)\cup \{e\}$.  There are $\frac{n-1}{2}$ edges in $E(C_H)$ that contain $u_0$, each containing one pair of vertices that could be an element of $S$; there are $\frac{n-1}{2}$ edges in $E(C_H)$ that contain three vertices of $X_1 \cup X_2$, each containing up to two pairs from $S$; and finally, $e$ contains exactly two pairs of $S$.  Therefore, edges in $E(C_H)\cup\{e\}$ contain at most $\frac{n-1}{2}+2\cdot\frac{n-1}{2}+2=\frac{3n+1}{2}$ pairs of $S$. Hence
$$|S|=|X_1| \cdot |X_2| \le \frac{3n+1}{2}.$$

Since $|X_i| \geq 3$ for $i=1,2$, and $|X_1|+|X_2|=n-1$, the left-hand side of this inequality is a quadratic function in $|X_1|$, and is minimum when $|X_1|=3$.  Hence $3(n-4) \le |X_1|\cdot |X_2| \le \frac{3n+1}{2}$, which simplifies to $3n \le 25$, contradicting the assumption $n\ge 13$.  We conclude that $H'$ has no cut edges.

Theorem~\ref{thm:eulerfamily} now implies that $H'$ admits an Euler family $\F=\{ T_1,\ldots,T_k\}$. Since $C_H$ contains a common anchor with each trail in $\F$, a concatenation of $T_1,\ldots,T_k$ with $C_H$ results in an Euler tour of $H$.
\end{proof}

Finally, we give constructions of Euler tours in Steiner triple systems of order less than 13.\\

\begin{lem}\label{lem:smallSTS}Let $H=(V,E)$ be a TS($n,1$).  If $4\leq n < 13$, then $H$ is eulerian.
\end{lem}

\begin{proof} Given that $4\leq n< 13$ and $H$ is a Steiner triple system, there are precisely two possibilities for $H$ up to isomorphism \cite{CD}.

{\sc Case 1: $H$ is the unique STS(7).} Let $V=\{1,2,\dotso ,7\}$ and (without loss of generality) $E=\{123,145,167,247,256,346,357\}$.  Then $$(1,145,5,256,2,123,3,357,7,247,4,346,6,167,1)$$ is an Euler tour of $H$.

{\sc Case 2: $H$ is the unique STS(9).}  Let $V=\{1,2,\dotso , 9\}$ and (without loss of generality) $E=\{123, 456, 789, 147, 258, 369, 159, 267, 348, 168,249,357\}$.  Then $$(1,147,7,357,3,369,6,267,2,123,3,348,8,168,6,456,4,249,9,789,8,258,5,159,1)$$ is an Euler tour of $H$.
\end{proof}

We summarize our results in this chapter with a single theorem.

\begin{thm}\label{thm:STS} Let $H$ be a TS($n,\lambda$).  Then $H$ is eulerian if and only if $H$ is not a TS$(3,1)$.
\end{thm}

\begin{proof} If $H$ is a TS($3,1)$, then Corollary~\ref{cor:TS3} states that $H$ is not eulerian.

Conversely, assume $H$ is not a TS$(3,1)$.  If $\lambda\geq 2$, then Theorem~\ref{thm:eulerianTS} states that $H$ is eulerian.  If $\lambda=1$ and $n\geq 13$, then Theorem~\ref{thm:eulerianSTS} states that $H$ is eulerian.  If $\lambda=1$ and $4\leq n<13$, then Lemma~\ref{lem:smallSTS} demonstrates that $H$ is eulerian.
\end{proof}

Having settled the matter for triple systems, we shall set our sights on more general hypergraphs in upcoming chapters.  We will see that some of the strategies we used (for triple systems of index at least 2) have great value when proving the more advanced results.
\cleardoublepage

\chapter{Eulerian Properties of Covering Hypergraphs}\label{chapter:covering}

In Chapter~\ref{chapter:STS}, we saw that all triple systems, with one exception, are eulerian.  However, for Steiner triple systems, we relied on the highly structured nature of the edge set in order to establish an Euler tour; for triple systems of higher index, we needed the existence of two edges containing a certain pair of vertices.

The motivation for this chapter is to prove that Steiner quadruple systems are eulerian.  An easy observation is that a Steiner quadruple system has, embedded in it, a Steiner triple system plus some additional edges.  To see this, delete any vertex from a Steiner quadruple system, and note that the edges that contained that vertex now have cardinality 3 and jointly contain every pair of vertices.  However, the ``leftover'' edges present a problem because it is not clear that we can add them to our Euler tour of the Steiner triple system.  What we have, in fact, can be reduced to a covering 3-hypergraph, so we need to know whether such hypergraphs are eulerian in order to easily show that Steiner quadruple systems are eulerian.

We will spend the majority of our time in this chapter proving that covering 3-hypergraphs are indeed eulerian, formally written as follows.\\

\setcounter{theo}{0}

\begin{thm}\label{thm:coveringinductionbase} Let $H$ be a covering 3-hypergraph.  Then $H$ is eulerian if and only if it has at least two edges.
\end{thm}

However, we need not stop at Steiner quadruple systems.  Using the process described earlier allows us to reduce {\em any} covering $(k+1)$-hypergraph to a covering $k$-hypergraph.  We will use Theorem~\ref{thm:coveringinductionbase} as the basis of induction for the main result of this chapter, as follows.\\

\begin{thm}\label{thm:coveringinduction}Let $k\geq 3$, and let $H$ be a covering $k$-hypergraph.  Then $H$ is eulerian if and only if it has at least two edges.
\end{thm}

\section{Basic Facts}
In order to prove a helpful result about cut vertices in graphs, we will need to present some additional facts about graphs.\\

\begin{defn}{\rm \cite{BM} A graph $T$ is called a {\em tree} if it is connected and has no cycles.  If $G$ is a graph, then a {\em spanning tree} of $G$ is a connected spanning subgraph of $G$ that has no cycles.  A {\em leaf} of a tree is a vertex of degree 1.
}
\end{defn}
Note that every connected graph has a spanning tree, which can be obtained simply by successively deleting edges of the graph that are not cut edges.\\

\begin{prop}{\rm \cite[Proposition 4.2]{BM}} Let $T$ be a tree of order at least 2.  Then $T$ has at least two leaves.\qed
\end{prop}
\begin{prop}\label{prop:noncut} Let $G$ be a connected graph of order at least 2.  Then $G$ has at least two vertices that are not cut vertices.  Furthermore, if $G$ admits a cycle of length $k$, then $G$ has at least $k$ vertices that are not cut vertices.
\end{prop}

\begin{proof} Since $G$ is connected, there exists a spanning tree $T$ in $G$.  Then there exist two leaves $u$ and $v$ of $T$, which are not cut vertices of $T$.  These leaves cannot be cut vertices of $G$ either, for $G-u$ and $G-v$ contain $T-u$ and $T-v$ as spanning trees, respectively.  Hence $G$ has at least two vertices that are not cut vertices. 

Now suppose there exists a cycle $C$ of length $k\geq 3$ in $G$.  Each vertex of $C$ is either a cut vertex of $G$ or not.  Fix some $v\in V(C)$, and suppose $v$ is a cut vertex of $G$.  Then $G-v$ has connected components $G_1,\dotso , G_{\ell}$, with $\ell\geq 2$.  The vertices in $C-v$ are all in the same connected component, so assume $G_{\ell}$ is the connected component containing the vertices of $C-v$.  Now consider the following two connected subgraphs of $G$: let $X_1=G[V(G_1)\cup \dotso\cup V(G_{\ell-1})\cup\{v\}]$ and let $X_2 = G[V(G_{\ell})\cup\{v\}]$.  

We first show that $X_1$ and $X_2$ are connected.  Let $a,b\in V(X_i)$ for some $i\in\{1,2\}$.  Since $G$ is connected, it admits an $av$-path $P_1$.  Since $v$ is a cut vertex of $G$, we have that $a$ and the internal vertices of $P_1$ all lie in $V(G_j)$ for some $j$.  Likewise, there exists a $vb$-path $P_2$ in $G$ such that $b$ and the internal vertices of $P_2$ all lie in $V(G_j)$ for some (possibly different) $j$.  Hence $aP_1vP_2b$ is an $ab$-path in $X_i$, so $X_i$ is connected for $i=1,2$.

Note that $X_1$ is a connected graph of order at least 2, since it contains $v$ and at least one other vertex in $G_1$.  We have shown that such a graph contains two vertices that are not cut vertices.  Let $u$ be one of the non-cut vertices of $X_1$ that is distinct from $v$.

Observe that $G-u$ is connected: for any $x\in V(X_1-u)$, since $X_1-u$ is connected, there exists an $xv$-walk in $X_1-u$. For any $y\in V(X_2)$, since $X_2$ is connected, there exists a $yv$-walk in $X_2$.  Therefore, for any $a,b\in V(G-u)$, there exists an $ab$-walk obtained by concatenating an appropriate $av$-walk with a $vb$-walk.  Hence $G-u$ is connected, so $u$ is not a cut vertex of $G$.

Therefore, letting $v$ vary over all $V(C)$, we obtain a collection $N=\{v':v\in V(C)\}$ in which $v'=v$ if $v$ is not a cut vertex of $G$, and $v'$ is the non-cut vertex determined by $v$, as above, otherwise.  Note that if $v'$ is of the latter form, then it is not connected to vertices of $V(C)\setminus \{v\}$ in $G-v$.  What remains to be seen is that $N$ contains $|V(C)|$ distinct vertices.

Let $x,y\in V(C)$ be different from each other, with corresponding vertices $x',y'\in N$, and suppose $x'=y'$.  Denote $z=x'=y'$.  If $z\in V(C)$, then we have $x=y$, so assume $z\not\in V(C)$.  Let $P_1$ be a shortest $zx$-path in $G$ and $P_2$ be a shortest $zy$-path in $G$.

Since $x\neq y$, we have $x\in V(C)\setminus\{y\}$, so $x$ and $z$ are not connected in $G-y$.  From this, we conclude that $P_1$ traverses $y$, so $P_1$ contains a $zy$-path.  Similarly, we find that $P_2$ traverses $x$, so $P_2$ contains a $zx$-path.  By the minimality of $P_1$ and $P_2$, this implies that $|V(P_1)| > |V(P_2)| > |V(P_1)|$, a contradiction.  

Therefore, if $x\neq y$, then $x'\neq y'$.  Hence $N$ contains $|V(C)| = k$ distinct vertices that are not cut vertices of $G$, so $G$ has at least $k$ vertices that are not cut vertices.
\end{proof}

\begin{cor}\label{cor:noncut} Let $H$ be a hypergraph with an Euler family $\F$ and let $G$ be the incidence graph of $H$.  Let $\G$ be the subgraph of $G$ corresponding to $\F$, and let $G_1$ be one of the non-trivial connected components of $\G$.  Then at least two v-vertices of $G_1$ are not cut vertices of $\G$.  Furthermore, if $G_1$ admits a cycle containing $k$ of its v-vertices, then at least $k$ v-vertices of $G_1$ are not cut vertices of $\G$.
\end{cor}

\begin{proof} Since $G_1$ corresponds to a component of $\F$, it has at least two vertices, its v-vertices have even degree at least 2, and its e-vertices have degree 2.  Hence $G_1$ is the incidence graph to some connected hypergraph $H'$ that has order at least 2 and whose edges have cardinality 2.  Then Proposition~\ref{prop:noncut} implies that $H'$ has at least two vertices that are not cut vertices.  Since $H'$ has only edges of cardinality 2, we may apply Theorem~\ref{thm:cutvvertex} to find that these two vertices that are not cut vertices correspond to v-vertices of $G_1$ that are not cut vertices.  Hence $G_1$ has two v-vertices that are not cut vertices.

Furthermore, suppose $C$ is a cycle in $G_1$ traversing $k$ v-vertices.  Then $C$ corresponds to a cycle of length $k$ in $H'$ by Remark~\ref{remark:traversals}.  In this case, Proposition~\ref{prop:noncut} implies that $H$ has at least $k$ vertices that are not cut vertices.  Once again, the cut vertices of $H$ correspond exactly to cut v-vertices of $G$ by Theorem~\ref{thm:cutvvertex}, so we conclude that at least $k$ v-vertices of $G$ are not cut vertices.
\end{proof}

\section{Main Tool: Interchanging Cycles}

As proved in Theorem~\ref{thm:incidencegraph}, an Euler family of a hypergraph $H$ can be represented in its incidence graph $G$ by an even subgraph of $G$ in which every e-vertex has degree 2.  If this subgraph has only one non-trivial connected component, then $H$ is eulerian.  Therefore, our strategy in attacking Theorem~\ref{thm:coveringinductionbase} is to start with a subgraph corresponding to an Euler family, which may have many non-trivial connected components.  We will then swap some edges of the subgraph with some edges outside of it to obtain a new subgraph with fewer non-trivial connected components.

We accomplish this by taking the symmetric difference of the corresponding subgraph with an {\em interchanging cycle.}\\

\begin{defn}{\rm {\bf (Interchanging Cycles)} Let $H$ be a hypergraph with Euler family $\F$.  Denote by $G$ the incidence graph of $H$, and by $\G$ the subgraph of $G$ corresponding to $\F$.  For any $e\in E(G)$, we call $e$ a {\em $\G$-edge} if $e\in E(\G)$, and a {\em non-$\G$-edge} otherwise.

Suppose $C$ is a cycle in $G$.  If every e-vertex of $G$ traversed by $C$ is incident with exactly one $\G$-edge of $C$ (that is, an edge in $E(\G)\cap E(C)$), then we call $C$ an {\em $\F$-interchanging cycle}.  Further, if $C$ is an $\F$-interchanging cycle and $\G\Delta C$ has fewer non-trivial connected components than $\G$, then we call $C$ an {\em $\F$-diminishing cycle}. \phantom{aaaaa}
}
\end{defn}

\begin{lem}\label{lem:interchange} Let $H$ be a hypergraph and let $\F$ be an Euler family of $H$.  Let $G$ be the incidence graph of $H$, and $\G$ be the subgraph of $G$ corresponding to $\F$.  Suppose there is an $\F$-interchanging cycle $C$ in $G$.  Then $\G\Delta C$ is a subgraph of $G$ corresponding to an Euler family of $H$.  Additionally, if $\F$ is minimum, then $C$ is not $\F$-diminishing.
\end{lem}

\begin{proof}We show that, in $\G\Delta C$, every v-vertex has even degree and every e-vertex has degree 2.  

Since $C$ is a cycle of $G$, every vertex of $G$ is incident with exactly 0 or 2 of its edges.  Let $v$ be a v-vertex of $G$.  If $v$ is incident with no edges of $C$, then $\deg_{\G\Delta C}(v)=\deg_{\G}(v)$.  If $v$ is incident with two edges of $C$, then any number of these may be $\G$-edges: if zero, then $\deg_{\G\Delta C}(v)=\deg_{\G}(v) + 2$; if one, then $\deg_{\G\Delta C}(v)=\deg_{\G}(v)$; if two, then $\deg_{\G\Delta C}(v)=\deg_{\G}(v) - 2.$  In any case, we have that $\deg_{\G\Delta C}(v)$ is even because $\deg_{\G}(v)$ is.

Now let $e$ be an e-vertex of $G$.  If $e$ is incident with no edges of $E(C)$, then $\deg_{\G\Delta C}(e)=\deg_{\G}(e)=2$.  Otherwise, $e$ is incident with two edges of $E(C)$, and exactly one of them is a $\G$-edge because $C$ is an $\F$-interchanging cycle.  Once again, we have $\deg_{\G\Delta C}(e)=\deg_{\G}(e)=2$.

Therefore, $\G\Delta C$ is a subgraph of $G$ corresponding to an Euler family of $H$, as claimed.

The last statement follows immediately from the definition of $\F$-diminishing.
\end{proof}

\section{Technical Lemmas}

In this section, we produce results that will be repeatedly used in the larger theorems to come.  These reflect broad techniques or strategies that we employ to find $\F$-diminishing cycles, leaving only specific case work to be done in the main theorems. 

Lemma~\ref{lem:dimcycle} gives us a simple way of determining whether an $\F$-interchanging cycle is $\F$-diminishing, by investigating the connectedness of the incidence graph.

Corollaries~\ref{cor:simpledimcycle} and~\ref{cor:dimcycle} provide ways to find a cycle that satisfies the conditions of Lemma~\ref{lem:dimcycle}.

Corollary~\ref{cor:3comps} is a direct result of Corollary~\ref{cor:dimcycle}, assuring us that we can always find an $\F$-diminishing cycle of the kind described therein as long as the corresponding subgraph of the incidence graph has at least three connected components.

Finally, Corollary~\ref{cor:isolated} summarizes these results, reducing the set of cases we need to check.  We conclude that the only non-eulerian covering 3-hypergraphs must have Euler families of cardinality exactly 2.\\

\begin{lem}\label{lem:dimcycle} Let $H$ be a hypergraph with incidence graph $G$.  Let $\F$ be an Euler family of $H$ and let $\G$ be the subgraph of $G$ corresponding to $\F$.  Let $C$ be an $\F$-interchanging cycle in $G$, and suppose $C$ traverses vertices that lie in distinct connected components $G_1,\dotso , G_k$ of $\G$.

If, for each $i=1,\dotso , k$, we have that $G_i\setminus E(C)$ has just one non-trivial connected component, then $(\G\Delta C)[V(G_1)\cup \dotso \cup V(G_k)]$ has just one non-trivial connected component.  
\end{lem}

\begin{proof}  Note that all vertices of $C$ lie in $\G$, but not all edges.  Hence we can write $C=x_0P_0y_0x_1P_1y_1\dotso x_{\ell-1}P_{\ell-1}y_{\ell-1}x_0$, where
\begin{itemize}
\item $x_0,\dotso , x_{\ell-1},$ and $y_0,\dotso , y_{\ell-1}$ are vertices of $\G$;
\item $P_i$, for each $i\in\mathds{Z}_{\ell}$, is a maximal $x_iy_i$-subpath of $C$ contained in a single connected component, say $G_{f(i)}$, of $\G$.
\end{itemize}

Hence $y_0x_1,y_1x_2,\dotso y_{\ell-2}x_{\ell-1},$ and $y_{\ell-1}x_0$ each join a pair of vertices that are in different connected components of $\G$.  Hence these edges are non-$\G$-edges of $C$, so we see that these are in $E(\G\Delta C)$.  Note that some $P_i$ may be trivial, and $f(i)=f(j)$ for $i\neq j$ is possible.

Define an auxiliary graph $L$ as follows:
\begin{itemize}
\item $V(L)= \{\text{connected components of }G_i\setminus E(C),\text{ for }i=1,\dotso, k\}$;
\item $L$ has an edge $gh$ for each edge of $C$ with one end in $g$ and the other in $h$.
\end{itemize}

$L$ has a single non-trivial connected component, with an Euler tour $T_L$ corresponding to $C$.  If $T_L$ traverses the non-trivial connected component of $G_i\setminus E(C)$ for all $i\in\{1,\dotso,k\}$, then $(\G\Delta C)[V(G_1)\cup\dotso\cup V(G_k)]$ has a unique non-trivial connected component.  (Observe that this graph is the union of the connected components of all $G_i\setminus E(C)$, together with the edges of $C\setminus E(\G)$, that is, edges corresponding to $L$.)

Suppose, to the contrary, that there exists $s\in\{1,\dotso ,k\}$ such that the non-trivial connected component of $G_s\setminus E(C)$ is not traversed by $T_L$.  Then for all $i\in\{0,\dotso,\ell-1\}$ such that $f(i)=s$, we have that $P_i$ is trivial, and so $G_s$ contains no edges of $C$.  Hence $G_s\setminus E(C)=G_s$.  But $C$ traverses a vertex of $G_s$ by assumption, so $G_s$ must be trivial --- a contradiction.

We conclude that $C_L$ does traverse the non-trivial connected component of $G_i\setminus E(C)$ for all $i\in\{1,\dotso,k\}$, so $(\G\Delta C)[V(G_1)\cup\dotso\cup V(G_k)]$ has a unique non-trivial connected component.
\end{proof}

\begin{cor}\label{cor:simpledimcycle}Let $H$ be a hypergraph with incidence graph $G$.  Let $\F$ be an Euler family of $H$ and let $\G$ be the subgraph of $G$ corresponding to $\F$.  Let $C=v_0e_1v_1e_2\dotso v_{k-1}e_kv_0$ be a cycle in $G$.

If $C$ traverses exactly one $\G$-edge of each of at least two connected components of $\G$, then $C$ is an $\F$-diminishing cycle.
\end{cor}

\begin{proof}
Let $G_i$, for $i=1,\dotso ,k$, be the non-trivial connected components of $\G$.  Since the connected components that have an edge traversed by $C$ in the assumption must be non-trivial, let us assume that $G_1$ and $G_2$ are two of the connected components of $\G$ that each have exactly one $\G$-edge traversed by $C$.

Let $i\in\{1,2\}$.  Then Theorem~\ref{thm:incidencegraph} demonstrates that $G_i$ is even, so Theorem~\ref{thm:eulertour} implies that it is eulerian.  Proposition~\ref{prop:2-edge-connected} states that $G_i$ is 2-edge-connected, hence $G_i\setminus E(C)$ is connected.  Finally, Lemma~\ref{lem:dimcycle} implies that $(\G\Delta C)[V(G_1)\cup V(G_2)]$ has just one non-trivial connected component. 

Hence $\G\Delta C$ has $k-1$ non-trivial connected components.  By definition, we have that $C$ is an $\F$-diminishing cycle.
\end{proof}

\begin{cor}\label{cor:dimcycle}Let $H$ be a covering 3-hypergraph with incidence graph $G$.  Let $\F$ be an Euler family of $H$ and let $\G$ be the subgraph of $G$ corresponding to $\F$.  Let $C=v_0e_1v_1e_2\dotso v_{k-1}e_kv_0$ be a cycle in $G$.
\begin{description}
\item[(i)] Suppose $v_0,\dotso , v_{k-1}$ are v-vertices of pairwise distinct connected components of $\G$, and that none of $v_0,\dotso , v_{k-1}$ are cut vertices of $\G$.  Further, suppose that at least two of these v-vertices are non-isolated in $\G$.  Then $C$ is an $\F$-diminishing cycle.
\item[(ii)] If $C$ traverses a v-vertex in every non-trivial connected component of $\G$, then $H$ is eulerian, with an Euler tour corresponding to $\G\Delta C$.
\end{description}
\end{cor}

\begin{proof} We first show that $C$ is an $\F$-interchanging cycle by showing that every e-vertex traversed by $C$ is incident with exactly one edge in $E(\G)\cap E(C)$.  The e-vertices traversed by $C$ are $e_1,e_2,\dotso , e_k$, so fix $i\in\{1,2,\dotso ,k\}$.  Then $e_i$ is adjacent to $v_{i-1}$ and $v_i$ in $C$ (where indices of v-vertices are evaluated modulo $k$).  Since $v_{i-1}$ and $v_i$ lie in distinct connected components of $\G$, we cannot have that both $v_{i-1}e_i$ and $v_ie_i$ are $\G$-edges.  On the other hand, since $\deg_{G}(e_i)=3$, we know that $e_i$ is incident with only one non-$\G$-edge in $G$, hence at least one of $v_{i-1}e_i$ and $v_ie_i$ is a $\G$-edge.  Thus we conclude that exactly one of $v_{i-1}e_i$ and $v_ie_i$ is a $\G$-edge, demonstrating that $C$ is an $\F$-interchanging cycle.

Now, let $G_0,\dotso ,G_{k-1}$ be the connected components of $\G$ containing $v_0,\dotso , v_{k-1}$, respectively.  We will show that $G_0\setminus E(C),\dotso , G_{k-1}\setminus E(C)$ each have just one non-trivial connected component, so that we may apply Lemma~\ref{lem:dimcycle}.

Consider connected component $G_i$, for some $i\in\mathds{Z}_k$.  It contains a single v-vertex $v_i$ traversed by $C$, so necessarily there are only up to two edges of $C$ contained in $G_i$.  If zero or one edge, then $G_i\setminus E(C)$ is connected because $G_i$ is a 2-edge-connected graph.  However, if there are two edges of $C$ in $G_i$, then we observe that, since $v_i$ is not a cut vertex of $G_i$, we know that $G_i-v_i$ is connected.  Then $G_i-v_i$ is a connected subgraph of $G_i\setminus E(C)$ with just one vertex fewer: this vertex may be isolated in $G_i\setminus E(C)$ or not.  In either case, there is just one non-trivial connected component in $G_i\setminus E(C)$.

Therefore, we may apply Lemma~\ref{lem:dimcycle} and conclude that $(\G\Delta C)[V(G_0)\cup \dotso\cup V(G_{k-1})]$ has just one non-trivial connected component.  By assumption, at least two of $v_0,\dotso, v_{k-1}$ are not isolated in $\G$, so at least two of $G_0,\dotso , G_{k-1}$ are non-trivial.  Hence $\G\Delta C$ has fewer non-trivial connected components than $\G$, so $C$ is an $\F$-diminishing cycle.

In addition, if $\G$ has no non-trivial connected components except those among $G_0,\dotso , G_{k-1}$, then $\G\Delta C$ has just one non-trivial connected component.  In that case, Lemma~\ref{lem:interchange} and Theorem~\ref{thm:incidencegraph} imply that $\G\Delta C$ corresponds to an Euler tour, and thus $H$ is eulerian.
\end{proof}

\begin{cor}\label{cor:3comps}Let $H$ be a covering 3-hypergraph with incidence graph $G$.  Let $\F$ be an Euler family of $H$ and $\G$ be the subgraph of $G$ corresponding to $\F$.  If $\G$ has at least three connected components, then $H$ is eulerian.
\end{cor}

\begin{proof}
Let $G_i$, for $i\in\mathds{Z}_k$, be the connected components of $\G$, and suppose $k\geq 3$.  Corollary~\ref{cor:noncut} implies that each non-trivial connected component of $\G$ has a v-vertex that is not a cut vertex of $\G$, whereas any v-vertices of trivial connected components are not cut vertices of $\G$, either.  From each $G_i$, choose a v-vertex $v_i$ that is not a cut vertex.  We will construct an $\F$-interchanging cycle that traverses $v_i$ for each $i\in\mathds{Z}_k$.

Since $H$ is a covering 3-hypergraph, there exists an edge $e_i\in E(H)$ containing $v_{i-1}$ and $v_i$, for each $i=1,\dotso , k$ (where $v_k$ denotes $v_0$). Define $C=v_0e_1v_1e_2\dotso v_{k-1}e_kv_0,$ a closed walk of $G$.

To show that $C$ is a cycle, we must show that $e_1,\dotso , e_k$ are pairwise distinct.  Suppose, however, that $e_i=e_j$ for some $i < j\leq k$.  Then $v_{i-1}, v_i, v_{j-1}, v_j\in e_i$.  However, since $|e_i|=3$ in $H$, we must have either that $v_i=v_{j-1}$, or $v_j=v_{i-1}$ (the latter implying $i=1$ and $j=k$).  In either case, we see that $e_i$ is adjacent in $G$ to three v-vertices that lie in different connected components in $\G$.  This implies that $e_i$ is incident with at least two non-$\G$-edges, a contradiction since $e_i$ is incident with two $\G$-edges.  Hence $e_i\neq e_j$ for every $i<j\leq k$.

Now, by our selection of $v_1,\dotso ,v_k$, we see that $C$ is a cycle that traverses every connected component of $\G$.  By Lemma~\ref{lem:onecc}, at most one of these connected components is trivial, so at least two are not.  Therefore, Corollary~\ref{cor:dimcycle} allows us to conclude that $H$ is eulerian.
\end{proof}

\begin{cor}\label{cor:isolated}Let $H$ be a covering 3-hypergraph with incidence graph $G$.  Let $\F$ be a minimum Euler family of $H$ and $\G$ be the subgraph of $G$ corresponding to $\F$.  Then $|\F|\leq 2$.  If $\G$ has an isolated vertex, then $H$ is eulerian.
\end{cor}

\begin{proof}Suppose $|\F|\geq 3$.  Then $\G$ has at least three non-trivial connected components.  Corollary~\ref{cor:3comps} implies that $H$ is eulerian, a contradiction.  So $|\F|\leq 2$.

Suppose $\G$ has an isolated vertex.  If $|\F|=1$, then $H$ is eulerian.  On the other hand, if $|\F|=2$, then $\F$ has three connected components: two non-trivial and one trivial.  Corollary~\ref{cor:3comps} implies that $H$ is eulerian.
\end{proof}

\section{Proof of Theorem~\ref{thm:coveringinductionbase} for $n\geq 7$}

This is a long proof, so it is necessary to summarize it in sections.

We first prove Lemma~\ref{lem:mindegree}, which will lay down the groundwork for our proof strategy.  Lemma~\ref{lem:order7setup} will furnish our proof with most of the details needed to set up the case work.  In Theorem~\ref{thm:order7}, we will begin with a minimum Euler family $\F$ for $H$ that has two components and which minimizes the degree of an arbitrary, but fixed, vertex $v_0$.  Lemma~\ref{lem:mindegree} assures us that we need only find an $\F$-interchanging cycle --- not an $\F$-diminishing cycle, as expected --- as long as it traverses $v_0$.  No matter what the Euler family looks like after we perform such an interchange, it will either contradict one of the minimality conditions, or be an Euler family with three components.

There are two main stages involved in proving Theorem~\ref{thm:order7}.

In the first stage, handled by Lemma~\ref{lem:order7setup}, we explore the neighbourhood of $v_0$.  Since we would like to find an $\F$-interchanging cycle traversing $v_0$, we need to consider every way that this might not be possible.  Such a cycle would have to traverse two of $v_0$'s neighbours in $H$, called $a$ and $b$, which are not joined to $v_0$ via a 2-path in $\G$.  This reduces the problem to finding a particular $ab$-path in $G$.  We assume such a path does not exist, and so we force $a$ and $b$ to be in certain locations relative to $v_0$.

Once Lemma~\ref{lem:order7setup} is completed, we use these facts about $a, b$, and the rest of $\G$ and start the proof of Theorem~\ref{thm:order7} by breaking it into cases.  The cases depend on the location of at least two v-vertices that have not been explored already.  We need both of these v-vertices to complete the proof, which is why the proof only applies when the order of $H$ is at least 7.  Generally, we find an explicit $\F$-diminishing cycle, but sometimes we find an $\F$-interchanging cycle that brings us back to one of the previous, solved cases.\\

\begin{lem}\label{lem:mindegree} Let $H$ be a covering 3-hypergraph with incidence graph $G$, and let $v_0\in V(H)$.  Let $\mathcal{F}$ be a minimum Euler family for $H$ such that $\deg_{\G}(v_0)$ is minimum over all minimum Euler families, where $\G$ is the subgraph of $G$ corresponding to $\F$.

Suppose $C$ is an $\F$-interchanging cycle.  If $H$ is not eulerian, then $\deg_{\G\Delta C}(v_0) \geq \deg_{\G}(v_0)$ and $\G\Delta C$ has exactly two connected components, both non-trivial.
\end{lem}

\begin{proof}
Assume $H$ is not eulerian.  Then Corollary~\ref{cor:isolated} implies that $\G$ has exactly two connected components, both non-trivial.  Suppose, to the contrary, that $C$ is an $\F$-interchanging cycle with $\deg_{\G\Delta C}(v_0)$ $< \deg_{\G}(v_0).$  If $\G\Delta C$ has more non-trivial connected components than $\G$, then $\G\Delta C$ has at least three connected components, so Corollary~\ref{cor:3comps} implies that $H$ is eulerian, a contradiction.  Hence $\G\Delta C$ has exactly two non-trivial connected components, and so it corresponds to a minimum Euler family, contradicting the fact that $\deg_{\G}(v_0)$ is minimum over all minimum Euler families.  

Therefore, we find that $\deg_{\G\Delta C}(v_0) \geq \deg_{\G}(v_0)$.  Furthermore, Corollary~\ref{cor:isolated} implies that $\G\Delta C$ cannot have any isolated vertices.  Hence $\G\Delta C$ has exactly two connected components, both non-trivial.
\end{proof}

\begin{lem}\label{lem:order7setup} Let $H$ be a covering 3-hypergraph of order $n\geq 7$ that is not eulerian.  Let $G$ be the incidence graph of $H$, and let $v_0$ be a v-vertex of $G$.  Let $\F$ be a minimum Euler family of $H$ with the property that $\deg_{\G}(v_0)$ is minimum over all minimum Euler families of $H$, where $\G$ is the subgraph of $G$ corresponding to $\F$.   

Suppose further that there exist distinct e-vertices $e_1,e_2\in V(G)$ such that $v_0e_1, v_0e_2\in E(\G).$  Let $e_1=v_0v_1a$ for some $v_1$ and $a$, and let $e_2=v_0v_2b$ for some vertices $v_2$ and $b$ of $H$.  Assume that $v_1e_1$ and $v_2e_2$ are $\G$-edges.

Then the following hold:
\begin{description}
\item[(0)] $|\F|=2$ and $a\neq b$;
\item[(1)] If there exists an e-vertex $e$ of $G$ with $e\not\in\{e_1,e_2\}$, such that $ea, eb\in E(G)$, then both $ea$ and $eb$ are $\G$-edges.
\item[(2)] $a$ and $b$ lie in the same connected component of $\G$;
\item[(3)] Let $x\in V(H)\setminus\{v_0,v_1,v_2\}$ such that $x$ lies in a different connected component of $\G$ from $a$ and $b$.  Then there exists an edge $e_x=abx$ in $H$; any edge of $H$ containing $x$ and $a$, or $x$ and $b$, must be parallel to $e_x$; and $xe_x$ is a non-$\G$-edge.
\item[(4)] $a$ and $b$ do not lie in the same connected component of $\G$ as $v_0$.

\end{description}
\end{lem}

\begin{proof}
{\bf (0):} If $|\F|\geq 3$, then Corollary~\ref{cor:isolated} implies that $H$ is eulerian, contradicting our assumption.  Likewise, if $|\F|\leq 1$, then $H$ is eulerian or empty, a contradiction. Hence $|\F|=2$.

Suppose $a=b$.  Then $C=v_0e_1ae_2v_0$ is an $\F$-interchanging cycle with $\deg_{\G\Delta C}(v_0) < \deg_{\G}(v_0)$, a contradiction by Lemma~\ref{lem:mindegree}.  Hence $a\neq b$.

{\bf (1):} Suppose to the contrary that at least one (and hence exactly one) of $ea$ and $eb$ is a non-$\G$-edge.  Then $C=v_0e_1aebe_2v_0$ is an $\F$-interchanging cycle with $\deg_{\G\Delta C}(v_0) < \deg_{\G}(v_0)$, a contradiction by Lemma~\ref{lem:mindegree}.

Therefore, both $ea$ and $eb$ are $\G$-edges, as claimed.

{\bf (2):} Suppose $a$ and $b$ lie in different connected components of $\G$.  Since $H$ is a covering 3-hypergraph, there exists an edge $e\in E(H)$ containing both $a$ and $b$.  By our supposition, one of $ea$ and $eb$ is necessarily a non-$\G$-edge, so (1) implies that $e=e_1$ or $e=e_2$.  Without loss of generality, assume $e=e_1$; hence $v_1=b$.

Observe that $v_0, v_1,$ and $v_2$ lie in the same connected component of $\G$.  We must have $a\neq v_2$, otherwise $a$ and $b$ are in the same connected component of $\G$.  Let $e_3\in E(H)$ be an edge containing $a$ and $v_2$.  Let $C_1=ae_1be_2v_2e_3a$.  Then $C_1$ is an $\F$-interchanging cycle because $a$ and $v_2$ are in distinct connected components of $\G$, so at least one (and hence exactly one) of $v_2e_3$ and $ae_3$ is not in $\G$.

Define $\F'$ to be an Euler family corresponding to $\G\Delta C_1$, and let $G_{\F'}=\G\Delta C_1$ for brevity's sake.  Since $H$ is not eulerian, Lemma~\ref{lem:mindegree} implies that $G_{\F'}$ has exactly two connected components, both non-trivial. We claim that $v_0$ and $v_2$ lie in different connected components of $G_{\F'}$: first of all, we have that $be_2v_0e_1a$ is a path in $G_{\F'}$, so these v-vertices all lie in a connected component of $G_{\F'}$ together.  If $v_2$ also lies in this connected component of $G_{\F'}$, then all the other v-vertices of the graph lie in the same connected component with them as well, since their connectivity to $v_0, b, v_2,$ and $a$ are the same in $\G$ as in $G_{\F'}$.  Since we know $G_{\F'}$ must have two connected components, we must have $v_2$ disconnected from $v_0,b,$ and $a$.

Now, we cannot have that $a$ is isolated in $\G$ because Corollary~\ref{cor:isolated} implies that $H$ is eulerian, contradicting our assumptions on $H$.  Let $c\neq a$ be a v-vertex in the same connected component of $\G$ as $a$.  We wish to construct an $\F'$-interchanging cycle traversing $b,v_2,c,$ and $v_0$, but we must first find an edge containing $c$ and $b$, and a different edge containing $c$ and $v_2$.  

Suppose, however, that there exists an edge $e'=v_2bc\in E(H)$.  Since $c$ is in a different connected component of $\G$ from $b$ and $v_2$, we have that $e'c$ is a non-$\G$-edge.  Since $e'\not\in V(C_1)$, this implies $e'c$ is a non-$G_{\F'}$-edge as well.  However, since $v_2$ and $b$ are in different connected components of $G_{\F'}$, we have that either $e'v_2$ or $e'b$ is a non-$G_{\F'}$-edge, contradicting the fact that $e'$ is only incident with one non-$G_{\F'}$-edge.

So $v_2bc\not\in E(H)$.  Then there exist distinct edges $e_4,e_5\in E(H)$, where $e_4$ contains $b$ and $c$, and $e_5$ contains $v_2$ and $c$.  Since $c$ is in a different connected component of $\G$ from both $v_2$ and $b$, we have that one of $be_4$ and $ce_4$ is a non-$\G$-edge and one of $v_2e_5$ and $ce_5$ is a non-$\G$-edge.  As neither $e_4$ nor $e_5$ is in $V(C_1)$, the status of these edges as non-$\G$-edges is preserved so that they are non-$G_{\F'}$-edges as well.

Now, let $C_2=v_0e_1be_4ce_5v_2e_2v_0$.  Observe that $C_2$ is an $\F'$-interchanging cycle of $G_{\F'}$ with $\deg_{G_{\F'}\Delta C_2}(v_0) < \deg_{\G}(v_0).$ Since $|\F'|=2$ and $\deg_{G_{\F'}}(v_0)=\deg_{\G}(v_0)$ is minimum over all minimum Euler families, we may apply Lemma~\ref{lem:mindegree} to obtain a contradiction.

Therefore, we must have $a$ and $b$ lying in the same connected component of $\G$, completing the proof of (2).

{\bf (3):} Let $x\in V(H)\setminus\{v_0,v_1,v_2\}$ and assume $x$ lies in a different connected component of $\G$ from $a$ and $b$.  Since $H$ is a covering 3-hypergraph, there exists an edge $e\in E(H)$ containing $a$ and $x$, and some $e'\in E(H)$ containing $b$ and $x$.  If $e\neq e'$, then $C=v_0e_1aexe'be_2v_0$ is an $\F$-interchanging cycle with $\deg_{\G\Delta C}(v_0) < \deg_{\G}(v_0),$ contradicting Lemma~\ref{lem:mindegree}.  Hence $e=abx=e'$, and $ae$ and $be$ are $\G$-edges by (1), so $xe$ is a non-$\G$-edge.  

{\bf (4):} Suppose that $a$ and $b$ lie in the same connected component of $\G$ as $v_0$.  Since $\G$ consists of two non-trivial connected components, there must be at least two v-vertices in the other connected component.  Let $G_1$ be the connected component containing $v_0$ and let $G_2$ be the other connected component, which contains some v-vertices $c$ and $d$.

(3) states that there are edges $e_c=abc$ and $e_d=abd$ such that $ce_c$ and $de_d$ are non-$\G$-edges.

Now, there exists an edge $e_3\in E(H)$ containing $v_0$ and $c$.  Note $e_3\not\in\{e_1,e_2\}$ because $e_3$ contains $c$, but neither $e_1$ nor $e_2$ do.  Let $C=v_0e_1ae_cce_3v_0$.  Then $C$ is an $\F$-interchanging cycle because $c$ and $v_0$ do not lie in the same connected component of $\G$.  If $v_0e_3\in V(G_1)$, then there are two $\G$-edges of $C$ incident with $v_0$ and so $\deg_{\G\Delta C}(v_0) < \deg_{\G}(v_0)$, contradicting Lemma~\ref{lem:mindegree}.  Hence $v_0e_3$ is a non-$\G$-edge and so $e_3\in V(G_2)$.

Since $G_2$ is a 2-edge-connected graph and $G_2 \setminus E(C) = G_2 \setminus ce_3$, we have that $G_2\setminus E(C)$ is connected.

Since $G_1$ is 2-edge-connected, we have that $G_1\setminus v_0e_1$ is connected.  Then $ae_dbe_c$ is an $ae_c$-path in $(G_1\setminus v_0e_1)\setminus ae_c$, so $(G_1\setminus v_0e_1)\setminus ae = G_1\setminus E(C)$ is connected, by Proposition~\ref{prop:G-uv}.  Then $\G\Delta C$ has just one non-trivial connected component by Lemma~\ref{lem:dimcycle}, contradicting the fact that $H$ is not eulerian.

We conclude that $a$ and $b$ are not in the same connected component of $\G$ as $v_0$, completing the proof of (4).
\end{proof}

With the setup afforded by Lemma~\ref{lem:order7setup}, we know a lot about five of the v-vertices of $\G$, namely $v_0,v_1,v_2,a,$ and $b$.  In Theorem~\ref{thm:order7}, we need only investigate the positions of two other vertices, and this will be enough to show that there must be an $\F$-interchanging cycle that contradicts our assumption that $H$ is not eulerian.\\

\begin{thm}\label{thm:order7} Let $H$ be a covering 3-hypergraph of order $n\geq 7$.  Then $H$ is eulerian.
\end{thm}

\begin{proof} Suppose $H$ is not eulerian.  We establish the same setup as in the statement of Lemma~\ref{lem:order7setup}.  Let $G$ be the incidence graph of $H$, and fix some v-vertex $v_0$.  Corollary~\ref{cor:quasieuleriancovering} states that $H$ is quasi-eulerian, so let $\F$ be a minimum Euler family of $H$ with the property that $\deg_{\G}(v_0)$ is minimum over all minimum Euler families of $H$.  By Corollary~\ref{cor:isolated}, we know that $|\F|=2$ and that $v_0$ is not isolated in $\G$, hence $\deg_{\G}(v_0)\geq 2$.

We have distinct e-vertices $e_1,e_2\in V(G)$ such that $v_0e_1, v_0e_2\in E(\G).$  Let $e_1=v_0v_1a$ for some $v_1$ and $a$, and let $e_2=v_0v_2b$ for some vertices $v_2$ and $b$.  (Note that we cannot assume that $v_1,v_2,a,$ and $b$ are pairwise distinct, except that $v_1\neq a$ and $v_2\neq b$.)  Since the e-vertices $e_1$ and $e_2$ are each incident with one non-$\G$-edge, let $ae_1$ and $be_2$ be non-$\G$-edges.  Then $v_1e_1$ and $v_2e_2$ are $\G$-edges.

Now $H$ satisfies the assumptions of Lemma~\ref{lem:order7setup}, with v-vertices and e-vertices named exactly as in the statement of that lemma.  We have that $v_0,v_1,$ and $v_2$ lie in the same connected component $G_1$ of $\G$, and $a$ and $b$ lie in the other connected component $G_2$ of $\G$.  (Hence $\{v_1,v_2\}\cap \{a,b\}=\emptyset$ and $a\neq b$, but it is possible that $v_1=v_2$.)

Since $|V(H)|\geq 7$, we must have at least two v-vertices in $\G$ besides $v_0,v_1,v_2,a$, and $b$.  We shall complete the proof by splitting into cases depending on where two of these additional v-vertices lie.

{\sc Case 1: there exist v-vertices $c\in V(G_1)\setminus \{v_0,v_1,v_2\}$ and $d\in V(G_2)\setminus \{a,b\}$.}

Since $c$ lies in a different connected component from $a$ and $b$, Lemma~\ref{lem:order7setup} (3) states that there exists an edge $e_c=abc$.

Let $e_3\in E(H)$ be an edge containing $a$ and $d$, and let $e_4\in E(H)$ be an edge containing $b$ and $d$.  Suppose $ae_3$ or $de_3$ is a non-$\G$-edge.  We know there exists an edge $f\in E(H)$ containing $c$ and $d$, and $f\neq e_3$ because $e_3$ must contain three v-vertices of $G_2$, but $f$ does not.  Hence $C=v_0e_1ae_3dfce_cbe_2v_0$ is an $\F$-interchanging cycle with $\deg_{\G\Delta C}(v_0) < \deg_{\G}(v_0),$ a contradiction with Lemma~\ref{lem:mindegree}.  Hence $ae_3$ and $de_3$ are $\G$-edges.  By a similar argument, we conclude that $be_4$ and $de_4$ are $\G$-edges.  (Note that this also implies that $e_3\neq e_4$.)

Let $x$ be the v-vertex such that $xe_3$ is a non-$\G$-edge (so $x\not\in\{a,d\}$).  If $x=b$, then $C=v_0e_1ae_3be_2v_0$ is an $\F$-interchanging cycle with $\deg_{\G\Delta C}(v_0) < \deg_{\G}(v_0),$ contradicting Lemma~\ref{lem:mindegree}.  Hence $x\neq b$.

If $x\in V(G_2)$, then there exists an edge $f$ containing $x$ and $c$, and as we have seen before, we have that $C=v_0e_1ae_3xfce_cbe_2v_0$ is an $\F$-interchanging cycle with $\deg_{\G\Delta C}(v_0) < \deg_{\G}(v_0),$ contradicting Lemma~\ref{lem:mindegree}.  Hence $x\in V(G_1)$.

If $x\in \{v_0,v_1\}$, then $C=xe_1ae_3x$ is an $\F$-interchanging cycle with just one $\G$-edge in each of $G_1$ and $G_2$.  Then Corollary~\ref{cor:simpledimcycle} states that $C$ is an $\F$-diminishing cycle, contradicting Lemma~\ref{lem:interchange}.

If $x\in V(G_1)\setminus\{v_0,v_1,v_2\}$, then by Lemma~\ref{lem:order7setup} (3), there exists an edge $e_x=abx$.  Then $C=v_0e_1ae_3xe_xbe_2v_0$ is an $\F$-interchanging cycle with $\deg_{\G\Delta C}(v_0) < \deg_{\G}(v_0),$ contradicting Lemma~\ref{lem:mindegree}.

We thus conclude that $x=v_2$ (and $v_1\neq v_2$), so $e_3=adv_2$.  By an analogous argument (swapping the roles of $a$ and $b$, of $v_1$ and $v_2$, and of $e_1$ and $e_2$), we may conclude that $e_4=bdv_1$.

Let $e_5$ be an edge containing $c$ and $d$; note that $e_5\not\in \{e_c,e_3,e_2\}$ because none of those edges contain both $c$ and $d$.  Let $C=be_cce_5de_3v_2e_2b$. Then $C$ is an $\F$-interchanging cycle. Define $G_{\F'}=\G\Delta C$ and let $\F'$ be an Euler family corresponding to $G_{\F'}$.  Lemma~\ref{lem:mindegree} states that $G_{\F'}$ has exactly two connected components, both non-trivial.

Notice that $be_2$ and $v_0e_2$ are $G_{\F'}$-edges, so $v_0$ and $b$ are in the same connected component of $G_{\F'}$.  We also have that $be_4d$ is a 2-path of $\G$ that is edge-disjoint from $C$, so $be_4d$ is a 2-path of $G_{\F'}$ as well.  Finally, we have that $v_0e_1v_1$ is a 2-path of $G_{\F'}$ as it is in $\G$ and is edge-disjoint from $C$.  Hence $v_0,b,d,$ and $v_1$ are all in the same connected component of $G_{\F'}$.

On the other hand, we have that $ae_3v_2$ and $ae_cc$ are 2-paths of $G_{\F'}$, so $a, v_2$, and $c$ are in a connected component of $G_{\F'}$ together.  Of course, this must be a different connected component from the one containing $v_0$, for if not, then every non-isolated vertex of $G_{\F'}$ lies in the same connected component (any non-isolated v-vertex other than these seven must lie in the same connected component as one of them, for they do so in $\G$).

To summarize what we know about $G_{\F'}$, we have $v_0,v_1,b,$ and $d$ in one connected component, and $v_2,a,$ and $c$ in the other.  (Any additional v-vertices may be in either connected component.)

Observe that, since $\F$ is a minimum Euler family such that $\deg_{G_{\F'}}(v_0) = \deg_{\G}(v_0)$ is a minimum, we have that $G_{\F'}$ satisfies the initial assumptions made of $\G$.  We may apply Lemma~\ref{lem:order7setup} to $G_{\F'}$, with the roles of $v_2$ and $b$ swapped.  Since $d$ is an additional v-vertex in $v_0$'s connected component and $c$ is an additional v-vertex in $a$'s connected component, we conclude that $G_{\F'}$ falls into Case 1 of this proof as well, with the roles of $c$ and $d$ swapped.

We conclude that, in particular, there exists an e-vertex in $G_{\F'}$ that corresponds to $e_4$ in $\G$.  Since $e_4=bdv_1\in E(H)$, the corresponding edge must be $v_1v_2c$.

Call this edge $e_6=v_1v_2c\in E(H)$.  Let $C'=v_0e_1ae_cbe_4v_1e_6v_2e_2v_0$.  Observe that $C'$ is an $\F'$-interchanging cycle with $\deg_{G_{\F'}\Delta C'}(v_0) < \deg_{G_{\F'}}(v_0),$ contradicting Lemma~\ref{lem:mindegree}.  This concludes Case 1.

{\sc Case 2: $a$ and $b$ are the only v-vertices in $G_2$.}

Let $X=V(H)\setminus\{v_0,v_1,v_2,a,b\}.$  We know $|X|\geq 2$ since $|V(H)|\geq 7$, and $X\subseteq V(G_1)$ by assumption.

Fix some $x\in X$.  There exists an edge $e_x=abx$ by Lemma~\ref{lem:order7setup} (3). 

{\sc Claim:} {\em For any $e\in E(H)$: if $xe$ is a non-$\G$-edge, then $e=abx$.}

Suppose $xe$ is a non-$\G$-edge, but $e\neq abx$.  Evidently, we must have $a\not\in e$ and $b\not\in e$ by Lemma~\ref{lem:order7setup}.3.

If $v_0\in e$, then $C=v_0e_1ae_xxev_0$ is an $\F$-interchanging cycle with $\deg_{\G\Delta C}(v_0) < \deg_{\G}(v_0),$ since $v_0e$ is a $\G$-edge.  This contradicts Lemma~\ref{lem:mindegree}.

If $e=xv_1v_2$, then this implies $v_1\neq b$ and so there must be an edge $e_4$ (distinct from $e_2$) containing $b$ and $v_1$.  Let $C=v_0e_1ae_xxev_1e_4be_2v_0$.  Observe that $C$ is an $\F$-interchanging cycle with $\deg_{\G\Delta C}(v_0) < \deg_{\G}(v_0),$ contradicting Lemma~\ref{lem:mindegree}.

Therefore, we must have $y\in e$ for some $y\in X, y\neq x$.  Lemma~\ref{lem:order7setup} (3) states that there exists an edge $e_y=aby\in E(H)$.  Let $C=v_0e_1ae_xxeye_ybe_2v_0$.  Observe that $C$ is an $\F$-interchanging cycle with $\deg_{\G\Delta C}(v_0) < \deg_{\G}(v_0),$ contradicting Lemma~\ref{lem:mindegree}.

This completes the proof of the claim.

Now, there exists an edge $e_3$ containing $v_0$ and $x$, and, by the above claim, we know that $xe_3$ is a $\G$-edge.  Suppose first that $v_0e_3$ is a $\G$-edge, and let $y$ be the v-vertex such that $ye_3$ is a non-$\G$-edge.

If $y=a$, then $C=v_0e_1ae_3v_0$ is an $\F$-interchanging cycle with $\deg_{\G\Delta C}(v_0) < \deg_{\G}(v_0),$ contradicting Lemma~\ref{lem:mindegree}.  If $y=b$, then $C=v_0e_2be_3v_0$ similarly yields a contradiction.

Hence $y\in V(G_1)$.  Note that we may apply Lemma~\ref{lem:order7setup} to $\G$ in which $x$ takes the role of $v_2$, and $y$ takes the role of $b$, and $e_3=v_0xy$ takes the role of $e_2=v_0v_2b$ (all other named vertices and edges take the role of themselves). Then Lemma~\ref{lem:order7setup} (2) implies that $a$ and $y$ lie in the same connected component, a contradiction. 

Since we obtain a contradiction if $v_0e_3$ is a $\G$-edge, it must be a non-$\G$-edge.  Let $C=v_0e_1ae_xxe_3v_0$.  Then $C$ is an $\F$-interchanging cycle.  Define $G_{\F'}=\G\Delta C$ and let $\F'$ be an Euler family corresponding to $G_{\F'}$.  Lemma~\ref{lem:mindegree} implies that $G_{\F'}$ has exactly two connected components, both non-trivial.  We will use $G_{\F'}$ to show that $v_1\neq v_2$.

We have that $v_0e_2v_2$ is a 2-path of $G_{\F'}$, so $v_0$ and $v_2$ are in the same connected component.

Let $z$ be an arbitrary element of $X\setminus\{x\}$.  By Lemma~\ref{lem:order7setup} (3), there exists an edge $e_z=abz\in E(H)$, and $ae_zb$ is a 2-path in $\G$.  Hence $ae_zb$ is also a 2-path in $G_{\F'}$.  This implies that $a$ and $b$ are in the same connected component of $G_{\F'}$.

We have that $ae_1v_1$ is a 2-path of $G_{\F'}$, so $a$ and $v_1$ are in the same connected component as well.

We also have that $be_xx$ is a 2-path of $G_{\F'}$, so $b$ and $x$ are in the same connected component of $G_{\F'}$.


To summarize what we know of $G_{\F'}$ so far, we have $a, b, v_1,$ and $x$ together in a connected component, and $v_0$ and $v_2$ together in a connected component.  For any $v\in V(H)\setminus\{v_0,v_1,v_2,a,b,x\}$, there exists a minimum-length path $P$ in $\G$ connecting $v$ to one of these six v-vertices that does not traverse any of the other six v-vertices, so $P$ is also a path in $G_{\F'}$.  Therefore, every vertex in $V(H)\setminus\{v_0,v_1,v_2,a,b,x\}$ is connected in $G_{\F'}$ to one of these six vertices.  Hence we cannot have $b$ and $v_2$ together in the same connected component, otherwise $G_{\F'}$ has only one connected component; but we know $G_{\F'}$ has exactly two connected components.

If $v_1=v_2$, then we immediately obtain a contradiction since that would imply that $G_{\F'}$ has just one non-trivial connected component.  Hence $v_1\neq v_2$, completing our claim, so we return our attention to $\G$ and $H$.  We now have that there exists an edge $e_5\neq e_1$ containing $a$ and $v_2$.

If $e_5=av_2y$ for some $y\in X$, then we obtain a contradiction since there cannot be an edge containing $a$ and $y$ without $b$, by Lemma~\ref{lem:order7setup} (3).

If $e_5=av_2b$, then $C'=v_2e_2be_5v_2$ is an $\F$-interchanging cycle with one $\G$-edge in each connected component.  Since the connected components are even graphs, Theorem~\ref{thm:eulertour} and Proposition~\ref{prop:2-edge-connected} together imply that they are 2-edge-connected, so $G_i\setminus E(C')$, for each $i\in\{1,2\}$, is connected.  Then $C'$ satisfies the conditions of Lemma~\ref{lem:dimcycle}, so $C'$ is $\F$-diminishing, contradicting Lemma~\ref{lem:interchange}.

If $e_5=av_2v_0$, then $C'=v_0e_1ae_5v_0$ is an $\F$-interchanging cycle with $\deg_{\G\Delta C'}(v_0) < \deg_{\G}(v_0),$ contradicting Lemma~\ref{lem:mindegree}.

Hence $e_5=av_2v_1$, so $C'=v_1e_1ae_5v_1$ is an $\F$-interchanging cycle.  Let $G_{\F^*}=\G\Delta C'$, and let $\F^*$ be an Euler family corresponding to $G_{\F^*}$. We have that $v_0e_1$ and $v_0e_2$ are $G_{\F^{*}}$-edges, while $e_1v_1$ and $e_2b$ are non-$G_{\F^*}$-edges.  Lemma~\ref{lem:mindegree} assures us that $G_{\F^*}$ has exactly two connected components, both non-trivial.

Note that $\F^*$ is a minimum Euler family with $\deg_{G_{\F^*}}(v_0) = \deg_{\G}(v_0)$ a minimum.  We may apply Lemma~\ref{lem:order7setup} to $G_{\F^*}$ with the roles of $a$ and $v_1$ swapped.  Lemma~\ref{lem:order7setup} (2) then shows that $b$ and $v_1$ lie in the same connected component of $G_{\F^*}$.  However, note that $G_2$ is a subgraph of $G_{\F^*}$, and $a$ and $v_0$ are in the same connected component in $G_{\F^*}$.  This implies that all the v-vertices lie in a single connected component in $G_{\F^*}$, contradicting our assumption that there are exactly two connected components.

In all cases, we obtain a contradiction, concluding Case 2.

{\sc Case 3: $v_0, v_1,$ and $v_2$ are the only v-vertices of $G_1$.}

Let $e_x=abx$ be an edge of $H$ containing $a,b$, and some vertex $x$.  By Lemma~\ref{lem:order7setup} (1), we have that $ae_x$ and $be_x$ are $\G$-edges, and hence $xe_x$ is a non-$\G$-edge.

If $x\in\{v_0,v_1\}$, then $C=xe_1ae_xx$ is an $\F$-diminishing cycle because $G_1\setminus E(C)$ and $G_2\setminus E(C)$ are each connected, meaning $\G\Delta C$ is connected.

If $x=v_2$, then $C=v_2e_2be_xv_2$ is an $\F$-diminishing cycle for the same reason as above.

In either case, we obtain a contradiction by Lemma~\ref{lem:interchange}.  Hence $x\in V(G_2)$.  

Since $G_1$ has at least two v-vertices that are not cut vertices (by Corollary~\ref{cor:noncut}), at least one of $v_0$ and $v_1$ is not a cut vertex.  Let $y\in\{v_0,v_1\}$ be a vertex that is not a cut vertex of $G_1$.  Let $e_3$ be an edge of $H$ containing $x$ and $y$.

Then $C=ye_1ae_xxe_3y$ is an $\F$-interchanging cycle.  Let $G_{\F'}=\G\Delta C$ and let $\F'$ be an Euler family corresponding to $G_{\F'}$.

Since $H$ is not eulerian, Lemma~\ref{lem:mindegree} implies that $G_{\F'}$ has exactly two connected components, both non-trivial.

Suppose $y=v_0$.  If $v_0e_3$ is a $\G$-edge, then $\deg_{G_{\F'}}(v_0) < \deg_{\G}(v_0),$ which contradicts Lemma~\ref{lem:mindegree}.  Hence $v_0e_3$ must be a non-$\G$-edge and we can conclude that $G_1\setminus E(C)$ is connected. Now, observe that $\F'$ is a minimum Euler family that minimizes $\deg_{G_{\F'}}(v_0) = \deg_{\G}(v_0)$, so it satisfies our assumptions on $\F$.  Since $a$ is also in the same connected component of $G_{\F'}$ as $v_0,v_1,$ and $v_2$ via edge $ae_1\in E(G_{\F'})$, we have that $v_0$ lies in the same connected component in $G_{\F'}$ as at least three other v-vertices. Then if we attempt to apply this theorem to $G_{\F'}$, we conclude that $\F'$ falls into either Case 1 or 2, and so we obtain a contradiction.

Hence $y=v_1.$  Assume first that $v_1$ and $x$ are in the same connected component of $G_{\F'}$.  Since $v_1$ is not a cut vertex of $G_1$, we know that either $G_1\setminus E(C)$ is connected, or $(G_1\setminus E(C)) - v_1$ is connected and $v_1$ is isolated in $G_1\setminus E(C)$.  If $v_1$ is isolated in $G_1\setminus E(C)$, then we deduce that it is isolated in $G_{\F'}$ as well, contradicting the assumption that it is in the same connected component as $x$.  Hence $G_1\setminus E(C)$ is connected.  Since $v_1$ and $x$ are assumed to be in the same connected component in $G_{\F'}$, and $b$ and $x$ lie in the same connected component of $G_{\F'}$ as each other (via path $be_xx$), we see that $v_0,v_1,v_2,a,x,$ and $b$ lie in a connected component of $G_{\F'}$ together.  Since all other v-vertices are in the same connected component as one of these in $G_{\F'}$, we see that $G_{\F'}$ is connected, a contradiction.

Hence $v_1$ and $x$ must be in different connected components of $G_{\F'}$.  Since $b$ and $x$ lie in the same connected component of $G_{\F'}$ (again via path $be_xx$), we conclude that $b$ and $v_1$ are in different connected components from each other in $G_{\F'}$.  Now $e_1=v_0v_1a$ and $e_2=v_0v_2b$, and $v_1e_1$ and $be_2$ are non-$G_{\F'}$-edges, so we may apply Lemma~\ref{lem:order7setup} (2) to $\F'$ --- swapping the roles of $v_1$ and $a$ --- to conclude that $v_1$ and $b$ lie in the same connected component of $G_{\F'}$. However, this is a contradiction, concluding Case 3.

Therefore, we conclude that $H$ is eulerian.
\end{proof}

\section{Proof of Theorem~\ref{thm:coveringinductionbase} for Orders $3\leq n\leq 6$}

In this section, we prove Theorem~\ref{thm:coveringinductionbase} for the orders that were missed by Theorem~\ref{thm:order7}.  The proof for orders 3, 4, and 5 --- found in Lemmas~\ref{lem:coveringorder3},~\ref{lem:coveringorder4}, and~\ref{lem:coveringorder5}, respectively, are relatively short due to the small number of possibilities in the hypergraph.\\

\begin{lem}\label{lem:coveringorder3} Let $H$ be a covering 3-hypergraph of order 3 and with at least two edges.  Then $H$ is eulerian.
\end{lem}

\begin{proof} Note that $H$ is a TS-(3,$\lambda$) with $\lambda\geq 2$ equal to the number of edges of $H$.  Then Corollary~\ref{cor:TS3} states that $H$ is eulerian.
\end{proof}

In the upcoming Lemma~\ref{lem:coveringorder4}, we employ a sort of counting strategy that has not been used since the proof of Theorem~\ref{thm:eulerianSTS}, where we showed that Steiner triple systems are eulerian.  We count the number of pairs of vertices that must be contained in edges, since the hypergraph is a covering 3-hypergraph.  We use this information to deduce that one pair must be in multiple edges, so we can form an $\F$-diminishing cycle with those edges.

In the case where the hypergraph has order 4, this method keeps the proof reasonably neat, though the number of possibilities is rather small.  However, this strategy returns in Lemma~\ref{lem:coveringorder6}, where the number of cases is more unwieldy, and it is there that it proves its worth.\\

\begin{lem}\label{lem:coveringorder4} Let $H$ be a covering 3-hypergraph of order 4.  Then $H$ is eulerian.
\end{lem}

\begin{proof} Let $G$ be the incidence graph of $H$.  Since $H$ necessarily has at least two edges, Corollary~\ref{cor:quasieuleriancovering} implies that $H$ is quasi-eulerian.  Let $\F$ be a minimum Euler family for $H$, with corresponding subgraph $\G$ of $G$.  Suppose $H$ is not eulerian: then Corollary~\ref{cor:3comps} implies that $\G$ has exactly two connected components, both non-trivial.

Let $G_1$ be one connected component of $\G$, and let $G_2$ be the other.  Define $S=\{\{x,y\}: x\in V(G_1), y\in V(G_2), x\text{ and }y\text{ are v-vertices}\}$.

Since $H$ has order 4, it must be the case that $G_1$ and $G_2$ contain two v-vertices each.  Since each edge of $H$ contains two vertices of one connected component and one vertex of the other, we have that each edge contains two elements from $S$.  Now, each connected component must have at least two e-vertices, so $H$ must have at least four edges.  This implies that the edges of $H$ contain eight elements (counting multiplicities) of $S$.  Since $|S|=4$, there exists an element $\{x,y\}$ of $S$ that is contained in two edges $e_1$ and $e_2$ of $H$, by Pigeonhole Principle.

Now, Corollary~\ref{cor:noncut} says that $G_1$ and $G_2$ each contain at least two v-vertices that are not cut vertices, so neither $x$ nor $y$ can be a cut vertex.

Let $C=xe_1ye_2x$.  Then $C$ is an $\F$-diminishing cycle by Corollary~\ref{cor:dimcycle}, contradicting Lemma~\ref{lem:interchange}.

We conclude that $H$ is eulerian.
\end{proof}

\begin{lem}\label{lem:coveringorder5} Let $H$ be a covering 3-hypergraph of order 5.  Then $H$ is eulerian.
\end{lem}

\begin{proof}Let $G$ be the incidence graph of $H$.  Since $H$ necessarily has at least two edges, Corollary~\ref{cor:quasieuleriancovering} implies that $H$ is quasi-eulerian.  Let $\F$ be a minimum Euler family for $H$, with corresponding subgraph $\G$ of $G$.  Suppose $H$ is not eulerian: then Corollary~\ref{cor:3comps} implies that $\G$ has exactly two connected components, and they must both be non-trivial by Corollary~\ref{cor:isolated}.

Without loss of generality, let $\{a,b\}$ be the set of v-vertices of one connected component $G_1$, and let $\{u,v,w\}$ be the set of v-vertices of the other connected component $G_2$.  Now, there must be at least two e-vertices $e_1$ and $e_2$ in $G_1$.

If $e_1$ and $e_2$ are parallel edges of $H$ --- say $e_1=abu=e_2$ --- then let $C=ae_1ue_2a$.  Observe that $C$ is an $\F$-interchanging cycle, and that $b, u, v,$ and $w$ are in the same connected component of $\G\Delta C$.  This implies that $\G\Delta C$ has just one non-trivial connected component, and so $C$ is an $\F$-diminishing cycle, contradicting Lemma~\ref{lem:interchange}.

Hence, without loss of generality, assume that $e_1=abu$ and $e_2=abv$.  Since $H$ is a covering 3-hypergraph, there exists an edge $e_3$ containing $a$ and $w$.  Suppose first that $e_3\in V(G_1)$, that is, that $e_3=abw$; then, as there must be an even number of e-vertices in $G_1$, there exists another edge $e_4=abx$, where $x\in\{u,v,w\}$.  In this case, let $C=aexe_4a$, where $e$ is chosen appropriately from among $e_1,e_2,$ and $e_3$.  As above, we see that $C$ is an $\F$-diminishing cycle, a contradiction with Lemma~\ref{lem:interchange}.

Hence $e_3\not\in V(G_1)$.  Then $u\in e_3$ or $v\in e_3$.  Suppose, without loss of generality, that $u\in e_3$.  Let $C=ae_1ue_3a$.  Then $C$ is an $\F$-interchanging cycle with one $\G$-edge in each connected component of $\G$, so Corollary~\ref{cor:simpledimcycle} says that $C$ is an $\F$-diminishing cycle.  However, this is a contradiction with Lemma~\ref{lem:interchange}.  

Therefore, we conclude that $H$ is eulerian.
\end{proof}

Now we approach the proof for covering 3-hypergraphs of order 6, an order that is too large to be proved simply, yet too small to be accomplished by the proof of Theorem~\ref{thm:order7}.  There are enough vertices in the hypergraph that we must now break the proof down into two broad cases, employing a different proof strategy for each.  Since we know the Euler family must have at most two components, those components can either have a 3-3 split or a 4-2 split on the vertices.

In the 3-3 case, we use the familiar counting argument to show that there exists an interchanging cycle of length 4.  If it is not a diminishing cycle, then we can modify it to find one that is.

In the 4-2 case, we have somewhat more difficulty.  We perform an interchange on a cycle that might not be diminishing, but can be used to scout information about the hypergraph.  In the end, this information will be vital in counting the pairs that are covered by the hypergraph, which leads to a contradiction.\\

\begin{lem}\label{lem:coveringorder6} Let $H$ be a covering 3-hypergraph of order 6.  Then $H$ is eulerian.
\end{lem}

\begin{proof} Let $G$ be the incidence graph of $H$, and suppose $H$ is not eulerian.  Corollary~\ref{cor:quasieuleriancovering} assures us that $H$ is quasi-eulerian, so let $\F$ be a minimum Euler family for $H$, and let $\G$ be the subgraph of $G$ corresponding to $\F$.

If $\G$ has an isolated vertex or any number of non-trivial connected components other than two, then $H$ is eulerian, by Corollary~\ref{cor:isolated} or~\ref{cor:3comps}, respectively.  Hence we may assume that $\G$ has exactly two connected components $G_1$ and $G_2$, and neither of them is trivial.

{\sc Case 1: $G_1$ and $G_2$ each have three v-vertices.} Let $a, b,$ and $c$ be the v-vertices of $G_1$, and let $x,y,$ and $z$ be the v-vertices of $G_2$.  Now, there are no $\G$-edges joining vertices of $G_1$ to vertices of $G_2$, and yet every pair in $S=\{\{u,v\}: u\in \{a,b,c\}, v\in \{x,y,z\}\}$ must be contained in some edge of $H$.

In fact, any edge of $H$ containing one pair in $S$ must contain exactly two such pairs, since the edge is a triple containing two v-vertices of one connected component of $\G$ and one v-vertex of the other.  Since there are nine pairs in $S$, and every edge of $H$ contains an even number of such pairs, at least one pair must be contained in two edges.

Without loss of generality, assume that $a,x\in e_1$ and $a,x\in e_2$, where $e_1,e_2\in E(H)$.  Then $C=ae_1xe_2a$ is an $\F$-interchanging cycle of $\G$, and since $\F$ is minimum, we have that $C$ is not an $\F$-diminishing cycle by Lemma~\ref{lem:interchange}.

If $G_1\setminus E(C)$ and $G_2\setminus E(C)$ each have just one non-trivial connected component, $C$ is an $\F$-diminishing cycle by Lemma~\ref{lem:dimcycle}, a contradiction.  Hence, without loss of generality, we assume that $G_1\setminus E(C)$ has two non-trivial connected components.  Since $G_1$ has no cut edge, we have that both $ae_1$ and $ae_2$ are $\G$-edges.  Consequently, both $xe_1$ and $xe_2$ are non-$\G$-edges, and so $G_2\setminus E(C)$ has just one non-trivial connected component.  If $\deg_{\G}(a)=2$, then $a$ is isolated in $G_1\setminus E(C)$.  Then $G_1\setminus E(C)$ must have just one connected component, for if not, then deleting $xe_2$ from the connected graph $G_1\setminus xe_1$ yields three connected components, in violation of Remark~\ref{remark:cutedge}.  Since $G_1\setminus E(C)$ and $G_2\setminus E(C)$ each have just one non-trivial connected component, again Lemma~\ref{lem:dimcycle} says that $C$ is an $\F$-diminishing cycle, a contradiction.  Hence $\deg_{\G}(a) \geq 4$ and, since deleting two edges incident with $a$ disconnects $G_1$ without isolating $a$, we deduce that $a$ is a cut vertex of $G_1$.

Since $a$ is a cut vertex of $G_1$, Corollary~\ref{cor:noncut} implies that $b$ and $c$ are not cut vertices of $G_1$.  Observe that, since each of $e_1$ and $e_2$ lies in $G_1$, each contains $b$ or $c$.  Without loss of generality, assume $b\in e_1$.  If we also have $b\in e_2$, then let $C=be_1xe_2b$.  Observe that $C$ is an $\F$-interchanging cycle, and since $b$ is not a cut vertex of $G_1$, we additionally have that $G_1\setminus E(C)$ (as well as $G_2\setminus E(C)$) is connected.  Then Lemma~\ref{lem:dimcycle} implies that $C$ is an $\F$-diminishing cycle, a contradiction with Lemma~\ref{lem:interchange} since $\F$ is minimum.  Hence $c\in e_2$.

Suppose $P$ is a 2-path from $c$ to $b$.  Then $ae_2cPbe_1a$ is a cycle of $G_1$ containing $a,b,$ and $c$.  Then Corollary~\ref{cor:noncut} implies that $a$ is not a cut vertex of $G_1$, contradicting our assumption on $a$.  Hence no 2-path from $c$ to $b$ exists in $\G$.

However, since $b$ and $c$ each have degree at least 2 in $\G$, they must each be connected to $a$ by a 2-path that does not traverse $e_1$ or $e_2$.  But then $G_1\setminus E(C)$ is connected, a contradiction.

We conclude that $G_1$ and $G_2$ cannot each have three v-vertices. 

{\sc Case 2: $G_1$ has two v-vertices and $G_2$ has four, without loss of generality.} Let $a$ and $b$ be the v-vertices of $G_1$, and let $w,x,y,$ and $z$ be the v-vertices of $G_2$.  Note that neither $a$ nor $b$ is a cut vertex of $G_1$, by Corollary~\ref{cor:noncut}.

There must exist an e-vertex $e_1\in V(G_1)$, adjacent to both $a$ and $b$ in $G_1$.  It must also be joined to another v-vertex --- say $w$ --- via a non-$\G$-edge.

Suppose there exists an edge $e$, distinct from $e_1$, containing $w$ and $a$.  Then let $C=ae_1wea$, and observe that $C$ is an $\F$-interchanging cycle. Then $G_2\setminus E(C)$ is connected, and either $G_1\setminus E(C)$ is connected or has $a$ as an isolated vertex, since $a$ is not a cut vertex of $G_1$, so $G_1\setminus E(C)$ has just one non-trivial connected component.  Lemma~\ref{lem:dimcycle} states that $C$ is in fact an $\F$-diminishing cycle, contradicting Lemma~\ref{lem:interchange}.  

Therefore, no edge containing $w$ also contains $a$, except for $e_1$. ($\ast$)

An analogous statement holds for $w$ and $b$.

Now, since $H$ is a covering 3-hypergraph, there must be edges of $H$ containing $w$ and the other v-vertices of $G_2$.  By ($\ast$), any such edge must contain three v-vertices of $G_2$.  Without loss of generality, let $e_2=wxy$ be such an edge of $H$.  Exactly one edge of $G$ incident with $e_2$ is a non-$\G$-edge.  We may assume, without loss of generality, that this non-$\G$-edge is either $we_2$ or $ye_2$. (Assuming $xe_2$ is the non-$\G$-edge would be equivalent to assuming $ye_2$ is.)

Let $e_3$ be an edge of $H$ containing $y$ and $a$.  We have that $e_3$ is distinct from both $e_1$ and $e_2$ because neither of those contains both $y$ and $a$.

Let $C=ae_1we_2ye_3a$, and observe that $C$ is an $\F$-interchanging cycle.  

Suppose $ae_3$ is a $\G$-edge.  Then $G_2\setminus E(C)$ is connected as it is $G_2$ with one edge ($we_2$ or $ye_2$) deleted.  Since $a$ is not a cut vertex of $G_1$, we also have that $G_1\setminus E(C)=G_1\setminus\{ae_1,ae_3\}$ is either connected or has a single non-trivial connected component with vertex set $V(G_1)\setminus\{a\}$.  In either case, Lemma~\ref{lem:dimcycle} implies that $C$ is an $\F$-diminishing cycle, contradicting Lemma~\ref{lem:interchange}.  Hence $ae_3$ is a non-$\G$-edge, and so $ye_3$ must be a $\G$-edge.

Since $\F$ is minimum, Lemma~\ref{lem:interchange} implies that $C$ cannot be an $\F$-diminishing cycle.  Let $\F'$ be an Euler family corresponding to $\G\Delta C$, so that we may denote $\G\Delta C=G_{\F'}$.  Note that we must have $|\F'|=2$, for if $|\F'|\geq 3$, then Corollary~\ref{cor:3comps} implies that $H$ is eulerian, a contradiction.

We now split into two cases, depending on which of $we_2$ and $ye_2$ is a $\G$-edge.  Since we assumed earlier that one of them is a non-$\G$-edge, exactly one of them must be a $\G$-edge.

{\sc Case A: $ye_2$ is a $\G$-edge and $we_2$ is a non-$\G$-edge.}  Since $G_{\F'}$ has two connected components, let $G_1'$ be the connected component of $G_{\F'}$ containing $a$, and let $G_2'$ be the other connected component.  Since $a$ and $b$ are the sole v-vertices of $G_1$ and they each have degree at least 2 in $G_{\F'}$, there must be another edge $e_4\in E(H)$ containing them both, where $ae_4$ and $be_4$ are both $\G$- (and hence $G_{\F'}$-) edges.  Then $b\in V(G_1')$.  We also have $w\in V(G_1')$ since $we_1b$ is a path in $G_{\F'}$, and $x\in V(G_1')$ since $xe_2w$ is a path in $G_{\F'}$.  This leaves $y$ and $z$ in $V(G_2')$, since $G_2'$ must contain two v-vertices.

Now, we have that $a,y\in e_3$, but do not yet know the third vertex of $e_3$.  However, this third vertex corresponds to a v-vertex that is adjacent to $e_3$ in both $\G$ and $G_{\F'}$.  It must be a v-vertex that is in both $G_2$ (since $ye_3$ is a $\G$-edge) and $G_1'$ (since $ae_3$ is a $G_{\F'}$-edge): we must have $x\in e_3$.  (By ($\ast$), it cannot be $w$.)

In $G_1'$, we now have a cycle $ae_4be_1we_2xe_3a$ traversing four v-vertices.  Corollary~\ref{cor:noncut} asserts that none of the four v-vertices of $G_1'$ are cut vertices of $G_1'$; in particular, we have that $x$ is not a cut vertex of $G_1'$.  Observe that $C'=xe_3ye_2x$ is an $\F'$-interchanging cycle in which $xe_3$ and $xe_2$ are $G_{\F'}$-edges.  Then $G_1'\setminus E(C')$ has at most one non-trivial connected component, since all $G_{\F'}$-edges of $C'$ are incident to $x$, which is not a cut vertex in $G_1'$.  We also have that $G_2'\setminus E(C') = G_2'$ is connected, so Lemma~\ref{lem:dimcycle} implies that $G_{\F'}\Delta C'$ has just one non-trivial connected component; hence $C'$ is $\F'$-diminishing.  Since $\F'$ is minimum, this contradicts Lemma~\ref{lem:interchange}.

{\sc Case B: $we_2$ is a $\G$-edge and $ye_2$ is a non-$\G$-edge.}  Now, since $G_{\F'}$ is disconnected, we have $a,b,$ and $w$ in one connected component of $G_{\F'}$, called $G_1'$; and $x$ and $y$ in another connected component, called $G_2'$.  If $z$ is in $G_2'$ as well, then we have two connected components with three v-vertices each, a contradiction with ($\ast$).  Hence $z$ is in $G_1'$.

First, consider what vertices are contained in $e_3\in E(H)$: we already know it contains $a$ and $y$.  Since the e-vertex $e_3$ is adjacent to $y$ in $G_2$ and adjacent to $a$ in $G_1'$, it must be adjacent to another v-vertex in both $G_2$ and $G_1'$.  The only candidates are $w$ and $z$, but it cannot be $w$ because $e_3$ also contains $a$ and this would lead to a contradiction with ($\ast$).  Therefore, we conclude that $z\in e_3$.

Since $x$ and $y$ are the only v-vertices in a connected component of $G_{\F'}$, there must be two 2-paths from $x$ to $y$ in $G_2'$, of which one exists in $\G$; call this path $P_1$.

Since $\deg_{\G}(z)\geq 2$ and $\deg_{G_{\F'}}(z)\geq 2$ and we only know of edge $ze_3$, vertex $z$ must be adjacent in $G_2$ and in $G_1'$ to another e-vertex.  This e-vertex cannot be any of $e_1,e_2,$ and $e_3$, so there must be one in both $G_2$ and $G_1'$; call it $e_5$.  Now, since $e_5$ is not traversed by $C$, it must have the same neighbours in both $G_2$ and $G_1'$, hence it must be adjacent to $w$ in both.  So $ze_5w$ is a $zw$-path of length 2 in $\G$.

Finally, we already know that $we_2x$ and $ye_3z$ are 2-paths in $\G$.  Hence $we_2xP_1ye_3ze_5w$ is a cycle of $\G$ containing all the v-vertices of $G_2$.  Then Corollary~\ref{cor:noncut} implies that $G_2$ has no v-vertices that are cut vertices.  It should be clear, again by Corollary~\ref{cor:noncut}, that neither $a$ nor $b$ are cut vertices of $G_1$, since they are the only v-vertices in $G_1$.

In order to derive a contradiction, we consider the set of pairs of v-vertices in opposite connected components of $\G$.  Let $S=\{\{u,v\}: u\in\{a,b\}, v\in\{w,x,y,z\}\}$ be this set of all such pairs.  Since $H$ is a covering 3-hypergraph, every element of $S$ must be contained in at least one edge.  We already know $e_1 = abw, e_2 = wxy, e_3 = ayz,$ and also that $w,z\in e_5$, so we will be able to deduce what pairs are contained in edges of $H$ and which have yet to be covered (by edges that we know about).

In terms of what pairs of $S$ are covered by edges we know about, we have the following:

\begin{itemize}
\item $e_1$ contains the pairs $\{a,w\}$ and $\{b,w\}$;
\item As in Case A, there exists an edge $e_4$, distinct from $e_1$, containing $a$ and $b$ in $\G$, so $e_4$ contains the pairs $\{a,v\}$ and $\{b,v\}$ for some $v\in\{w,x,y,z\}$; and
\item $e_3$ contains the pairs $\{a,y\}$ and $\{a,z\}$.
\end{itemize}

Suppose there exists a pair $\{s,t\}\in S$ contained in two edges $e,f\in E(H)$.  Then $C=setfs$ is an $\F$-interchanging cycle traversing one v-vertex from each connected component, and neither $s$ nor $t$ is a cut vertex of $\G$ or isolated in $\G$.  Then Corollary~\ref{cor:dimcycle} implies that $H$ is eulerian, a contradiction.

This implies that no pair can be contained in two edges.  We see, therefore, that $v$ cannot be equal to $w,y,$ or $z$, so it must be $x$.  Then the only pairs left to cover are $\{b,y\}$ and $\{b,z\}$, implying that there exists an edge $e=byz.$  Since $b\in V(G_1)$ and $y,z\in V(G_2)$, we see that $e\in V(G_2)$, so $ye$ and $ze$ are $\G$-edges.  Since $e\not\in V(C)$, we must have that $ye$ and $ze$ are edges in $G_{\F'}$ as well, contradicting the fact that $y$ and $z$ lie in different connected components of $G_{\F'}$.

In all cases, we have obtained a contradiction, so we conclude that $H$ is eulerian, as required.
\end{proof}

\section{Summary of Main Results}

\begin{thm}\label{thm:covering} Let $H$ be a covering 3-hypergraph of order at least 3 and with at least two edges.  Then $H$ is eulerian.
\end{thm}

\begin{proof} Let $n$ be the order of $H$.  If $n=3,4,5,$ or 6, then $H$ is eulerian by Lemma~\ref{lem:coveringorder3},~\ref{lem:coveringorder4},~\ref{lem:coveringorder5}, or~\ref{lem:coveringorder6}, respectively.

If $n\geq 7$, then $H$ is eulerian by Theorem~\ref{thm:order7}.

Therefore, $H$ is eulerian.
\end{proof}

Having settled the matter for covering 3-hypergraphs, we now turn our attention to covering $k$-hypergraphs, for any $k\geq 3$.  We use induction to prove Theorem~\ref{thm:coveringinduction}, for which the induction step is quite simple.  Theorem~\ref{thm:covering} serves as the induction basis necessary to complete this proof.

{\bf Theorem~\ref{thm:coveringinduction}.} Let $k\geq 3$, and let $H$ be a covering $k$-hypergraph.  Then $H$ is eulerian if and only if $H$ has at least two edges.

\begin{proof}
We may assume that $H$ is non-empty, for an empty hypergraph cannot be a covering $k$-hypergraph.

It should be clear that a hypergraph with only one edge cannot admit a closed walk, so it is not eulerian.  We need only prove sufficiency.

We prove this using induction on $k$.  When $k=3$, we have that $H$ is a covering 3-hypergraph with at least two edges, so Theorem~\ref{thm:covering} implies that $H$ is eulerian.  Suppose that, for some fixed $k\geq 3$, our result holds: that is, any covering $k$-hypergraph with at least two edges is eulerian.

Let $H=(V,E)$ be a covering $(k+1)$-hypergraph with $|E(H)|\geq 2.$  Fix some $v\in V$ and let $\mathcal{B}=\{e\setminus\{v\}:e\in E\text{ such that }v\in e\}$.  Let $H'=(V\setminus\{v\}, \mathcal{B})$.  Note that $H'$ is a $k$-uniform hypergraph.  It also has the property that every $(k-1)$-tuple of $V(H')$ lies in at least one edge: to see this, let $X\subset V(H')$ be a set of cardinality $k-1$.  Then, in $H$, the $k$-subset $X\cup \{v\}$ lies in some edge $e$.  Then there exists in $\mathcal{B}$ a corresponding edge $e\setminus \{v\}$, which contains $X$.

Now, let $E'=\{e\in E: v\not\in e\}$.  Obtain a set $\mathcal{C}$ of edges, each of cardinality $k$, by taking all the edges of $E'$ and removing an arbitrary vertex from each of them (the choice of vertex does not matter).  Let $H^*=(V\setminus\{v\},\mathcal{B}\cup\mathcal{C}),$ in which the edges of $H^*$ are obtained by the multiset union of $\mathcal{B}$ and $\mathcal{C}$.  Then $H^*$ is a $k$-uniform hypergraph such that every $(k-1)$-tuple of $V(H^*)$ lies in at least one edge: that is, we have that $H^*$ is a covering $k$-hypergraph.  Since $|E(H^*)|=|E(H)|\geq 2$, we may apply the induction hypothesis to get that $H^*$ is eulerian, so it admits an Euler tour $T$.  

Define a function $\varphi: E(H^*)\rightarrow E(H)$ by mapping $e$ to $e\cup \{v\}$ if $e$ originally comes from $\mathcal{B}$, or by mapping $e$ to its corresponding edge in $E'$ if $e$ originally comes from $\mathcal{C}$.  Then $\varphi$ is a bijection, and since $V(H^*)\subseteq V(H)$, we may apply Lemma~\ref{lem:truncatedhypergraph} to obtain an Euler tour of $H$ (by regarding $\{T\}$ as an Euler family of cardinality 1).

Therefore, the result holds by induction.
\end{proof}
\cleardoublepage

\chapter{Eulerian Properties of $\ell$-Covering \protect\hidemath-Hypergraphs}\label{chapter:quasi-eulerian}


\section{Introduction}
Following the results of Chapter~\ref{chapter:covering}, we would like to prove that all $\ell$-covering $k$-hypergraphs are eulerian.  Recall that an $\ell$-covering $k$-hypergraph is a $k$-uniform hypergraph in which every $\ell$-subset of vertices lie together in at least one edge.  We have already proven that when $k=\ell+1$, such hypergraphs are eulerian, so perhaps we can use similar techniques to extend these results to larger $k$ (relative to $\ell$).

Unfortunately, the first thing we would need to do is prove a result about $(k-2)$-covering $k$-hypergraphs, necessarily with $k\geq 4$.  But the interchanging cycles that we used to prove Theorem~\ref{thm:coveringinductionbase} are much more unwieldy if the hypergraph is not 3-uniform.  When we looked at the incidence graph of a 3-uniform hypergraph admitting an Euler family $\F$, we knew that every pair of edges incident with an e-vertex had at least one $G_{\F}$ edge.  This is not the case for $k$-uniform hypergraphs with $k\geq 4$, so it is much harder to find interchanging cycles.

Instead, we will have to be satisfied to prove that $\ell$-covering $k$-hypergraphs are quasi-eulerian using Lov\'{a}sz's $(g,f)$-factor Theorem.  We will be able to prove the following result.\\

\begin{thm} Let $H$ be an $\ell$-covering $k$-hypergraph for some $2\leq\ell<k$.  Then $H$ is quasi-eulerian if and only if it has at least two edges.
\end{thm}

We will prove this by induction on $k$ using a method similar to Theorem~\ref{thm:coveringinduction}.  Once again, the induction step will be relatively simple, but the basis of induction is the focus of a majority of this chapter.  Part of the basis of induction comes from Theorem~\ref{thm:coveringinduction}, and another part comes from the following result, which we will prove later in this chapter.\\

\begin{thm} Let $k\geq 4$, and let $H$ be a 2-covering $k$-hypergraph of order at least $k$ and size at least 2.  Then $H$ is quasi-eulerian.
\end{thm}

We have already presented Lov\'a{s}z's Theorem as Theorem~\ref{thm:lovasz2}, but we will present it again here.  The techniques that Bahmanian and \v{S}ajna used to modify Theorem~\ref{thm:lovasz2} form the starting point for our ideas, and we expand on them to suit our purposes for $\ell$-covering $k$-hypergraphs.\\

\begin{thm}\label{thm:lovasz}{\em (The $(g,f)$-factor Theorem, Lov\'{a}sz \cite{LL,AK})} Let $G=(V,E)$ be a graph and $f,g:V\rightarrow\mathds{N}$ be functions such that $g(x)\leq f(x)$ and $g(x)\equiv f(x)$ (mod 2) for all $x\in V$.  Then $G$ has a $(g,f)$-factor $F$ such that $\deg_F(x)\equiv f(x)$ (mod 2) for all $x\in V$ if and only if, for all disjoint $S,T\subseteq V$, we have
\begin{equation}\label{eqn:lovaszcondition}
\sum_{x\in S}f(x) + \sum_{x\in T}(\deg_G(x) - g(x)) - e_G(S,T) - q(S,T)\geq 0,
\end{equation}
where $e_G(S,T)$ denotes the number of edges with one end in $S$ and the other in $T$, and $q(S,T)$ is the number of connected components $C$ of $G-(S\cup T)$ such that $$\sum_{x\in V(C)}f(x) + e_G(V(C),T)\equiv 1\text{ (mod 2).}$$
\end{thm}

Lov\'{a}sz's Theorem gives us the necessary and sufficient conditions for existence of spanning subgraphs with specific degree requirements.  Since we will only be using it to find subgraphs corresponding to Euler families in bipartite graphs, we can strengthen the conditions to obtain a sufficient (but not necessary) condition instead.  These strengthened conditions, however, are much easier to work with, and are satisfiable when applied to 2-covering $k$-hypergraphs, as we shall see.

Recall that $c(G^*-X)$ denotes the number of connected components of $G^*-X$.\\

\begin{lem}\label{lem:relax}Fix $k\geq 4$.  Let $H$ be a $k$-hypergraph of order $n$ and size $m$, free of cut edges, and let $G$ be its incidence graph.  Define $r=(m+n)^2$, and obtain a graph $G^*$ from $G$ by appending $2r$ loops to every v-vertex.

Define $f:V(G^*)\rightarrow \mathds{Z}$ by 
\begin{equation*}
f(x) = 	\left\{
      			\begin{array}{l c l}
      				2 & : & x\text{ is an e-vertex} \\
      				2r & : & x\text{ is a v-vertex}.
			\end{array}
		\right.
\end{equation*}

If, for all $X\subseteq E(H)$ with $|X|\geq 2$, we have $|X|\geq 2\lfloor\frac{c(G^*-X)+3}{k}\rfloor$, then $G^*$ has an $(f,f)$-factor.  Furthermore, this $(f,f)$-factor can be used to obtain an Euler family for $H$, so $H$ is quasi-eulerian.
\end{lem}

\begin{proof} Let $G^*$ and $f$ be defined as in the statement of the lemma, and let $g=f$.  Assume that for all $X\subseteq E(H)$ satisfying $|X|\geq 2$, we have $|X|\geq 2\lfloor\frac{c(G^*-X)+3}{k}\rfloor$.

Let $S,T\subseteq V(G^*)$ be disjoint subsets.  We show that (\ref{eqn:lovaszcondition}) holds for $G^*, f, g, S,$ and $T$.

First we note that Condition~(\ref{eqn:lovaszcondition}) becomes 
\begin{equation}\label{eqn:newcondition}
\gamma(S,T)=\sum_{x\in S}f(x) + \sum_{x\in T}(\deg_{G^*}(x)-f(x)) - e_{G^*}(S,T)-q(S,T)\geq 0,
\end{equation}
where $q(S,T)$ is the number of connected components $C$ of $G^*-(S\cup T)$ such that $e_{G^*}(V(C), T)$ is odd.

Furthermore, observe that $e_{G^*}(S,T)\leq mn$, the number of edges in $K_{n,m}$; and $q(S,T)\leq m+n,$ the number of vertices in $G^*$.  We also have that $\deg_{G^*}(x) \geq f(x)$ for any $x\in V(G^*)$.

{\sc Case 1: $S$ contains a v-vertex $v$.}  Then
\begin{align*}
\gamma(S,T)&=\sum_{x\in S}f(x) + \sum_{x\in T}(\deg_{G^*}(x)-f(x)) - e_{G^*}(S,T)-q(S,T)\\
&\geq 2r + 0 - mn - m - n,
\end{align*}
since $\deg_{G^*}(x)\geq f(x)$ for all $x\in T$, and $e_{G^*}(S,T)\leq mn$ and $q(S,T)\leq m+n$.  Then $2r - mn - m - n\geq 0$ since $2r = 2(m+n)^2,$ so Condition (\ref{eqn:newcondition}) is satisfied.

{\sc Case 2: $T$ contains a v-vertex $v$.}  Then
\begin{align*}
\gamma(S,T)&=\sum_{x\in S}f(x) + \sum_{x\in T}(\deg_{G^*}(x)-f(x)) - e_{G^*}(S,T)-q(S,T)\\
&\geq \sum_{x\in S}f(x) + (\deg_{G^*}(v)-f(v)) - e_{G^*}(S,T)-q(S,T)\\
&\geq 0 + (4r - 2r) - mn - m - n\\
&\geq 2r - mn - m - n,
\end{align*}
since $f(x)\geq 0$ for all $x\in V(G^*)$, and $\deg_{G^*}(v) - f(v) \geq 4r - 2r$, and $e_{G^*}(S,T)\leq mn$ and $q(S,T)\leq m+n$.  Then $2r - mn - m - n\geq 0$ since $2r = 2(m+n)^2,$ so Condition (\ref{eqn:newcondition}) is satisfied.

{\sc Case 3: Neither $S$ nor $T$ contains v-vertices.}  Then $e_{G^*}(S,T) = 0$ since $S$ and $T$ are subsets of the set of e-vertices, which is an independent set.
\begin{itemize}
\item {\sc Case 3A: $T=\emptyset$.} We have that $e_{G^*}(V(C),T) = 0$ for all connected components $C$ of $G^*-(S\cup T)$, and $f$ takes only even values.  Then by the definition of $q(S, T)$, we have $q(S, T) = 0$.  Hence 
\begin{align*}
\gamma(S,T)=\sum_{x\in S}f(x)\geq 0,
\end{align*}
since $f$ is nonnegative.  Therefore, Condition (\ref{eqn:newcondition}) holds.

\item {\sc Case 3B: $S=\emptyset$ and $|T| = 1$.}  Note that $T$ contains a single e-vertex.  Since $H$ has no cut edges, Theorem~\ref{thm:cutvvertex} implies that $G$ has no cut e-vertices, so neither does $G^*$. Then $q(S,T) \leq 1$ since there is only one connected component of $G^* - (S\cup T)$.  Then
\begin{align*}
\gamma(S,T)&= \sum_{x\in S}f(x) + \sum_{x\in T}(\deg_{G^*}(x)-f(x)) - e_{G^*}(S,T)-q(S,T)\\
&\geq 0 + (k-2) - 0 - 1\\
&= k - 3\\
&\geq 0.
\end{align*}

\item {\sc Case 3C: $T\neq\emptyset$ and $|S\cup T|\geq 2$.}

Since $q(S,T)$ is the number of connected components $C$ of $G^*-(S\cup T)$ such that $e_{G^*}(V(C),T)$ is odd, it cannot be greater than either the number of connected components of $G^*-(S\cup T)$ or the number of edges in an edge cut $[T,\overline{T}]$ of $G^*$.  Since all the vertices of $T$ are e-vertices that have degree $k$ in $G^*$, we have that the latter quantity is at most $k|T|$. Hence $q(S,T)\leq \text{min}\{c(G^*-(S\cup T)), k|T|\}$.  Then
\begin{align*}
\gamma(S,T)&= \sum_{x\in S}f(x) + \sum_{x\in T}(\deg_{G^*}(x)-f(x)) - e_{G^*}(S,T)-q(S,T)\\
&=2|S| + (k-2)|T| - q(S,T)\\
&\geq 2|S| + 2|T| + (k-4)|T| - \text{min}\{c(G^*-(S\cup T)), k|T|\}
\end{align*}
\begin{equation}\label{eqn:3B}=2|S\cup T| + (k-4)|T| - \text{min}\{c(G^*-(S\cup T)), k|T|\}.
\end{equation}

Define $t = \lfloor\frac{c(G^*-(S\cup T)) + 3}{k}\rfloor$, and suppose first that $|T|>t$. Then $|T|\geq t + 1$.  Furthermore, we note that $kt-3\leq c(G^*-(S\cup T))\leq kt+k-4$ from the definition of $t$. In particular, since $k|T|\geq kt+k$, we have $\text{min}\{c(G^*-(S\cup T)), k|T|\}=c(G^*-(S\cup T))\leq kt+k-4$.  Then
\begin{align*}
\gamma(S,T)&\geq 2|S\cup T| + (k-4)(t + 1) - kt-k+4\\
&= 2|S\cup T| - 4t\\
&= 2|S\cup T| - 4\Big\lfloor\frac{c(G^*-(S\cup T)) + 3}{k}\Big\rfloor.
\end{align*}

Letting $X=S\cup T$, we have that $|X|\geq 2$.  By the assumption of the lemma that $|X| \geq 2\lfloor\frac{c(G^*-X) + 3}{k}\rfloor,$ we see that $\gamma(S,T)\geq 0$.  Therefore, in the case where $|T|>t$, Condition (\ref{eqn:newcondition}) holds.

Hence we may assume that $|T|\leq t$.  Then from $\gamma(S,T)\geq 2|S\cup T| + (k-4)|T| - \text{min}\{c(G^*-(S\cup T)), k|T|\}$ in (\ref{eqn:3B}), we have
\begin{align*}
\gamma(S,T)&\geq 2|S\cup T| + (k-4)|T| - k|T|\\
&=2|S\cup T| - 4|T|\\
&\geq 2|S\cup T| - 4t\\
&= 2|S\cup T| - 4\Big\lfloor\frac{c(G^*-(S\cup T)) + 3}{k}\Big\rfloor,
\end{align*}
and again we see that our assumption on $X=S\cup T$ suffices to show that $\gamma(S,T)\geq 0$.

Hence Condition (\ref{eqn:newcondition}) holds in this case.
\end{itemize}

Since Condition (\ref{eqn:newcondition}) holds for all disjoint $S,T\subseteq V(G^*)$, we conclude that $G^*$ has an $(f,f)$-factor $F$ by Theorem~\ref{thm:lovasz}.

If we delete loops of $F$, we obtain a spanning subgraph $F'$ of $G$ in which all v-vertices have even degree and all e-vertices have degree 2.  This corresponds to an Euler family for $H$ by Theorem~\ref{thm:incidencegraph}, and so $H$ is quasi-eulerian.
\end{proof}

Even with the strengthened condition of Lemma~\ref{lem:relax}, we will be able to prove that all 2-covering $k$-hypergraphs satisfy it, for all $k\geq 4$.

\section{Technical Lemmas}

In preparation to prove that all 2-covering $k$-hypergraphs are quasi-eulerian, we need a few technical lemmas.  Most of them are used to clean up certain cases in the main proof.\\

\begin{lem}\label{lem:2-connected} Let $k\geq 4$, and let $H$ be a 2-covering $k$-hypergraph of order at least $k+1$.  Then $H$ has no cut edges.
\end{lem}

\begin{proof} Fix distinct $u,v\in V(H)$.  We will show that $u$ and $v$ are in a cycle of $H$.

If $u$ and $v$ are contained in at least two edges together, then $u$ and $v$ are in a 2-cycle that traverses two of these edges.  

Now, assume that $u$ and $v$ are contained in exactly one edge together, say $e$.  Let $w\in V(H)\setminus e$.  Now, since $H$ is a 2-covering $k$-hypergraph, there exist edges $e_1, e_2\in E(H)$ such that $u$ and $w$ lie in $e_1$ together, and $v$ and $w$ lie in $e_2$ together.  We have $e\not\in\{e_1,e_2\}$ because $e$ does not contain $w$, but $e_1$ and $e_2$ do.  Furthermore, if $e_1=e_2$, then $e_1$ contains both $u$ and $v$, so there are at least two distinct edges containing $u$ and $v$, a contradiction.

Now, since $e, e_1,$ and $e_2$ are three distinct edges, we have that $ueve_2we_1u$ is a cycle of $H$ containing both $u$ and $v$.  We conclude that there are two edge-disjoint $uv$-paths in $H$ for any pair $u,v\in V(H)$.  Hence $H$ has no cut edges.
\end{proof}

\begin{lem}\label{lem:numedges} Let $k\geq 4$, and let $H$ be a 2-covering $k$-hypergraph of order $n\geq k+1$.  Suppose one of the following holds:
\begin{itemize}
\item $n>\frac{3k}{2};$ or
\item $k\geq 7.$
\end{itemize}
Then $H$ has at least $2\lfloor\frac{n+3}{k}\rfloor$ edges.
\end{lem}

\begin{proof} Since there are $\binom{n}{2}$ pairs of vertices to cover, and each edge covers $\binom{k}{2}$ pairs, we necessarily have $|E(H)|\geq \frac{n(n-1)}{k(k-1)}$.

{\sc Case 1: $n\leq 2k-4$.} It suffices to show that $|E(H)|\geq 2$ because now $2\lfloor\frac{n+3}{k}\rfloor\leq 2$. Clearly $H$ has at least two edges, since it has order at least $k+1$.  Therefore, in this case we have $|E(H)|\geq 2\lfloor\frac{n+3}{k}\rfloor$.

{\sc Case 2: $n\geq 3k-3$.}  Then
\begin{align*}
|E(H)|\geq \frac{n(n-1)}{k(k-1)} &\geq \frac{(3k-3)(n-1)}{k(k-1)}\\
&= \frac{3(n-1)}{k}\\
&= \frac{2n + n - 3}{k}\\
&\geq \frac{2n + 3k - 6}{k}\\
&\geq \frac{2n + 6}{k}\text{ since $k\geq 4$}\\
&\geq 2\Big\lfloor\frac{n+3}{k}\Big\rfloor.
\end{align*}

{\sc Case 3: $2k-3\leq n\leq 3k-4$.}  It suffices to show that $|E(H)|\geq 4$ because we now have $2\lfloor\frac{n+3}{k}\rfloor\leq 4$.  Suppose $|E(H)|\leq 3$.  Since $n\geq k+1$ by assumption, we cannot have any vertices of degree 1: such a vertex would only have $k-1$ neighbours via its incident edge, but it needs to be neighbours with at least $k$ others. Then the Handshake Lemma implies that
\begin{align*}
2n\leq \sum_{v\in V(H)}\deg(v) &= k|E(H)|\leq 3k,\\
\text{hence }2n&\leq 3k,\\
\text{so }n&\leq \frac{3k}{2}.
\end{align*}

Now, this brooks a contradiction if $n>\frac{3k}{2}$, so we must have $k\geq 7$.  Then from $2k-3\leq n\leq \frac{3k}{2}$ we get $k\leq 6$, a contradiction.

Therefore, in all cases we have $|E(H)|\geq 2\lfloor\frac{n+3}{k}\rfloor.$
\end{proof}

\begin{lem}\label{lem:2-intersection} Let $H$ be a hypergraph with $|E(H)|\geq 2$, and suppose $E(H)$ satisfies the following:
\begin{itemize}
\item For all $e,f\in E(H)$, we have $|e\cap f|\geq 2$; and
\item There exist distinct $e, f\in E(H)$ such that $|e\cap f|\geq 3$.
\end{itemize}

Then $H$ is eulerian.
\end{lem}

\begin{proof}
Let $E(H)=\{e_1,\dotso , e_m\}$ and let $e_1,e_m$ be a pair of distinct edges such that $|e_1\cap e_m|\geq 3$.

Fix some $v_1\in e_2\cap e_1$. For $i=2,\dotso , m-1$, let $v_i$ be a vertex in $(e_{i+1}\cap e_i)\setminus \{v_{i-1}\}$.  Since $|e_{i+1}\cap e_i|\geq 2$, there is always at least one vertex to choose from.  Let $W=v_1e_2\dotso e_{m-1}v_{m-1}$ be the walk determined by these choices of $v_1, \dotso , v_{m-1}$.

To extend $W$ to an Euler tour, choose $v_0$ from among the vertices in $(e_1\cap e_m)\setminus\{v_1, v_{m-1}\}$.  By our assumption, this set is non-empty, so this choice is well defined.  Then $v_0e_1v_1\dotso v_{m-1}e_mv_0$ is a closed strict trail that traverses every edge of $E(H)$, so it is an Euler tour of $H$. Hence $H$ is eulerian.
\end{proof}

\begin{cor}\label{cor:smallcases}Let $H$ be a 2-covering $k$-hypergraph of order $n$.  If $(k,n)=(4,6)$ or $n\leq 2k-3$, then $H$ is eulerian.
\end{cor}

\begin{proof} First, suppose $(k,n)=(4,6)$. For all $e,f\in E(H)$, we have $|e\cap f|\geq 2$.  If there exists a pair of distinct edges $e,f\in E(H)$ such that $|e\cap f|\geq 3$, then Lemma~\ref{lem:2-intersection} implies that $H$ is eulerian.  Hence assume $|e\cap f|=2$ for all $e,f\in E(H)$.

Let $V(H)=\{v_1,\dotso , v_6\}$ and $E(H)=\{e_1,\dotso , e_m\}$, where $m$ is the size of $H$.  Since $H$ is a 2-covering 4-hypergraph, for every $v_i\in V(H)$, we necessarily have $\deg(v_i)\geq \frac{n-1}{k-1} = \frac{5}{3}$, so every vertex of $H$ has degree at least 2.  Without loss of generality, let $e_1=v_1v_2v_3v_4$ and $e_2=v_1v_2v_5v_6$ be two edges of $H$.  Since $(V(H), \{e_1,e_2\})$ is not a 2-covering 4-hypergraph, there must be at least one more edge in $E(H)$.  Such an edge $e_3$ must contain exactly two vertices of $e_1$ and exactly two vertices of $e_2$.  If $e_3$ contains any vertices in $e_1\cap e_2$, then it will necessarily contain at least three vertices from either $e_1$ or $e_2$, so we must have $e_3=v_3v_4v_5v_6$.

Observe that $e_3$ is the unique 4-subset that satisfies $|e_3\cap e_1|=|e_3\cap e_2|=2$.  Hence $E(H)=\{e_1,e_2,e_3\}$ and $W=v_3e_1v_2e_2v_5e_3v_3$ is an Euler tour of $H$, so $H$ is eulerian.

Now suppose $n\leq 2k-3$.  Then every pair of edges $e,f\in E(H)$ satisfies $|e\cap f|\geq 3$.  Then $H$ is eulerian by Lemma~\ref{lem:2-intersection}.
\end{proof}

\begin{lem}\label{lem:order4} Let $k\geq 3$, and let $H$ be a 2-covering $k$-hypergraph of order $k$.  Then $H$ is eulerian if and only if $H$ has at least two edges.
\end{lem}

\begin{proof} If $H$ has just one edge, then it is not eulerian as no closed trail can be formed.

Note that our assumptions indicate that each edge $e\in E(H)$ satisfies $e=V(H)$. In particular, we have $e\cap f = V(H)$ for all $e,f\in E(H)$.

Then $H$ satisfies the conditions of Lemma~\ref{lem:2-intersection}, so $H$ is eulerian.
\end{proof}

We now present an optimization problem that will be used to establish an upper bound on the number of edges in a certain graph.  Since we have not seen optimization problems elsewhere in this thesis, we will need to review some basic definitions.\\

\begin{defn}{\rm {\bf (Optimization)} Let {\bf x} = $(x_1,x_2,\dotso , x_q)\in\mathds{R}^q$ be a $q$-tuple of indeterminates.  An {\em optimization problem} (P) consists of a real-valued function $f$ in {\bf x} that is to be maximized or minimized, called the {\em objective function}, along with a set of constraints on {\bf x}.

If ${\bf x^*}\in\mathds{R}^q$ satisfies the given constraints, then ${\bf x^*}$ is called a {\em feasible solution} for (P).  The set $\{{\bf y}\in\mathds{R}^q: {\bf y}\text{ is a feasible solution for (P)}\}$ is called the {\em feasible set} of (P).  

If $x^*$ is a feasible solution, and $f({\bf x^*})\geq f({\bf y})$ (or $f({\bf x^*})\leq f({\bf y})$, if (P) is a minimization problem) for all feasible solutions ${\bf y}$, then ${\bf x^*}$ is an {\em optimal solution} for (P).
}
\end{defn}

\begin{lem}\label{lem:notturan} Let $n,k,q\in\mathds{Z}^+$ be such that $n\geq qk$. Define the following optimization problem:
\begin{align*}
\text{{\em maximize} \hspace{5px}}& f(x_1,\dotso ,x_q) = \binom{x_1}{2} + \cdots + \binom{x_q}{2}&(P)\\
\text{{\em subject to} \hspace{5px}}& x_1 + \cdots + x_q = n,\\
&x_1,\dotso , x_q\geq k,\\
&x_1,\dotso , x_q\in\mathds{Z}^+.
\end{align*}

Then $x_1=\cdots = x_{q-1}=k, x_q=n-k(q-1)$ is an optimal solution to (P).
\end{lem}

\begin{proof} First of all, observe that $x_1=x_2=\cdots = x_{q-1}=k, x_q=n-k(q-1)$ is a feasible solution to (P).  Furthermore, there exists an optimal solution for (P) because the feasible set of (P) is finite.

Let ${\bf x^*} =(x_1^*,\dotso , x_q^*)$ be an optimal solution for (P) in which $x_1\leq x_2\leq\dotso\leq x_q$.  Suppose, for the sake of obtaining a contradiction, that $x_q < n-k(q-1)$.

Then there exists $1\leq i\leq q-1$ such that $x_i>k$.  Let $i$ be the smallest index with this property.

We claim that $f(x_1,\dotso, x_{i-1}, x_i - 1,x_{i+1},\dotso , x_{q-1},x_q+1)>f({\bf x^*})$.  Indeed, we have
\begin{align*}
&f(x_1,\dotso, x_{i-1}, x_i - 1,x_{i+1},\dotso , x_{q-1},x_q+1)\\
=&\sum_{\substack{j=1\\j\neq i}}^{q-1}\binom{x_j}{2} + \binom{x_i - 1}{2} + \binom{x_q + 1}{2}\\
=&\sum_{\substack{j=1\\j\neq i}}^{q-1}\binom{x_j}{2} + \frac{(x_i - 1)(x_i - 2)}{2} + \frac{(x_q+1)x_q}{2}\\
=&\sum_{\substack{j=1\\j\neq i}}^{q-1}\binom{x_j}{2} + \frac{x_i(x_i - 1)}{2} - \frac{2(x_i-1)}{2} + \frac{x_q(x_q - 1)}{2} + \frac{2x_q}{2}\\
=&\sum_{j=1}^q\binom{x_j}{2} - (x_i - 1) + x_q\\
=&f({\bf x^*}) -x_i + 1 + x_q\\
>&f({\bf x^*})\text{ since $x_q\geq x_i$.}
\end{align*}

Observe that we have $x_1\leq \dotso\leq x_{i-1}\leq x_i - 1\leq x_{i+1}\leq\dotso\leq x_{q-1}\leq x_q + 1$ because we know $x_i - 1 \geq k$ and $x_{i-1} \leq k$ by our choice of $i$.  This contradicts the optimality of ${\bf x^*}$ since $(x_1,\dotso, x_{i-1}, x_i - 1,x_{i+1},\dotso , x_{q-1},x_q+1)$ is a feasible solution to (P) whose components are in nondecreasing order, and its objective value is higher than $f({\bf x^*})$.

Therefore, any optimal solution satisfying $x_1\leq\dotso\leq x_q$ must also satisfy $x_q\geq n-k(q-1)$.  However, if $x_q> n-k(q-1)$, then we have $x_1+\dotso + x_{q-1} + x_q > k + \dotso + k + n - k(q-1) = n$, a contradiction.  Therefore, we have $x_q = n-k(q-1)$ in such an optimal solution, and so we must have $x_1=\dotso = x_{q-1} = k$.

We conclude that $(x_1,\dotso , x_{q-1},x_q) = (k,\dotso ,k,n-k(q-1))$ is an optimal solution for (P).
\end{proof}

\section{Main Results}

We present our two main results of the chapter in this section.  First, we show in Theorem~\ref{thm:2coveringkhypergraphs} that all 2-covering $k$-hypergraphs with at least two edges are quasi-eulerian.  We then use this result in Theorem~\ref{thm:lcoveringkhypergraphs} as the basis of induction to show that all $\ell$-covering $k$-hypergraphs with at least two edges are quasi-eulerian.\\

\begin{thm}\label{thm:2coveringkhypergraphs} Let $k\geq 4$, and let $H$ be a 2-covering $k$-hypergraph of order $n$ with at least two edges.  Then $H$ is quasi-eulerian.
\end{thm}

\begin{proof} First, assume that $n=k$. Since $H$ has at least two edges, Lemma~\ref{lem:order4} implies that $H$ is eulerian, so it is quasi-eulerian.  Hence we may assume $n>k$ and, consequently, $n\geq 5$.

Suppose $n\leq 2k-3$. Then Corollary~\ref{cor:smallcases} implies that $H$ is eulerian, so $H$ is quasi-eulerian.  Hence we may assume $n\geq 2k-2.$

Now suppose $k\leq 6$ and $n\leq\frac{3k}{2}$.  Since $n\geq 2k-2$, this assumption implies $(k,n)=(4,6)$.  In this case, Corollary~\ref{cor:smallcases} implies that $H$ is eulerian, so $H$ is quasi-eulerian.  Hence we have $k\geq 7$ or $n>\frac{3k}{2}$.  Lemma~\ref{lem:numedges} then implies that $|E(H)|\geq 2\big\lfloor\frac{n+3}{k}\big\rfloor$.

We will show that $H$ satisfies the condition of Lemma~\ref{lem:relax}.

Let $m$ be the size of $H$, and define $r=(m+n)^2$.  Let $G$ be the incidence graph of $H$, and let $G^*$ be the graph obtained from $G$ by adjoining $2r$ loops to every v-vertex. 

Define $f,g:V(G^*)\rightarrow \mathds{Z}$ by 
\begin{equation*}
f(x) = g(x) =	\left\{
      				\begin{array}{l c l}
      					2 & \text{if} & x\text{ is an e-vertex} \\
      					2r & \text{if} & x\text{ is a v-vertex}.
				\end{array}
			\right.
\end{equation*}

Fix any $X\subseteq E(H)$ with $|X|\geq 2$, and denote $q=c(G^*-X)$.  Now, suppose 
\begin{equation}\label{assumption:X}
|X|<2\Big\lfloor\frac{q+3}{k}\Big\rfloor.
\end{equation} Note that any connected component of $G^*-X$ must contain at least one v-vertex, so $q\leq n$.  Furthermore, any non-trivial connected component of $G^*-X$ has at least $k$ v-vertices.

Observe that, if $q\leq 2k-4$, then from our assumption that $|X|<2\lfloor\frac{q+3}{k}\rfloor$ we get that $|X| < 2$, so $|X|=1$.  This is a contradiction with the assumption that $|X|\geq 2$, hence we may assume that 
\begin{equation}\label{assumption:q}
q\geq 2k-3.
\end{equation}

Let $\ell$ denote the number of isolated v-vertices of $G^*-X$.

{\sc Case 1: $\ell\geq 1$.}  First, suppose $\ell=n$.  Then $X=E(H), q = n$, and $|X|=|E(H)|\geq 2\lfloor\frac{n+3}{k}\rfloor = 2\lfloor\frac{q + 3}{2}\rfloor$. Then we have $|X|\geq 2\lfloor\frac{c(G^*-X)+3}{k}\rfloor$, contradicting our assumption on $X$.  Hence we may assume $\ell<n$.

 In fact, since any non-trivial connected component of $G^*-X$ has at least $k$ v-vertices, and there must be $q-\ell$ non-trivial connected components, so we have 
 \begin{equation}\label{eqn:n0}
 n\geq\ell + k(q-\ell).
 \end{equation}  Rearranging, we can write
 \begin{equation}\label{eqn:n1}
 \frac{n-\ell}{k}\geq q - \ell.
 \end{equation}We may rearrange this to get  
 \begin{equation}\label{eqn:q}
 q\leq\ell +\frac{n-\ell}{k},
 \end{equation} which will be used later.  Since $q>\ell$, we may also infer from~(\ref{eqn:n0}) that 
 \begin{equation}\label{eqn:n2}
 n\geq\ell + k.
 \end{equation}

Let $U=\{v_1,\dotso , v_{\ell}\}$ denote the set of isolated v-vertices in $G^*-X$.  Let $S=\{\{u,v\}: u\in U, v\in V(H)\setminus U\}$.  Then $|S|=\ell(n-\ell)$.

For $e\in E(H)$, if the edge $e$ contains $b$ vertices in $U$, then it contains $k-b$ vertices in $V(H)\setminus U$, and covers $b(k-b)$ pairs of $S$.  Then $b(k-b)$ is maximized when $b$ and $k-b$ are as close to equal as possible, so we may conclude that every edge of $E(H)$, and hence every edge of $X$, covers up to $\lfloor\frac{k}{2}\rfloor\lceil\frac{k}{2}\rceil\leq\frac{k^2}{4}$ pairs of $S$.

By dividing the number of pairs of $S$ that $X$ covers ({\em i.e.} $\ell(n-\ell)$) by the maximum number that each edge of $X$ can cover, we get $|X|\geq\frac{4\ell(n-\ell)}{k^2}$.

Our assumption (\ref{assumption:X}) that $|X|<2\lfloor\frac{q+3}{k}\rfloor$ implies that $|X|\leq\frac{2q+6}{k} - 1$.

By putting these together, we get 
\begin{equation}\label{eqn:master}
\frac{4\ell(n-\ell)}{k^2}\leq\frac{2q+6}{k} - 1.
\end{equation}  Then we substitute $q\leq \ell+\frac{n-\ell}{k}$ from Inequality (\ref{eqn:q}) and rearrange the inequality to arrive at
\begin{equation}\label{eqn:l>1}
n(4\ell -2)\leq 2k\ell -2\ell + 6k - k^2 + 4\ell^2.
\end{equation}

At this point, we will substitute $n\geq\ell + k$ from~(\ref{eqn:n2}) and isolate $\ell$ to get
\begin{equation}\label{eqn:lsmall}
\ell\leq 4 - \frac{k}{2}\leq 2.
\end{equation}

On the other hand, if in the left-hand side of Inequality (\ref{eqn:master}) we substitute $\frac{n-\ell}{k}\geq q - \ell$ from~(\ref{eqn:n1}), we get $4\ell(q-\ell) \leq 2q + 6 - k$.  We can simplify this to
\begin{equation}\label{eqn:step2}
(4\ell - 2)q - 4\ell^2 \leq 6 - k.
\end{equation}

Substitute $k\geq 4$ into Inequality (\ref{eqn:step2}) to obtain
\begin{equation}\label{eqn:step3}
(4\ell - 2)q - 4\ell^2 \leq 2.
\end{equation}

Now, we have $\ell\geq 1$ by our assumption at the beginning of this case, and $\ell\leq 2$ from Inequality~(\ref{eqn:lsmall}).  Substituting either $\ell=1$ or $\ell=2$ yields $q\leq 3$.  However, this contradicts our assumption that $q\geq 2k-3\geq 5$, concluding Case 1. 




{\sc Case 2: $\ell=0$.}  Therefore, every connected component of $G^*-X$ has at least $k$ v-vertices.

Let $C_1,C_2,\dotso , C_q$ be the connected components of $G^*-X$.  

We count the number of pairs of vertices that are not covered by edges in $H-X$, and count how many edges must be in $X$ for this to be the case.  We are counting pairs of v-vertices that lie in distinct connected components of $H-X$ (and hence, of $G^* -X)$: this is equivalent to counting the number of edges in the complete multipartite graph with part sizes $|V'(C_1)|, \dotso , |V'(C_q)|$, where $V'(C_i)$ represents the set of v-vertices in $V(C_i)$.  

There are
\begin{equation}\label{eqn:counting}
\binom{n}{2} - \sum_{i=1}^{q}\binom{|V'(C_i)|}{2}
\end{equation}

edges in such a graph; hence, at least that many pairs of vertices are covered by edges of $X$.  We apply Lemma~\ref{lem:notturan} to the sum, using the constraints $|V'(C_i)|\geq k$ for all $1\leq i\leq q$, noting that we have $n\geq kq$ as well.  Then the quantity on Line (\ref{eqn:counting}) is bounded from below by
\begin{equation}
\binom{n}{2} - \sum_{i=1}^{q}\binom{|V'(C_i)|}{2} \geq \binom{n}{2} - (q-1)\binom{k}{2}-\binom{n-k(q-1)}{2}.
\end{equation}

On the other hand, each edge of $X$ covers up to $\binom{k}{2}$ pairs of v-vertices in distinct connected components.  Hence 
\begin{align*}
|X|&\geq\frac{\binom{n}{2} - (q-1)\binom{k}{2}-\binom{n-k(q-1)}{2}}{\binom{k}{2}}\\
&=\frac{n(n-1)-(q-1)k(k-1)-(n-k(q-1))(n-k(q-1)-1)}{k(k-1)}.
\end{align*}

Now, by assumption we have $|X|\leq\frac{2q+6}{k}$, so $$\frac{n(n-1)-(q-1)k(k-1)-(n-k(q-1))(n-k(q-1)-1)}{k(k-1)}\leq\frac{2q+6}{k}.$$

Collecting $n$ on the left-hand side yields $$n(2kq-2k)\leq k^2q^2 - k^2q +2kq + 6k - 2q - 6.$$  When we substitute $n\geq kq$, we can obtain a quadratic in $q$ on the left-hand side:
\begin{equation}\label{eqn:case0}
k^2q^2 - (k^2 +2k - 2)q - 6(k-1)\leq 0.
\end{equation}

Let $f(x)= k^2x^2 - (k^2 + 2k - 2)x - 6(k-1)$.  Then the discriminant of $f(x)$ is $(k^2 + 2k - 2)^2 + 24k^2(k-1)$, which is non-negative since $k\geq 4$.  Hence we get two real roots of $f(x):$ 
\begin{equation}\label{eqn:roots}
x_1,x_2 = \frac{k^2+2k-2 \pm\sqrt{(k^2+2k-2)^2 + 24k^2(k-1)}}{2k^2}.
\end{equation}
Define
\begin{align*}
g(k)&=(k^2+2k-2)^2 + 24k^2(k-1) -k^6\\
&=-k^6 + k^4 + 28k^3 - 24k^2 -8k + 4.
\end{align*}

We can verify using a computer algebra system that the real roots of $g(k)$ are all less than 3, so we have $g(k)<0$ for $k\geq 4$.

Then $(k^2 + 2k -2)^2 + 24k^2(k-1) < k^6$.  Note that the left-hand side of this inequality is the discriminant of $f(x)$, so we know it is non-negative.  We may take square roots of both sides and substitute the resulting right-hand side into (\ref{eqn:roots}) to write $$x_1 < \frac{k^2+2k-2+k^3}{2k^2}.$$ 

Since $f(q)\leq 0$ from (\ref{eqn:case0}), we have $q<x_1$.  From our assumption (\ref{assumption:q}), we have $2k-3\leq q$.  Putting these together, we get
\begin{center}
\begin{tabular}{crl}
 & $2k-3 < x_1$ & \hspace{-10px}$<\frac{k^2+2k-2+k^3}{2k^2}$ \\ 
$\Rightarrow$ & $4k^3-6k^2$ & \hspace{-13px} $<k^2+2k-2+k^3$ \\ 
$\Rightarrow$ & $3k^3-7k^2 -2k + 2$ & \hspace{-8px}$<0$. \\ 
\end{tabular}
\end{center}

Let $h(k)$ denote $3k^3-7k^2 -2k + 2$.  It is easy to verify that $h(k)$ has no real roots on the interval $k\geq 4$.  Hence $h(k)\geq 0$, which is a contradiction.

Since we have obtained a contradiction in every case, we conclude that $|X|\geq \lfloor\frac{c(G^*-X)+3}{k}\rfloor$.  $H$ has no cut edges by Lemma~\ref{lem:2-connected}, so we may apply Lemma~\ref{lem:relax} and conclude that $H$ is quasi-eulerian.
\end{proof}

We are now ready to prove our main result, using Theorem~\ref{thm:2coveringkhypergraphs} in the basis of induction of Theorem~\ref{thm:lcoveringkhypergraphs}.  We use Theorem~\ref{thm:coveringinduction} as well as a strategy similar to its proof.

\begin{thm}\label{thm:lcoveringkhypergraphs}
Let $H$ be an $\ell$-covering $k$-hypergraph for $2\leq\ell < k$.  Then $H$ is quasi-eulerian if and only if $H$ has at least two edges.
\end{thm}

\begin{proof}
$\Rightarrow$: Let $H$ be a quasi-eulerian $\ell$-covering $k$-hypergraph for $2\leq\ell < k$.  Then Lemma~\ref{lem:trivialcases} implies that we cannot have $|E(H)|=1$.  Since $H$ is non-empty by definition, we must have $|E(H)|\geq 2$.

$\Leftarrow$: Fix some $k'\geq 1$ and, for all $\ell\geq 2$, define the proposition \\$P_{k'}(\ell)$: ``All $\ell$-covering $(\ell+k')$-hypergraphs with at least two edges are quasi-eulerian.''

If $k'=1,$ then Theorem~\ref{thm:coveringinduction} states that $P_{k'}(\ell)$ holds for all $\ell\geq 2$.  Hence assume $k'\geq 2$.

We will prove $P_{k'}(\ell)$ by induction on $\ell$.  $P_{k'}(2)$ follows from Theorem~\ref{thm:2coveringkhypergraphs}.  Suppose that, for some $\ell\geq 2$, the proposition $P_{k'}(\ell)$ holds: that is, any $\ell$-covering $(\ell+k')$-hypergraph with at least two edges is quasi-eulerian.

Let $H=(V,E)$ be an $(\ell+1)$-covering $(\ell+k'+1)$-hypergraph with $|E|\geq 2.$  Fix some $v\in V$ and let $\mathcal{B}=\{e\setminus\{v\}:e\in E\text{ such that }v\in e\}$.  Let $H'=(V\setminus\{v\}, \mathcal{B})$.  Note that $H'$ is an $(\ell+k')$-uniform hypergraph.  It also has the property that every $\ell$-tuple of $V(H')$ lies in at least one edge.  To see this, let $X\subset V(H')$ be a set of cardinality $\ell$.  Then, in $H$, the $(\ell+1)$-subset $X\cup \{v\}$ lies in some edge $e$.  Then there exists in $\mathcal{B}$ a corresponding edge $e\setminus \{v\}$, which contains $X$.

Now, let $E'=\{e\in E: v\not\in e\}$.  Obtain a set $\mathcal{C}$ of edges, each of cardinality $\ell+k'$, by taking all the edges of $E'$ and removing an arbitrary vertex from each of them (the choice of vertex does not matter).  Define $H^*=(V\setminus\{v\},\mathcal{B}\cup\mathcal{C}),$ in which the edge set of $H^*$ is taken as a multiset union of $\mathcal{B}$ and $\mathcal{C}$.  Then $H^*$ is an $(\ell+k')$-uniform hypergraph such that every $\ell$-tuple of $V(H^*)$ lies in at least one edge: that is, we have that $H^*$ is an $\ell$-covering $(\ell+k')$-hypergraph.  Since $|E(H^*)|=|E(H)|\geq 2$, we may apply the induction hypothesis to get that $H^*$ is quasi-eulerian, so it admits an Euler family $\F^*$.

Define a function $\varphi: E(H^*)\rightarrow E(H)$ by defining $\varphi(e)=e\cup \{v\}$ if $e$ originally comes from $\mathcal{B}$, or $\varphi$ maps $e$ to its corresponding edge from $E'$ if $e$ originally comes from $\mathcal{C}$.  Since $V(H^*)\subset V(H)$ and $\varphi$ is a bijection, we may apply Lemma~\ref{lem:truncatedhypergraph} to obtain an Euler family $\F$ for $H$.  Hence $H$ is quasi-eulerian and $P_{k'}(\ell+1)$ holds.

Therefore, by induction, we have proven that $P_{k'}(\ell)$ holds for all $\ell\geq 2$.  Letting $k'$ vary over all $k'\geq 2$, we have proven that, for $2\leq\ell <k$, all $\ell$-covering $k$-hypergraphs with at least two edges are quasi-eulerian.
\end{proof}
\cleardoublepage

\chapter{Eulerian Properties of Hypergraphs with Particular Edge Cuts}\label{chapter:edgecuts}
\chaptermark{Hypergraphs with Particular Edge Cuts}


\section{Introduction}

In this chapter, we leave design hypergraphs behind and focus on constructing Euler families and Euler tours in hypergraphs that have certain kinds of edge cuts.  Analogous results on vertex cuts in hypergraphs have been produced by Steimle and \v{S}ajna \cite{SS}.  The reason edges of edge cuts are valuable to investigate is that they serve a very crucial role when constructing a closed walk: they allow a traversal to reach edges that are only accessible by using these edge-cut edges.  This is not an important issue for graphs since an Euler tour can be constructed efficiently using a greedy algorithm.  

Recall Theorem~\ref{thm:nocutedges}, which implies that a hypergraph that has a cut edge may be eulerian (though it must be a trivial cut edge).  One thing that makes dealing with graphs easier than dealing with hypergraphs is that a cut edge precludes an Euler tour in a graph, as we can see in the following result.\\

\begin{prop}\label{prop:nocutedges}
Let $G$ be a graph.  If $G$ is eulerian, then $G$ has no cut edges.
\end{prop}

\begin{proof}
Assume $G$ is eulerian and let $e\in E(G)$.  Let $T=v_0e_1v_1e_2v_2\dotso e_mv_0$ be an Euler tour of $G$, and without loss of generality, assume $e_1=e$.  Then $v_0e_mv_{m-1}e_{m-1}$\\$v_{m-2}\dotso e_2v_1$ is a $v_0v_1$-walk of $G\setminus e$, so Proposition~\ref{prop:G-uv} states that $G\setminus e$ is connected.  Hence $e$ is not a cut edge of $G$.  Letting $e$ vary over all $E(G)$, we see that $G$ has no cut edges, as required.
\end{proof}

Of course, we will not solely investigate hypergraphs that have just a single cut edge, but this distinction between graphs and hypergraphs is instructive nonetheless.  Between this result and Euler's result about connected graphs being eulerian if and only if they are even (Theorem~\ref{thm:eulertour}), there are easy ways to check whether a graph is eulerian or not.  As we saw in Chapter~\ref{chapter:previousresults} (and, in particular, Theorem~\ref{thm:complexity}), deciding whether a hypergraph is eulerian, even when restricted to a specific class of hypergraphs, is NP-complete.

This jump in difficulty when going from graphs to hypergraphs is due to (among other things) the fact that, in a hypergraph, there are multiple ways that any given edge can be traversed.  However, since edges in a fixed edge cut are in some way responsible for affording access to different parts of the hypergraph,  the possibilities for how and when these edges are traversed in a closed trail is more limited than it is for other kinds of edges.  

We can take advantage of the limitations of edge-cut edges to reduce the problem of existence of an Euler family or tour to finding one in some related, and smaller, hypergraphs.  The main idea is that these edge-cut edges get traversed in a way that is reliant on how the others get traversed, so in prescribing how one gets traversed, we gain information about the others.

This chapter is separated into sections based on the kind of edge cut the hypergraph has.  There are, generally speaking, similar results for Euler tours and Euler families in each section.  We will introduce some new tools that help us analyze these hypergraphs, and these tools can be used for problems that are beyond the scope of this chapter.

At the end of the chapter, we produce an algorithm that uses the accrued results to search for an Euler tour in any hypergraph.

\section{Definitions and Basic Facts}



Recall the definition of an edge cut for hypergraphs:\\

\begin{defn}
{\rm Let $H$ be a hypergraph and let $S\subsetneq V(H)$ be nonempty.  Then $F = [S,\bar{S}]= \{e\in E(H): e\cap S\neq\emptyset, e\cap\overline{S}\neq\emptyset\}$ is called an {\em edge cut} of $H$.  $F$ is called {\em minimal} if $H$ has no edge cut properly contained in $F$.
}
\end{defn}

If $F$ is an edge cut of $H$, then we necessarily have that $H\setminus F$ is disconnected.  We will usually denote the connected components of $H\setminus F$ by $H_i$, $i\in I$, for some index set $I$.

We first give a result that demonstrates that deleting an edge cut yields a disconnected subhypergraph, and vice-versa: any disconnected subhypergraph of $H$ is obtained from $H$ by deleting, at minimum, an edge cut from it (and then possibly some more edges).

\begin{lem}\label{lem:minedgecut}
Let $H$ be a hypergraph, and let $F\subseteq E(H).$ Then $H\setminus F$ is disconnected if and only if $H$ has an edge cut $F'\subseteq F$.
\end{lem}

\begin{proof}
$\Rightarrow$: Assume that $H\setminus F$ is disconnected, and let $H_1,\dotso , H_k$, with $k\geq 2$, be the connected components of $H\setminus F$.  Define $S=V(H_1)$, so that $\bar{S}=V(H_2)\cup\dotso\cup V(H_k)$.  Let $F'=[S,\bar{S}]$, so $F'$ is an edge cut of $H$.

We now show that $F'\subseteq F$.  Suppose $f\in F'$.  Then $f$ contains a vertex $u\in S$ and a vertex $v\in\bar{S}$.  If $f\not\in F$, then $u$ and $v$ are connected in $H\setminus F$, a contradiction because $u$ and $v$ lie in different connected components of $H\setminus F$.  Hence $f\in F$, and we conclude that $F'\subseteq F$.

$\Leftarrow$:  Assume $H$ has an edge cut $F'\subseteq F$. Let $F'=[S,\bar{S}]$, where $\emptyset\subsetneq S\subsetneq V(H)$.  Let $u\in S$ and $v\in\bar{S}$.  If there exists a $uv$-walk $W$ in $H\setminus F'$, then $W$ traverses an edge $e$ that contains a vertex in $S$ and a vertex in $\bar{S}$.  Then $e$ must be in $F'$, a contradiction.  Therefore, there is no such walk $W$.

Hence $H\setminus F'$ is disconnected.  Since $F\supseteq F'$, it follows that $H\setminus F$ is disconnected as well.
\end{proof}

A lot of results we will be developing in this chapter require minimal edge cuts, since they have a handy property that we will see in the upcoming lemma.  Accordingly, we will introduce some convenient terminology to help us along.

\begin{defn}
{\rm Let $H$ be a hypergraph and $H'$ be a subhypergraph of $H$. We say that an edge $e\in E(H)$ {\em intersects} $H'$ if $e\cap V(H')\neq\emptyset$.
}
\end{defn}

\begin{lem}\label{lem:intersects}
Let $H$ be a hypergraph, and let $F$ be an edge cut of $H$.  Let $H_i$, for $i\in I$, be the connected components of $H\setminus F$.

Then $F$ is a minimal edge cut if and only if every edge of $F$ intersects $H_i$ for all $i\in I$.
\end{lem}

\begin{proof}
$\Rightarrow$: We prove the contrapositive: assume there is some $f\in F$ and some $i$ such that $f$ does not intersect $H_i$.  Let $F' = F\setminus\{f\}$.  Since $H_i$ is a connected component of $H\setminus F$, there is no edge in $H\setminus F$ that joins a vertex of $V(H_i)$ to a vertex of $V(H)\setminus V(H_i)$.  Likewise, since $f$ does not intersect $H_i$, there is no edge in $H\setminus F'$ that joins a vertex of $V(H_i)$ to a vertex of $V(H)\setminus V(H_i)$.  Then $H\setminus F'$ is disconnected, so Lemma~\ref{lem:minedgecut} implies that there is an edge cut $F''\subseteq F'\subset F$. Hence $F$ is not a minimal edge cut.

$\Leftarrow$: Suppose each edge of $F$ intersects every connected component of $H\setminus F$.  Fix some $f\in F$ and let $F'=F\setminus\{f\}$.

Since $f$ intersects every connected component of $H\setminus F$, we see that $H\setminus F'$ is connected.  Then Lemma~\ref{lem:minedgecut} implies that there is no edge cut contained in $F'$.  As $f$ is arbitrary, this shows that there is no edge cut properly contained in $F$, so $F$ is a minimal edge cut.
\end{proof}

We next present a simple result that shows that an Euler family of a hypergraph $H$ contains an Euler family of every union of connected components of $H$.\\

\begin{lem}\label{lem:decomposingfamilies}
Let $H$ be a hypergraph with connected components $H_i,$ for $i\in I$.  Then the following hold:

\begin{description}
\item[(1)] If, for each $i\in I$, we have that $H_i$ has an Euler family $\F_i$, then $\bigcup\limits_{i\in I}\F_i$ is an Euler family of $H$.  If each $\F_i$ is spanning in $H_i$, then so is $\F$ in $H$.
\item[(2)] If $H$ has an Euler family $\F$, then $\F$ has a partition $\{\F_i:i\in I\}$ such that, for each $i\in I$, we have that $\F_i$ is an Euler family of $H_i$.  If $\F$ is spanning in $H$, then so is each $\F_i$ in $H_i$.
\end{description}
\end{lem}

\begin{proof}
{\bf (1)} For each $i\in I$, assume $H_i$ has an Euler family $\F_i$.

Then $\F=\bigcup\limits_{i\in I}\F_i$ is a collection of anchor-disjoint closed trails that traverses every edge of $\bigcup\limits_{i\in I}H_i=H$ exactly once.  Hence $\F$ is an Euler family of $H$.

Suppose each $\F_i$ is spanning. Then $\F$ traverses all the vertices that are traversed by any $\F_i$.  Since each vertex of every connected component of $H$ is traversed by some $\F_i$, we see that $\F$ is spanning in $H$.

{\bf (2)} Assume $H$ has an Euler family $\F$.  For each $i\in I$, define $\F_i=\{T\in\F: T\text{ is a trail in }H_i\}$.  Since every trail of $\F$ is in exactly one $H_i$, we have that $\{\F_i: i\in I\}$ is a partition of $\F$.

Now, fix $i\in I$ and let $e\in E(H_i)$.  Then $e$ is traversed in $\F$, so there exists unique $T\in\F$ that traverses $e$. We then have that $T$ is a trail of $H_i$, so $T\in\F_i$.  Hence $e$ is traversed exactly once in $\F_i$.  Since the trails of $\F$ are anchor-disjoint, so too are the trails of $\F_i$, hence $\F_i$ is an Euler family of $H_i$.

Suppose now that $\F_i$ is not spanning in $H_i$, for some $i\in I$.  Then there exists $v\in V(H_i)$ not traversed by $\F_i$.  It must be the case that $v$ is not traversed by $\F$, so $\F$ is not spanning in $H$.  Therefore, by the contrapositive, each $\F_i$ is spanning in $H_i$ if $\F$ is spanning in $H$.
\end{proof}

\begin{cor}\label{cor:decomposingfamilies}
Let $H$ be a hypergraph with connected components $H_i$, for $i\in I$.  Assume that $\F$ is an Euler family of $H$, and $\{\F_i: i\in I\}$ a partition of $\F$ such that, for all $i\in I$, we have that $\F_i$ is an Euler family of $H_i$. 

Then, for each $J\subseteq I$, we have that $\bigcup\limits_{i\in J}\F_i$ is an Euler family of $\bigcup\limits_{i\in J}H_i$.
\end{cor}

\begin{proof}Let $\F$ be an Euler family of $H$.  Then Lemma~\ref{lem:decomposingfamilies} (2) produces the required partition $\{\F_i:i\in I\}$ of $\F$.  Fix some $J\subseteq I$ and let $H'=\bigcup\limits_{i\in J}H_i$.  Lemma~\ref{lem:decomposingfamilies} (1) demonstrates that $\bigcup\limits_{i\in J}\F_i$ is an Euler family of $H'$.
\end{proof}






\section{Technical Lemmas}
We are now ready to present the two main tools of this chapter that generate an auxiliary graph (or hypergraph).  The simpler nature of these auxiliary (hyper)graphs allows us to focus on how a trail might traverse edges of an edge cut.\\

\begin{defn}\label{defn:G}
{\rm Let $H$ be a connected hypergraph with a minimal edge cut $F$ and an Euler family $\F$.  Let $H_i$, for $i\in I$, denote the connected components of $H\setminus F$.  
\begin{description}
\item[(a)] For every $f\in F$ such that $f$ is traversed in $\F$ via an anchor in $H_i$ and an anchor in $H_j$, we denote $f_G=ij$.
\item[(b)] Define a multigraph $G=G(H,F,\F)$ as follows:
	\begin{itemize}
	\item $V(G)=I$, and
	\item $E(G)=\{f_G:f\in F\}$.
	\end{itemize}
\item[(c)] Let $\mathcal{P}_G$ be the partition of $I$ into vertex sets of the connected components of $G$; that is, $$\mathcal{P}_G=\{J\subseteq I: G[J]\text{ is a connected component of }G\}.$$
\item[(d)] For any $J\subseteq I$, we define the following:
	\begin{itemize}
	\item $F_J=\{f\in F: f_G\in E(G[J])\}$;
	\item for each $f\in F_J: f'=f\cap\big(\bigcup\limits_{i\in J}V(H_i)\big)$;
	\item $F_J'=\{f':f\in F_J\}$;
	\item $H_J'=\big(\bigcup\limits_{i\in J}H_i\big) + F_J'$.
	\end{itemize}
\end{description}
} \end{defn}

Note that, since $G$ is a multigraph, it may have loops and parallel edges.  A loop is produced when an edge $f\in F$ is traversed in $\F$ via two vertices of the same connected component of $H\setminus F$.  Parallel edges are produced when multiple edges of $F$ are traversed via anchors of the same pair of connected components (though not necessarily the same anchors).

We explore some ramifications of these definitions in the following lemma.\\

\begin{lem}\label{lem:G}
Let $H$ be a nonempty connected hypergraph with a minimal edge cut $F$ and Euler family $\F$, and let $H_i$, for $i\in I$, be the connected components of $H\setminus F$.  Then $G=G(H,F,\F)$ satisfies the following:
\begin{description}
\item[(1)] $G$ is an even graph with $|I|$ vertices and $|F|$ edges.
\item[(2)] $H'=\bigcup\limits_{J\in\mathcal{P}_G}H_J'$ has an Euler family $\F'$ obtained from $\F$ by replacing each edge $f\in F$ with $f'\in F_J'$ for some appropriate choice of $J\in\mathcal{P}_G$.
\item[(3)] $\F'$ has a partition $\{\F_J': J\in\mathcal{P}_G\}$ such that $\F_J'$ is an Euler family of $H_J'$ for each $J\in\mathcal{P}_G$.  Hence $|\F|=\sum\limits_{J\in\mathcal{P}_G}|\F_J'|$.
\item[(4)] If $\F$ is spanning in $H$, then $\F'$ is spanning in $H'$, and $\F_J'$ is spanning in $H_J'$ for all $J\in\mathcal{P}_G$.
\item[(5)] If $|\F|=1$, then there exists a unique $J\in\mathcal{P}_G$ such that $G[J]$ is nonempty.  For this $J$, we have that $H_J'$ has a nontrivial Euler tour, and $H_i$ is empty for all $i\not\in J$.
\end{description}
\end{lem}

\begin{proof}
{\bf (1)}: It is evident that $G$ has $|I|$ vertices and $|F|$ edges from its definition, so we need only show that $G$ is an even graph.

Let $T\in\F$.  Then $T=u_1T_1v_1f_1u_2T_2v_2f_2\dotso v_sf_su_1$, where $T_i$ is a trail in $H_{j_i}$, $u_i,v_i\in V(H_{j_i})$, and $f_i\in\F$, for each $i=1,\dotso , s$.

Then $T'=j_1(f_1)_{G}j_2(f_2)_{G}\dotso (f_s)_{G}j_1$ is a closed trail in $G$, and $\{T':T\in\F\}$ is a decomposition of $G$ into closed trails.  Hence $G$ is even.

{\bf (2)} Let $f\in F$.  Then, since $\mathcal{P}_G$ is a partition of $I$ into the vertex sets of the connected components of $G$, we have that there exists a unique $J\in\mathcal{P}_G$ such that $f_G\in E(G[J])$.  Hence there is a unique $J\in\mathcal{P}_G$ such that $F_J'$ contains the corresponding edge $f'$.  Furthermore, since $f$ is traversed in $\F$ via vertices in $V(H_J')$, we know that $f'=f\cap V(H_J')$ contains those vertices.

We obtain $\F'$ from $\F$ by replacing each edge $f\in F$ with the corresponding edge $f'$ that lies in $F_J'$ for the unique $J'\in\mathcal{P}_G$ as described above.  Note that $H'$ is obtained from $H$ by replacing each $f\in F$ with $f'$, so this suggests that $\F'$ is an Euler family of $\F$.

Formally, we define a bijection $\varphi:E(H)\rightarrow E(H')$ that maps each $f\in F$ to the corresponding $f'$ described above, and preserves all the other edges.  Since we have $V(H)=V(H')$ and $\varphi(e)\subseteq e$ for all $e\in E(H)$, and every $e\in E(H)$ is traversed in $\F$ via vertices in $\varphi(e)$, we may apply Lemma~\ref{lem:truncatedhypergraph} (2) using the bijection $\varphi^{-1}$ to conclude that $\F'$ is an Euler family of $H'$.

{\bf (3)} Observe that $\{H_J':J\in\mathcal{P}_G\}$ is the set of connected components of $H'$.  Then Lemma~\ref{lem:decomposingfamilies} implies that $\F'$ has a partition $\{\F_J':J\in\mathcal{P}_G\}$ such that, for each $J\in\mathcal{P}_G$, we have that $\F_J'$ is an Euler family of $H_J'$.

It should be clear from our construction of $\F'$ that $|\F'|=|\F|$, so by summing cardinalities of a disjoint union of sets, we have $|\F|=|\F'|=\sum\limits_{J\in\mathcal{P}_G}|\F_J'|$.

{\bf (4)} Assume $\F$ is spanning in $H$.  

$\F'$ has the same set of anchors as $\F$, and $V(H')=V(H)$, so we have that $\F'$ is spanning in $H'$.

Let $v\in V(H_J')$ for some $J\in\mathcal{P}_G$.  Then there exists $T\in\F'$ that traverses $v$, since $\F'$ is spanning in $H'$ and $v\in V(H')$.  Then $T$ must be a closed trail in $H_J'$, so necessarily we have $T\in\F_J'$.  Hence $v$ is traversed by $\F_J'$, and so every vertex of $V(H_J')$ is traversed by $\F_J'$.

{\bf (5)} Assume $|\F|=1$. Since $F$ is nonempty, so is $G$, so assume that $J\in\mathcal{P}_G$ is such that $G[J]$ is nonempty.

Let $K\in\mathcal{P}_G$ be distinct from $J$.  Then by (3) of this lemma, there exists an Euler family $\F_J'$ of $H_J'$ and $\F_K'$ of $H_K'$.  Again by (3), we have that $1=|\F|\geq |\F_J'| + |\F_K'|$, hence one of $\F_J'$ and $\F_K'$ is empty.  But $G[J]$ is nonempty, which implies that $F_J$ (and hence $F_J'$) is nonempty.  So $H_J'$ is nonempty, so it cannot have an empty Euler family.  Therefore, we have that $\F_K'$ is empty, and so we must conclude that $H_K'$ is empty for all $K\neq J$.  This implies that $F_K'$ is empty, so we also conclude that $G[K]$ is empty for all $K\neq J$, hence $K$ is singleton.  Hence $H_K'=H_k$ for some $k\in I$ and $H_i$ is empty for all $i\not\in J$.

Furthermore, since $\F_J'$ is an Euler family of $H_J'$ of cardinality 1, it contains a nontrivial Euler tour of $H_J'$. 
\end{proof}

\begin{defn}
{\rm {\bf (Collapsed Hypergraph)} Let $H$ be a hypergraph and let \\$\{V_0, V_1, \dotso , V_k\}$ be a partition of $V(H)$ with $k\geq 1$.  We define the {\em collapsed hypergraph} $H\circ\{V_1,\dotso , V_k\}=(V^{\circ},E^{\circ})$ as follows:
\begin{itemize}
\item $V^{\circ}=V_0\cup\{v_1,\dotso , v_k\}$ for some distinct vertices $v_1,\dotso , v_k\not\in V(H)$;
\item First take $E'=\{(e\cap V_0) \cup \{v_i: 1\leq i\leq k\text{ and } e\text{ intersects } V_i\}: e\in E(H)\}$, then define $E^{\circ}=\{e\in E': |e|\geq 2\}$.
 \end{itemize}
 
The new vertices $v_1,\dotso , v_k$ are called the {\em collapsed vertices} of the collapsed hypergraph.
  
If $k=1$, then we may simply write $H\circ V_1$ instead of $H\circ\{V_1\}$, {\em i.e.}, we may omit the set braces around $V_1$, so long as this does not cause ambiguity.
} \end{defn}

The idea behind collapsed hypergraphs is that we identify all the vertices of $V_1$, then all the vertices of $V_2$, and so on.  The collapsed vertices are introduced, for clarity's sake, to stand in for these vertices that have been identified.  The edges of $E'$ are obtained from $E(H)$ in exactly the way one would expect after identifying sets of vertices; the difficulty that can arise is that we end up with edges of cardinality 1 if an edge of $E(H)$ lies entirely in one of the vertex sets $V_i$.  We discard these to obtain $E^{\circ}$, for we do not want edges of cardinality 1.\\

\begin{remark}{\rm
The goals of defining $G(H,F,\F)$ and a collapsed hypergraph are similar: we wish to simplify the parts of the hypergraph that we are not terribly interested in, so that we may focus on how the edge-cut edges must be used to navigate to each of those parts.  In the case of $G(H,F,\F)$, we are already given an Euler family $\F$, and so we use this auxiliary graph to analyze what the traversals must say about the connected components of $H\setminus F$.

In the case of the collapsed hypergraph, we do not need to already have an Euler family in mind, and we can choose to simplify only certain parts of the hypergraph.  This can be very useful in certain situations (as we will see in Theorem~\ref{thm:edgescard2}), but since we have a hypergraph, it is not permitted to have any ``loops,'' a restriction that $G(H,F,\F)$ does not have.  If the collapsed hypergraph is to be used in situations beyond what we explore in this chapter, then this weakness will have to be navigated with care, for it might not correctly model what is required.

We now present a lemma that allows us to show that certain collapsed hypergraphs are quasi-eulerian when the original hypergraph is.\\
}
\end{remark}

\begin{lem}\label{lem:collapsedG}
Let $H$ be a nonempty connected hypergraph with a minimal edge cut $F$ and Euler family $\F$.  Let $H_i$, for $i\in I$, be the connected components of $H\setminus F$, and assume that $G=G(H,F,\F)$ is connected.

Take any $K\subseteq I$ such that $G[K]$ is loopless, and denote $H^{(K)}=H\circ\{V(H_k):k\in K\}$.

Then $H^{(K)}$ has an Euler family $\F^{(K)}$ such that the following hold:

\begin{description}
\item[(1)] $|\F^{(K)}|\leq |\F|$;
\item[(2)] $\F$ and $\F^{(K)}$ have the same anchors in $\bigcup\limits_{i\in I\setminus K}V(H_i)$; and
\item[(3)] For any $T\in\F$ and $k\in K$, we have that $T$ traverses a vertex in $V(H_k)$ and a vertex outside of $V(H_k)$ if and only if $\F^{(K)}$ traverses the collapsed vertex $u_k$ corresponding to $V(H_k)$.
\end{description}
\end{lem}

\begin{proof}
Let $K\subseteq I$ and assume that $G[K]$ is loopless.  Let $u_k$, for $k\in K$, be the collapsed vertices of $H^{(K)}$.  Define a map $\varphi: V(H)\rightarrow V(H^{(K)})$ as follows:
\[   \varphi(v)= \left\{
\begin{array}{lcl}
      v & : & v\in V(H_i)\text{ for some }i\in I\setminus K \\
      u_k & : & v\in V(H_k)\text{ for some }k\in K. \\      
\end{array} 
\right. \]

That is, $\varphi$ maps each vertex of $H$ to its corresponding (possibly collapsed) vertex in $H^{(K)}$.  Then, for each $e\in E(H)$, we can refer to $\varphi(e)$ as the image of the set $e$ under $\varphi$, which is either an edge of $H^{(K)}$, or a singleton.

Now, for each $e\in E(H)$, we can investigate how $\varphi(e)$ looks depending on where $e$ is from.  If $e\in F$, then Lemma~\ref{lem:intersects} tells us that $\varphi(e)$ contains every collapsed vertex of $H^{(K)}$, in addition to vertices from $\bigcup\limits_{i\in I\setminus K}V(H_i)$.  On the other hand, if $e\not\in F$, then $e\in E(H_i)$ for some $i\in I$.  If $i\in K$, then $\varphi(e)$ is a singleton.  If $i\in I\setminus K$, then $\varphi(e) = e$.

Since we assume that $H$ has an Euler family $\F$, let $T=v_0e_1v_1\dotso v_{r-1}e_{r}v_0$ be a trail of $\F$.  Denote by $\varphi(T)$ the sequence obtained by applying $\varphi$ to each of its anchors and edges.  Since $T$ is a trail, we have that each $e_i$ contains $v_{i-1}$ and $v_i$, for $1\leq i\leq r$ (where $v_r$ denotes $v_0$).  Hence each $\varphi(e_i)$ contains $\varphi(v_{i-1})$ and $\varphi(v_i)$.

Now, we construct a new sequence $T^*$ from $\varphi(T)$ by deleting any subsequence $\varphi(e_i)\varphi(v_i)$ such that $\varphi(e_i)$ is singleton.  Observe that $\varphi(e_i)$ being singleton implies that $\varphi(v_i)=\varphi(v_{i-1})$, so $T^*$ has the property that each remaining $\varphi(e_j)$ contains the vertices immediately preceding and following it in the sequence.  Furthermore, we now have that no two consecutive anchors of $T^*$ are the same, so $T^*$ is a walk in $H^{(K)}$.

In fact, observe that all the remaining $\varphi(e_i)$ in $T^*$ are edges of $H^{(K)}$.  Since each edge of $T^*$ corresponds to a distinct edge of $T$, we have that $T^*$ is a trail.  Since $\varphi(E(H))$ includes every element of $E(H^{(K)})$ and $\F$ traverses every edge of $H$ exactly once, we have that the resulting collection $\F^{(K)}=\{T^*: T\in\F, T^*\text{ is nontrivial}\}$ traverses every edge of $H^{(K)}$ exactly once.  Hence $\F^{(K)}$ is an Euler family of $H^{(K)}$, and $|\F^{(K)}|\leq |\F|$.  This proves (1).

Furthermore, suppose $v\in\bigcup\limits_{i\in I\setminus K}V(H_i)$ is a vertex traversed by $\F$ but not by $\F^{(K)}$.  This implies that for each trail $T\in\F$ that traverses $v$, the corresponding trail $T^*$ is trivial and so is excluded from $\F^{(K)}$.  Since $T$ is not trivial, this implies that for any edge $e$ of $T$ that is traversed via $v$, we have that $\varphi(e)$ is singleton.  But this implies that $e\in E(H_k)$ for some $k\in K$, while $v\in e$ does not lie in $V(H_k)$ --- a contradiction.  It should be evident that the set of anchors of $\F^{(K)}$ in $\bigcup\limits_{i\in I\setminus K}V(H_i)$ is contained in the set of anchors of $\F$, so this is sufficient to show that $\F$ and $\F^{(K)}$ have the same anchors in $\bigcup\limits_{i\in I\setminus K}V(H_i)$, completing the proof of (2).

Finally, let $T\in\F$ and $k\in K$ be such that $T$ traverses a vertex of $V(H_k)$ and a vertex outside of $V(H_k)$.  Then $\varphi(T)$ contains $u_k$ as well as a vertex that is not $u_k$, so the resulting trail $T^*$ is nontrivial.  Hence $T^*\in\F^{(K)}$ and it traverses $u_k$.  On the other hand, now suppose $\F^{(K)}$ traverses the collapsed vertex $u_k$.  There must be some trail $T^*\in\F^{(K)}$ that traverses $u_k$, and, by virtue of $T^*$ being nontrivial, the corresponding trail $T\in\F$ cannot be a trail of $H_k$, although it traverses a vertex of $V(H_k)$.  Hence $T$ traverses a vertex in $V(H_k)$ and a vertex outside of $V(H_k)$, completing the proof of (3).
\end{proof}

We will begin to see how these auxiliary graphs and hypergraphs can help us pin down Euler families and Euler tours when we present the main results in subsequent sections.

\section{Eulerian Properties of Hypergraphs with Nonspecific Edge Cuts}

We first give a broad result about the existence of Euler families in hypergraphs with edge cuts.  It states that an Euler family is equivalent to some ``choice function'' $\alpha$ that determines into which two connected components of $H\setminus F$ each edge of the edge cut $F$ should go.  Of course, if we already have an Euler family $\F$, then the edges of our graph $G(H,F,\F)$ tell us how each edge of $F$ is assigned, and so it naturally produces such a choice function.  However, if we start off with a function and we want to ensure that it is a valid assignment of the edges of $F$, it needs to satisfy several conditions.\\

\begin{thm}\label{thm:fix} Let $H$ be a nonempty connected hypergraph with a minimal edge cut $F$.  Let $H_i$, for $i\in I$, be the connected components of $H\setminus F$.  Then $H$ has a (spanning) Euler family $\F$ if and only if there exist a function $\alpha: F\rightarrow\{ij: i,j\in I\}$ and a multigraph $G_{\alpha}$ defined by
\begin{itemize}
\item $V(G_{\alpha})=I$;
\item for each $i,j\in V(G_{\alpha})$, the multiplicity of the edge $ij$ in $G_{\alpha}$ is $|\alpha^{-1}(ij)|$.
\end{itemize}
such that for each connected component $C$ of $G_{\alpha}$, we have that the hypergraph $$H(C)=\big(\bigcup_{i\in V(C)}H_i\big) + \big\{f\cap\big(\bigcup_{i\in V(C)}V(H_i)\big) : f\in\alpha^{-1}(E(C))\big\}$$ has a (spanning) Euler family $\F_C$.

Furthermore, if $\mathcal{C}$ denotes the collection of connected components of $G_{\alpha},$ then these Euler families can be chosen to satisfy $|\F|=\sum\limits_{C\in\mathcal{C}}|\F_C|$.
\end{thm}

\begin{proof}
$\Rightarrow$: Assume $H$ has an Euler family $\F$, and let $G=G(H,F,\F)$.  Adopt the rest of the notation from Definition~\ref{defn:G}.

Now define $\alpha:F\rightarrow\{ij: i,j\in I\}$ by $\alpha(f) = f_G$.  Then $G_{\alpha}=G$.  

Let $C$ be a connected component of $G_{\alpha}$ (and, hence, a connected component of $G$).  Then $C=G[J]$ for some $J\in\mathcal{P}_G$, so by Lemma~\ref{lem:G} (3), we know that $H_J'$ has an Euler family, which we will call $\F_C$.

From Definition~\ref{defn:G}, we can see that 
\begin{align*}
F_J&=\{f\in F: f_G\in E(G[J])\}\\
&= \{f\in F:\alpha(f)\in E(C)\}\\
&= \{f\in F: f\in\alpha^{-1}(E(C))\}\text{, so}\\
F_J'&=\{f':f\in F_J\}\\ 
&= \{f\cap\big(\bigcup\limits_{i\in J}V(H_i)\big) : f\in\alpha^{-1}(E(C))\},
\end{align*}
hence $H_J'=H(C)$ and so $\F_C$ is an Euler family of $H(C)$.

Furthermore, if $\F$ is spanning in $H$, then Lemma~\ref{lem:G} (4) states that $\F_C$ is spanning in $H(C)$.

Finally, Lemma~\ref{lem:G} (3) implies that the collection of $\F_C$, for $C\in\mathcal{C}$, satisfies $|\F|=\sum\limits_{C\in\mathcal{C}}|\F_C|$.

$\Leftarrow$: Let $\alpha:F\rightarrow\{ij:i,j\in I\}$ be a function and $G_{\alpha}$ be the corresponding multigraph.  Assume that, for each connected component $C$ of $G_{\alpha}$, we have that $H(C)$ has an Euler family $\F_C$.

Let $C$ be a connected component of $G_{\alpha}$.  Define $E^{(C)}=E\big(\bigcup\limits_{i\in V(C)}H_i\big) \cup\alpha^{-1}(E(C))$ and $\varphi: E^{(C)}\rightarrow E(H(C))$ by
\[   \varphi(e)= \left\{
\begin{array}{lcl}
      e\cap\big(\bigcup\limits_{i\in V(C)}V(H_i)\big) & : & e\in F \\
      e & : & \text{otherwise}. \\      
\end{array} 
\right. \]
Then $\varphi$ is a bijection and $\varphi(e)\subseteq e$ for all $e\in E^{(C)}$.  Lemma~\ref{lem:truncatedhypergraph} (1) then states that $(V,E^{(C)})$ has an Euler family $\F_C'$ obtained from $\F_C$ by replacing each edge $e$ with $\varphi^{-1}(e)$.

We claim that $\F=\bigcup\limits_{C\in\mathcal{C}}\F_C'$ is the required Euler family of $H$.  Since $\{V(H(C)):C\in\mathcal{C}\}$ is a partition of $V(H)$, the components of $\F$ are pairwise anchor-disjoint, and since $\{E(H(C)):C\in\mathcal{C}\}$ is a partition of $E(H)$, every edge of $H$ is traversed exactly once in $\F$.

Furthermore, assume each $\F_C$ is spanning in $H(C)$.  Then each $\F_C'$ traverses the same set of anchors as $\F_C$, and $\bigcup\limits_{C\in\mathcal{C}}V(H(C)) = V(H)$, so $\F$ traverses every vertex of $H$.

Finally, since the families $\F_C$, for $C\in\mathcal{C}$, are pairwise disjoint, so are the families $\F_C'$.  Then we have that $|\F|=\sum\limits_{C\in\mathcal{C}}|\F_C'| =\sum\limits_{C\in\mathcal{C}}|\F_C|$. 
\end{proof}

We also present a version of Theorem~\ref{thm:fix} for Euler tours.\\

\begin{cor}\label{cor:fix}Let $H$ be a nonempty connected hypergraph with a minimal edge cut $F$.  Let $H_i$, for $i\in I$, be the connected components of $H\setminus F$.  Then $H$ has an Euler tour $T$ if and only if there exists a function $\alpha: F\rightarrow\{ij: i,j\in I\}$ and a multigraph $G_{\alpha}$ defined by
\begin{itemize}
\item $V(G)=I$;
\item For each $i,j\in V(G)$, the multiplicity of the edge $ij$ in $G_{\alpha}$ is $|\alpha^{-1}(ij)|$.
\end{itemize}
such that the following hold:
\begin{description}
\item[(1)] $G_{\alpha}$ has a single nonempty connected component $C^*$, and $$H(C^*)=\big(\bigcup\limits_{i\in V(C^*)}H_i\big) + \big\{f\cap\big(\bigcup\limits_{i\in V(C^*)} V(H_i)\big): f\in\alpha^{-1}(E(C^*))\big\}$$ has an Euler tour $T_{C^*}$;
\item[(2)] $H_i$ is an empty connected component of $H$, for each $i\not\in V(C^*)$.
\end{description}

Furthermore, we have that $T$ is spanning in $H$ if and only if $T_{C^*}$ is spanning in $H(C^*)$ and $C^*$ is the only connected component of $G_{\alpha}$.
\end{cor}

\begin{proof} $\Rightarrow$: Let $T$ be an Euler tour of $H$, and $G=G(H,F,\{T\})$.  Define $\alpha$ as in the proof of Theorem~\ref{thm:fix} ($\Rightarrow$), so that $G_{\alpha}=G$.  Let $\mathcal{C}$ be the set of all connected components of $G$.  By Theorem~\ref{thm:fix}, for each $C\in\mathcal{C}$, we have that $H(C)$ has an Euler family $\F_C$, and $1=|\{T\}|=\sum\limits_{C\in\mathcal{C}}|\F_C|$.

Hence there exists $C^*\in\mathcal{C}$ such that $|\F_{C^*}|=1$ and $|\F_C|=0$ for all $C\in\mathcal{C}\setminus\{C^*\}$.  Therefore, we have that $H(C^*)$ has an Euler tour $T_{C^*}$.  On the other hand, for all $C\neq C^*$, we have that $H(C)$ is empty; hence $C$ is empty and $\bigcup\limits_{i\in V(C)}H_i$ is empty.  It follows that $C=\{i\}$ for some $i\not\in V(C^*)$ and $H_i$ is empty.

Further, assume $T$ is spanning.  Theorem~\ref{thm:fix} says that $T_{C^*}$ is spanning.  It also implies that $\F_C$ is spanning in $H(C)$ for each $C\neq C^*$.  However, we have established that $\F_C$ is empty for such $C$, and an empty Euler family cannot traverse any vertices.  Therefore, in this case, there are no connected components of $G_{\alpha}$ except for $C^*$.

$\Leftarrow$: Let $\alpha$ be a function and $G_{\alpha}$ be the corresponding multigraph satisfying (1) and (2).  If $C=\{i\}$ is an empty connected component of $G_{\alpha}$, then $H_i=H(C)$ is an empty connected component of $H$ and has an empty Euler family $\F_C$.  Let $\F_{C^*}=\{T_{C^*}\}$.  Then $\alpha$ and $G_{\alpha}$ satisfy the conditions of Theorem~\ref{thm:fix}, and $H$ has an Euler family $\F$.  Morever, we have $|\F|=\sum\limits_{C\in\mathcal{C}}|\F_C| = |\{T_{C^*}\}|=1$, so $H$ has an Euler tour.

Furthermore, if $T_{C^*}$ is spanning in $H(C^*)$ and $C^*$ is the only connected component of $G$, then Theorem~\ref{thm:fix} implies that $T$ is spanning in $H$.
\end{proof}

Theorem~\ref{thm:fix} can seem unwieldy, but it leads to more digestible results on both Euler families and Euler tours, which we now present.\\

\begin{thm}\label{thm:edgecutfamily} Let $H$ be a hypergraph with a nonempty minimal edge cut $F$.  Let $H_i$, for $i\in I$, denote the connected components of $H\setminus F$.  

Then $H$ has a (spanning) Euler family $\F$ if and only if there exists $J\subseteq I$ with $1\leq |J|\leq |F|$ such that the following hold:
\begin{description}
\item[(1)] $H[\bigcup\limits_{j\in J}V(H_j)]$ has a nonempty (spanning) Euler family $\F_J$; and
\item[(2)] $H_i$ has a (spanning) Euler family $\F_i$ for all $i\not\in J$.
\end{description}

Furthermore, the Euler families $\F, \F_J,$ and each $\F_i$ can be chosen so that they satisfy $|\F|=|\F_J|+\sum\limits_{i\in I\setminus J}|\F_i|$.
\end{thm}

\begin{proof}$\Rightarrow$: Let $\F$ be an Euler family of $H$.  Construct $G=G(H,F,\F)$ and adopt the notation from Definition~\ref{defn:G}.

Let $\mathcal{C}'$ be the set of nonempty connected components of $G$, and let $J=\bigcup\limits_{C\in\mathcal{C}'}V(C)$.

Since $\bigcup\limits_{C\in\mathcal{C}'}C$ is a graph with $|J|$ vertices, $|F|$ edges, and minimum degree at least 2, the Handshaking Lemma~\ref{lem:handshaking} gives $2|F|\geq 2|J|$.  Furthermore, since $G$ is nonempty, we have $|J|\geq 1$.  Hence $1\leq |J|\leq |F|$.

If $i\in I\setminus J$, then $i$ is an isolated vertex of $G$, and $\{i\}\in\mathcal{P}_G$.  Hence $H_{\{i\}}=H_i$ has an Euler family $\F_i'$, by Lemma~\ref{lem:G} (3).

Also by Lemma~\ref{lem:G} (3), there exists an Euler family $\F_K'$ of $H_K'$ for each $K\in\mathcal{P}_G$.  Let $\F_J''=\bigcup\limits_{\substack{K\in\mathcal{P}_G\\ K\subseteq J}}\F_K'$ and $H_J''=\bigcup\limits_{\substack{K\in\mathcal{P}_G\\ K\subseteq J}}H_K'$.

Lemma~\ref{lem:decomposingfamilies} (1) states that $\F_J''$ is an Euler family of $H_J''$.

We use $\F_J''$ to construct an Euler family of $H[\bigcup\limits_{i\in J}V(H_i)]$ as follows.  First, observe that since $F_J=F$, we have $H[\bigcup\limits_{i\in J}V(H_i)]=H_J'$, and that $V(H_J')=V(H_J'')=\bigcup\limits_{i\in J}V(H_i)$.

Define $\varphi: E(H_J')\rightarrow E(H_J'')$ as follows.  If $e\in E(H_i)$ for some $i\in J$, then let $\varphi(e)=e$.  Otherwise, we have $e=f\cap V(H_J')$ for some $f\in F$.  We then let $\varphi(e)=f\cap V(H_K')$, where $K\in\mathcal{P}_G$ is such that $f_G\subseteq K$.  Then $\varphi$ is well defined since each edge of $G$ lies in exactly one connected component of $G$.  Moreover, we have that $\varphi$ is a bijection, and $\varphi(e)\subseteq e$ for all $e\in E(H_J')$.  It now follows from Lemma~\ref{lem:truncatedhypergraph} (1) that $H_J'$ has an Euler family $\F_J$ with $|\F_J|=|\F_J''|$.

Finally, from Lemma~\ref{lem:G}, we have that 
\begin{align*}
|\F|&=\sum\limits_{K\in\mathcal{P}_G}|\F_K'|\\
&=\sum\limits_{\substack{K\subseteq J\\K\in\mathcal{P}_G}}|\F_K'| +\sum\limits_{i\in I\setminus J}|\F_i'|\\
&= |\F_J''|+\sum\limits_{i\in I\setminus J}|\F_i'|\\
&= |\F_J|+\sum\limits_{i\in I\setminus J}|\F_i'|.
\end{align*}

Furthermore, assume now that $\F$ is spanning in $H$.  Then Lemma~\ref{lem:G} (4) implies that, for each $K\in\mathcal{P}_G$, we have that $\F_K'$ is spanning in $H_K'$.  Note that if $K\not\subseteq J$, we have that $H_K'=H_i$ for some $i\in I\setminus J$, so this implies that $\F_i'$, for each $i\in I\setminus J$, is spanning in $H_i$.

Now, if $K\subseteq J$, then we are given that $\F_J''$ is spanning in $H_J''$ by Lemmas~\ref{lem:G} (4) and~\ref{lem:decomposingfamilies} (1).  Hence $\F_J$, constructed from $\F_J''$ using Lemma~\ref{lem:truncatedhypergraph} (1), is spanning in $H_J'$, since $\F_J''$ and $\F_J$ traverse the same vertices and $V(H_J'')=V(H_J')$.

Therefore, the Euler families $\F_J$ and $\F_i$, for all $i\in I\setminus J$, are spanning in their respective hypergraphs if $\F$ is spanning in $H$.

$\Leftarrow$: Assume there exists $J\subseteq I$ with the properties listed in the statement of this theorem.  Let $\F_J$ be a nonempty Euler family of $H\big[\bigcup\limits_{j\in J}V(H_j)\big]$, and let $\F_i$ be an Euler family of $H_i$ for all $i\in I\setminus J$.  Then $\F'=\F_J\cup(\bigcup\limits_{i\in I\setminus J}\F_i)$ is an Euler family of $H':=H\big[\bigcup\limits_{j\in J}V(H_j)\big]\cup(\bigcup\limits_{i\in I\setminus J}H_i)$, since $H'$ is a disjoint union of hypergraphs.  We can observe that $\F'$ is spanning in $H'$ if and only if $\F_J$ and all $\F_i$ are spanning in their respective hypergraphs, and that $|\F'|=|\F_J| + \sum\limits_{i\in I\setminus J}|\F_i|$ since $\F'$ is a disjoint union of sets $\F_J$ and $\F_i$, for $i\in I\setminus J$.

Construct a bijection $\varphi: E(H)\rightarrow E(H')$ by defining, for all $e\in E(H)$:
\[  \varphi(e)= \left\{
\begin{array}{lcl}
      e & : & e\not\in F \\
      e\cap \big(\bigcup\limits_{j\in J}V(H_j)\big) & : & e\in F. \\      
\end{array} 
\right. \] 
Since $F$ is a minimal edge cut, we have that $f\cap V(H_J)\neq\emptyset$ for all $f\in F$ by Lemma~\ref{lem:intersects}.  Hence $\varphi$ is well defined since $f\cap V(H_J)\in E(H')$ for all $f\in F$.  Furthermore, we have that $\varphi$ is a bijection because each edge $f\in F$ is canonically mapped to a distinct $f\cap V(H_J)$ in the construction of the induced subhypergraph $H\big[\bigcup\limits_{j\in J}V(H_j)\big]$, and $\varphi$ does the same.

Then, since $V(H')=V(H)$ and $\varphi$ is a bijection with $\varphi(e)\subseteq e$ for all $e\in E(H)$, we may apply Lemma~\ref{lem:truncatedhypergraph} (1) (using $\varphi^{-1}$) to obtain an Euler family $\F$ of $H$ that satisfies $|\F|=|\F'|=|\F_J|+\sum\limits_{i\in I\setminus J}|\F_i|$.

Furthermore, if $\F_J$ and each $\F_i$, for $i\in I\setminus J$, are spanning in their respective hypergraphs, then $\F'$ is spanning in $H'$ as observed earlier, and so Lemma~\ref{lem:truncatedhypergraph} (1) states that $\F$ is spanning in $H$.
\end{proof}

\begin{cor}\label{cor:edgecuttour}
Let $H$ be a hypergraph with a nonempty minimal edge cut $F$.  Let $H_i$, for $i\in I$, denote the connected components of $H\setminus F$.  

Then $H$ has a (spanning) Euler tour $T$ if and only if there exists $J\subseteq I$ with $1\leq |J|\leq |F|$ such that the following hold:
\begin{description}
\item[(1)] $H\big[\bigcup\limits_{j\in J}V(H_j)\big]$ has a (spanning) Euler tour $T_J$; and
\item[(2)] $H_i$ is empty for all $i\in I\setminus J$.
\end{description}

Furthermore, if $|I|>|F|$, then $H$ does not have a spanning Euler tour.
\end{cor}

\begin{proof}
$\Rightarrow$: Let $T$ be an Euler tour of $H$.  Then $\F=\{T\}$ is an Euler family of $H$ of cardinality 1. By Theorem~\ref{thm:edgecutfamily}, there exists $J\subseteq I$ with $1\leq |J|\leq |F|$ such that $H\big[\bigcup\limits_{j\in J}V(H_j)\big]$ has an Euler family $\F_J$ of cardinality 1, which is spanning if $\F$ is, and each $H_i$, for $i\in I\setminus J$, has an empty Euler family.

This implies that $H\big[\bigcup\limits_{j\in J}V(H_j)\big]$ has an Euler tour $T_J$ that is spanning if $T$ is, and each $H_i$ is empty, for $i\in I\setminus J$.

$\Leftarrow$: Assume we have $J\subseteq I$ with $1\leq |J|\leq |F|$ satisfying properties (1) and (2) in the statement of this corollary.  Then $\F_J=\{T_J\}$ is an Euler family of $H\big[\bigcup\limits_{j\in J}V(H_j)\big]$, and each $H_i$ has an empty Euler family $\F_i$ for $i\in I\setminus J$.  We may apply Theorem~\ref{thm:edgecutfamily} to obtain an Euler family $\F$ of $H$ of cardinality 1, so $\F$ gives rise to an Euler tour of $H$.

Finally, suppose that $T$ is a spanning Euler tour of $H$, but $|I|>|F|$.  Then Theorem~\ref{thm:edgecutfamily} implies that there exists an empty spanning Euler family of $H_i$ for all $i\in I\setminus J$.  Since $I\setminus J$ is nonempty, this implies the existence of an empty spanning Euler family; however, an empty Euler family, by definition, cannot be spanning.  Therefore, there cannot be any spanning Euler tour of $H$ if $|I|>|F|$.

\end{proof}

\section{Eulerian Properties of Hypergraphs with a Cut Edge}

Situations in which our hypergraph $H$ has a cut edge are rather straightforward.  In order for the cut edge $f$ to be traversed in a closed trail, it must be traversed via two vertices in the same connected component of $H\setminus f$.  What we state here is simply a specific example of Theorem~\ref{thm:edgecutfamily}.\\

\begin{thm} Let $H$ be a connected hypergraph with a cut edge $f$.  Let $H_i,$ for $i\in I$, be the connected components of $H\setminus f$.  Then
\begin{description}
\item[(1)] $H$ has a (spanning) Euler family if and only if there exists $i\in I$ such that
	\begin{itemize}
	\item $H[V(H_i)]$ has a nonempty (spanning) Euler family; and
	\item $H_j$ has a (spanning) Euler family for all $j\neq i$.
	\end{itemize}
\item[(2)] $H$ has an Euler tour if and only if there exists $i\in I$ such that
	\begin{itemize}
	\item $H[V(H_i)]$ has an Euler tour; and
	\item $H_j$ is empty for all $j\neq i$.
	\end{itemize}
\item[(3)] $H$ has no spanning Euler tour.
\end{description}
\end{thm}

\begin{proof}
These results follow directly from Theorem~\ref{thm:edgecutfamily} and Corollary~\ref{cor:edgecuttour} when $|F|=1$.
\end{proof}

\section{Euler Tours in Hypergraphs With Edge Cuts of Cardinality 2}

When our hypergraph $H$ has an edge cut $F$ with just two edges in it, we can apply a new strategy involving some collapsed hypergraphs.  This strategy enables us to find Euler tours in a new way that is usually computationally less expensive than the strategy of Corollary~\ref{cor:edgecuttour}, providing us with an additional avenue of attack.  We require that $H\setminus F$ have exactly two nonempty connected components; however, this is not much of an additional assumption.  If $H\setminus F$ has more than two nonempty connected components, then of course there cannot be an Euler tour at all, by Corollary~\ref{cor:edgecuttour}; if it has fewer than two, then the strategy of this theorem ends up being identical to the one employed by Corollary~\ref{cor:edgecuttour}.\\

\begin{thm}\label{thm:F=2}
Let $H$ be a hypergraph with a minimal edge cut $F=\{f_1,f_2\}$ of cardinality 2. Let $H_i$, for $i\in I$, denote the connected components of $H\setminus F$.  Further, assume that $H_1$ and $H_2$ are nonempty and that $H_i$ is empty for all $i\not\in\{1,2\}$.

For each $i\in \{1,2\}$, let $H_i^*$ be the hypergraph $H\circ\big(\bigcup\limits_{j\neq i}V(H_j)\big)$. Then $H$ has an Euler tour $T$ if and only if each $H_i^*$, for $i=1,2$, has an Euler tour $T_i$ that traverses the collapsed vertex of $H_i^*$.

Furthermore, we have that $T$ is spanning in $H$ if and only if $|I|=2$ and $T_i$, for $i=1,2$, is spanning in $H_i^*$.
\end{thm}

\begin{proof}
$\Rightarrow:$ Let $T$ be an Euler tour of $H$.  We will use the notation from Definition~\ref{defn:G} with respect to $G=G(H,F,\{T\})$, although we will not use $G$ itself.  By Corollary~\ref{cor:edgecuttour}, there exists $J\subseteq I$ with $|J|\leq 2$ such that $H_J'=H\big[\bigcup\limits_{j\in J} V(H_j)\big]$ has an Euler tour $T_J$ and each $H_i$ is empty, for $i\in I\setminus J$.  By our assumption that $H_1$ and $H_2$ are nonempty, we have $J=\{1,2\}$.  If $T$ is spanning in $H$, then Corollary~\ref{cor:edgecuttour} tells us that $T_J$ is spanning in $H_J'$.

Let $F_J'=\{f_1', f_2'\}$.  We claim that $F_J'$ must be a minimal edge cut of $H_J'$.  It is certainly an edge cut because, for example, we can write $F_J'=\{e\in E(H_J'):e\text{ intersects both }V(H_1)\text{ and }V(H_2)\} = \{f_1',f_2'\}$. 

Suppose, however, that $F_J'$ is not a minimal edge cut of $H_J'$.  Without loss of generality, assume $f_1'$ is a cut edge of $H_J'$.  Observe that $f_2$ intersects $H_1$ and $H_2$, so $f_2'$ does as well.  Since $H_1$ and $H_2$ are connected subhypergraphs of $H_J'$, we have that $(H_1\cup H_2)+f_2'$ is connected.  However, we have $(H_1\cup H_2)+f_2' = H_J'\setminus f_1'$, so $f_1'$ is not a cut edge of $H_J'$, contradicting our assumption.  We conclude that $F_J'$ is a minimal edge cut of $H_J'$.

Construct $G'=(H_J',F_J',\{T_J\})$.  Now, Lemma~\ref{lem:G} parts (1) and (5) imply that $G'$ has two vertices, two edges, and is even and connected.  We conclude that $G'$ is a 2-cycle, so it has no loops.

Since $H_J'$ has an Euler tour and $G'$ is connected and loopless, Lemma~\ref{lem:collapsedG} states that $H_J'\circ V(H_2)$ and $H_J'\circ V(H_1)$ have Euler families of cardinality at most 1.  These cardinalities must both be equal to 1, since neither $H_J'\circ V(H_2)$ nor $H_J'\circ V(H_1)$ are empty.  Let $T_1$ and $T_2$ be Euler tours of $H_J'\circ V(H_2)$ and $H_J'\circ V(H_1)$, respectively.  

Observe that $H_J'\circ V(H_2)$ is isomorphic to $H_1^*$: first of all, their vertex sets are $V(H_1)$ together with some collapsed vertex $v$.  Since $F$ is a minimal edge cut of $H$, we have that any edge containing a vertex of $H_i$ for $i\geq 3$, also contains vertices of $H_1$ and $H_2$.  Hence the edge sets of these collapsed hypergraphs will both consist of $E(H_1)$ along with $(f_1\cap V(H_1))\cup \{v\}$ and $(f_2\cap V(H_1))\cup \{v\}$.  A similar argument shows that $H_J'\circ V(H_1)$ is isomorphic to $H_2^*$.

Therefore, we conclude that each $T_i$ corresponds to an Euler tour $T_i^*$ of $H_i^*$, for $i=1,2$.  Furthermore, since $H_1$ and $H_2$ are nonempty, we know that $T$ traverses vertices of each.  Then $T_1$ and $T_2$ traverse the collapsed vertex in their respective hypergraphs by Lemma~\ref{lem:collapsedG} (3), so $T_1^*$ and $T_2^*$ do as well.

In addition, again by Lemma~\ref{lem:collapsedG}, if $T$ is a spanning Euler family of $H$, then $T_1^*$ and $T_2^*$ are spanning in $H_1^*$ and $H_2^*$, respectively, and by Corollary~\ref{cor:edgecuttour}, we have $|I|=2$.

$\Leftarrow:$ For each $i=1,2$, let $v_i$ be the collapsed vertex of $H_i^*$, and let $T_i^*$ be an Euler tour of $H_i^*$ traversing $v_i$.

Let $f_i^j=(f_i\cap V(H_j))\cup \{v_j\}$, for each $i,j\in\{1,2\}$, be the edge of $H_j^*$ corresponding to $f_i$.

Since $T_1^*$ traverses $v_1$, and $v_1$ has only two incident edges, namely $f_1^1$ and $f_2^1$, we can write $T_1^*=v_1f_1^1uRvf_2^1v_1$, where $u,v\in V(H_1)$ and $R$ is a trail in $H_1$ that traverses every edge of $H_1$.  Similarly, we can write $T_2^*=v_2f_2^2wSxf_1^2v_2$, where $w,x\in V(H_2)$ and $S$ is a trail in $H_2$ that traverses every edge of $H_2$.

Then $T=uRvf_2wSxf_1u$ is an Euler tour of $H$.

If $T_1^*$ and $T_2^*$ are both spanning, then $R$ and $S$ traverse all the vertices of $H_1$ and $H_2$, respectively.  Hence $T$ traverses all the vertices in $V(H_1)\cup V(H_2)$.  If $|I|=2$, then $V(H_1)\cup V(H_2)=V(H)$, so $T$ is spanning as well.
\end{proof}

\section{Eulerian Properties of Hypergraphs with Edge Cuts Whose Edges Have Cardinality 2}

We now turn our attention to something more unusual: edge cuts whose edges all have cardinality 2.  Investigating such edge cuts gives us a lot of information about how these edges are traversed, so it is in some sense a powerful property.  As it turns out, we will get some mileage out of collapsed hypergraphs in this section, too.  Whenever we are able to use a collapsed hypergraph, the computational cost of searching for an Euler tour is reduced, so we are motivated to do so whenever we can.  However, our technique will not help us find Euler tours, so we will have to settle for finding Euler families.\\

\begin{thm}\label{thm:edgescard2}
Let $H$ be a hypergraph with a minimal nonempty edge cut $F$ such that $|e|=2$ for all $e\in F$.  Let $H_i$, for $i\in I$, denote the connected components of $H\setminus F$.  

For each $i\in I$, let $H_i^*$ be the hypergraph $H\circ\{V(H_j):j\neq i\}$.  Then the following hold:
\begin{description}
\item[(1)] $|I|=2$; and
\item[(2)] $H$ has a (spanning) Euler family if and only if $H_i^*$ has a (spanning) Euler family for $i=1,2$.
\end{description}
\end{thm}

\begin{proof}
{\bf (1)} Lemma~\ref{lem:intersects} implies that, since $F$ is a minimal edge cut, every edge of $F$ intersects each $H_i$ for $i\in I$.  Since each of these edges has cardinality 2, there must be only two connected components of $H\setminus F$, hence $|I|=2$.

{\bf (2)}$\Rightarrow$: Let $\F$ be an Euler family of $H$.  Construct $G=G(H,F,\F)$.  Then $G$ consists of two vertices with $|F|$ parallel edges joining them.

Since $G$ is connected and has no loops, Lemma~\ref{lem:collapsedG} implies that $H_1^*=H\circ V(H_2)$ and $H_2^* = H\circ V(H_1)$ have Euler families $\F_1$ and $\F_2$, respectively.  Furthermore, if $\F$ is spanning in $H$, then Lemma~\ref{lem:collapsedG} implies that $\F_1$ and $\F_2$ are spanning in their respective hypergraphs as well.

$\Leftarrow$: Let $\F_i^*$ be an Euler family for $H_i^*$, for $i=1,2$.  Let $h_2\in V(H_1^*)$ be the collapsed vertex corresponding to $H_2$, and $h_1\in V(H_2^*)$ be the collapsed vertex corresponding to $H_1$.

Let $T_1$ be the (unique) trail of $\F_1^*$ traversing $h_2$ and let $T_2$ be the trail of $\F_2^*$ traversing $h_1$.  Note that $\deg_{H_1^*}(h_2)=|F| = \deg_{H_2^*}(h_1)$, so we write $T_1$ and $T_2$ as a concatenation of closed trails, as follows:
$$T_1=h_2S_0h_2S_1h_2\dotso h_2S_{\frac{|F|}{2} - 1}h_2,\\$$
$$T_2=h_1R_0h_1R_1h_1\dotso h_1R_{\frac{|F|}{2} - 1}h_1,$$ where each $S_i$ is a nontrivial closed strict trail whose internal vertices and edges are in $H_1$, and similarly for each $R_i$ in $H_2$, for all $i\in\mathds{Z}_{\frac{|F|}{2}}$.

We can further write $S_i=h_2e_{2i}u_{2i}S_i'u_{2i+1}e_{2i+1}h_2$ and $R_i=h_1f_{2i}v_{2i}R_i'v_{2i+1}f_{2i+1}h_1$ for each $i\in\mathds{Z}_{\frac{|F|}{2}}$, where 

\begin{itemize}
\item $S_i'$ is a trail in $H_1$ and $R_i'$ is a trail in $H_2$;
\item $u_0,\dotso , u_{|F|-1}\in V(H_1)$ and $v_0,\dotso , v_{|F|-1}\in V(H_2)$; and
\item $\{e_0, \dotso , e_{|F|-1}\} = F = \{f_0,\dotso , f_{|F|-1}\}$.
\end{itemize}

Now, let $G$ be the incidence graph of $H$.  Let $(S_i')_G$ and $(R_i')_G$ denote the trail of $G$ corresponding to $S_i'$ or $R_i'$, respectively, for each $i\in\mathds{Z}_{|F|}$, as described in Remark~\ref{remark:traversals}.  Similarly, let $T_G$ represent the trail in $G$ corresponding to any closed trail $T$ of $\F_1^*\setminus T_1$ or $\F_2^*\setminus T_2$.  Such a trail does correspond to a subgraph of $G$ because it traverses only vertices and edges of $H_1$ or $H_2$, which are subhypergraphs of $H$.

Consider the subgraph $G'$ of $G$ whose vertex set is $V(G)$ and whose edge set is $$\big(\bigcup\limits_{i\in\mathds{Z_{|F|}}}\big(E((S_i')_G)\cup E((R_i')_G))\big)\cup\big(\bigcup\limits_{T\in(\F_1^*\setminus T_1)\cup(\F_2^*\setminus T_2)}E(T_G)\big).$$

Observe that $G'$ is essentially an e-vertex-disjoint union of trails of $G$, and these trails jointly traverse every e-vertex of $G$ exactly once, except for those e-vertices corresponding to edges of $F$.  Hence the degree of any e-vertex $e$ in $G$ is 2 if $e\not\in F$, and 0 if $e\in F$.  The degree of a v-vertex $v$ is a bit more complicated: it is twice the number of times $v$ is traversed by one of the constituent trails as an internal vertex, plus the number of times $v$ is traversed by some $R_i'$ or $S_i'$ as an initial or terminal vertex.

Let $F'=\{ue, ev: e=uv\in F\}$, so $F'$ is a set of edges of $G$.  Note that $E(G')\cap F'=\emptyset$ because the e-vertices corresponding to edges of $F$ are isolated in $G'$.  Then $G'+F'$ is a subgraph of $G$ in which every e-vertex has degree 2, and, for any v-vertex $v$, we have $\deg_{(G'+F')}(v) = \deg_{G'}(v) + |\{e\in F: v\in e\}|$.  Note that $|\{e\in F:v\in e\}|$ is equal to the number of times that $v$ is traversed by some $R_i'$ or $S_i'$ as an initial or terminal vertex because precisely the vertices of edges of $F$ are traversed in that way.  Hence $\deg_{(G'+F')}(v)$ is even for each v-vertex $v$.  Then $G'+F'$ corresponds to an Euler family $\F$ of $H$ by Theorem~\ref{thm:incidencegraph}.



Finally, if $\F_1^*$ and $\F_2^*$ are spanning in their respective hypergraphs, then $G'+F'$ has no v-vertices of degree 0, so $\F$ must be spanning in $H$.

\end{proof}

The reason that Euler tours elude us in the above proof is that we cannot necessarily thread together our trails of the collapsed hypergraphs in a way that produces only one trail in $H$, even if we insist on having an Euler tour in each collapsed hypergraph. 

\section{Using Edge Cuts to Compute an Euler Tour}

As discussed in Problem~\ref{problem:eulerfamily}, the problem of determining whether a hypergraph has an Euler family is in P.  As such, though our results in this chapter could contribute to improvements in such algorithms, it is more useful to put these results to work in an algorithm to solve the Euler tour problem, which is NP-complete (see {\em e.g.} Theorem~\ref{thm:complexity}).

We use the results from this chapter to produce a branch-and-bound algorithm, called {\fontfamily{lmtt}\selectfont findEulerTour} (Algorithm~\ref{algo:1}), that improves upon the na{\" i}ve brute-force algorithm. {\fontfamily{lmtt}\selectfont findEulerTour} relies on the ability to compute minimal edge cuts, which we know can be done efficiently \cite{CC}.  {\fontfamily{lmtt}\selectfont findEulerTour} makes use of the results we have accumulated in this chapter and recursively pares down the hypergraph by removing vertices from edges.  It then applies a brute-force algorithm on the pared-down hypergraph.  We assume that a brute-force algorithm, which we call {\fontfamily{lmtt}\selectfont bruteForceEulerTour}, is available, but do not describe it here.  The search space is exponential in the number of edges of the hypergraph, as well as the cardinalities of each edge, so it is crucial to cut down on these values wherever we can. 

A large part of the simplification process relies on the ``choice function'' $\alpha$ described in Theorem~\ref{thm:fix}.  We can prune the search space by making sure that $\alpha$ satisfies some of the simple necessary requirements described in the theorem, before attempting an expensive search utilizing it.  The goal is to produce the graph $G_{\alpha}=G(H,F,\F)$ before we even have an Euler family (or in our case, an Euler tour) $\F$.

{\fontfamily{lmtt}\selectfont findEulerTour} cannot simplify the hypergraph any further once we have an edge cut whose deletion disconnects the hypergraph into just two connected components.

For organizational purposes, we also make use of three helper functions, and they each correspond to one of the results from this chapter. They are {\fontfamily{lmtt}\selectfont findETCollapsed}, for when we apply Theorem~\ref{thm:F=2}; {\fontfamily{lmtt}\selectfont findETOneComponent}, for when we apply Corollary~\ref{cor:edgecuttour}; and {\fontfamily{lmtt}\selectfont findETLargeEdgeCut}, for when we use Corollary~\ref{cor:fix}.

\begin{algorithm}[H]
\BlankLine
\begin{algo}\label{algo:1} 
\end{algo}

\SetKwFunction{FindET}{findEulerTour}
\SetKwFunction{BruteForceET}{bruteForceEulerTour}
\SetKwFunction{FindETCollapsed}{findETCollapsed}
\SetKwFunction{FindETOneComp}{findETOneComponent}
\SetKwFunction{FindETLargeCut}{findETLargeEdgeCut}
\SetKwFunction{init}{// initialization}
\SetKwFunction{Null}{null}
\SetKwFunction{Return}{return}
\SetKwInOut{Input}{input}
\SetKwInOut{Output}{output}

\FindET function: \\*
\Input{A connected hypergraph $H$ with at least two vertices}
\Output{An Euler tour of $H$, or \Null if $H$ is not eulerian}
\BlankLine
\init :\\*
$H \gets H-\{v\in V(H): \deg(v)=1\};$ \CommentSty{// Delete vertices of degree 1}\\*
\lIf{$H$ has edges of cardinality 1 or $H$ has no vertices}{\Return\Null}
$F \gets$ minimal edge cut of $H$\;
$H_i \gets$ connected components of $H\setminus F$, for $i\in I$, indexed from 1 and sorted in nonincreasing order of size\;
$k \gets$ number of nonempty connected components of $H\setminus F$\;
\BlankLine
\If{$k>|F|$} {
	\CommentSty{// Not possible to have an Euler tour (Corollary~\ref{cor:edgecuttour})}\\
	\Return \Null\;
} 
\CommentSty{// Note: $|F|=1, k\leq 1$ not possible due to line 3}\\*
\uIf{$|F|=2$} {
	\uIf{$k = 2$} {
		\BlankLine
		\CommentSty{// Check collapsed hypergraphs according to Theorem \ref{thm:F=2}}\\*
		$T_1,T_2 \gets$ \FindETCollapsed($H,H_1,H_2$)\;
		\uIf{$T_1=\Null$ or $T_2=\Null$} {
			\Return \Null
		}
		\Else {
			\BlankLine
			\CommentSty{// We have Euler tours in both collapsed hypergraphs}\\*
			\Return Euler tour of $H$ assembled from $T_1$ and $T_2$\;
		}
	}
	\uElseIf{$k=1$} {
		\BlankLine
		\CommentSty{/*}\\
		\CommentSty{ * We can get rid of all empty connected components of $H\setminus F$}\\
		\CommentSty{ * except for one, then search for an Euler tour.  This}\\
		\CommentSty{ * corresponds to Corollary \ref{cor:edgecuttour}, where $|J|=1$ or 2.}\\
		\CommentSty{ */}\\
		\Return$\FindETOneComp(H, F, H_1, H_2)$\;
	}
	\Else(\CommentSty{// $k=0$, so $H$ has just two edges}){
		\Return any 2-cycle\;
	}
}
\Else(\CommentSty{// if $|F|\geq 3$}){
	\uIf(\CommentSty{// We cannot reduce any further}){$|I|=2$}{
		\Return$\BruteForceET(H)$\;
	}
	\Else {
		\Return$\FindETLargeCut(H,F)$\;
	}
}

\end{algorithm}

\begin{algorithm}[H]
\SetKwFunction{FindET}{findEulerTour}
\SetKwFunction{FindETCollapsed}{findETCollapsed}
\SetKwFunction{Null}{null}
\SetKwFunction{Return}{return}
\SetKwInOut{Input}{inputs}
\SetKwInOut{Output}{output}

\FindETCollapsed function (called on line 15 of \FindET): \\*
\Input{\begin{description}
	\item[1)] A connected hypergraph $H$ with at least two vertices and no vertices of degree 1;
	\item[2-3)] Disjoint subhypergraphs $H_1$ and $H_2$ of $H$ that represent the nontrivial connected components of $H$ when an edge cut of cardinality 2 is deleted
	\end{description}}
\Output{A pair of closed trails $T_1,T_2$ such that each $T_i$ is either an Euler tour of $H\circ(V(H)\setminus V(H_i))$, or \Null if such an Euler tour does not exist}
\BlankLine

\For{$i=1$ to $2$} {
	$collapsedH_i\gets H\circ(V(H)\setminus V(H_i))$\;
	$v_i\gets$ collapsed vertex of $collapsedH_i$\;
	$T_i\gets\Null$\;
	\BlankLine
	\ForEach{edge $e$ containing $v_i$ in $collapsedH_i$} {
		\ForEach{vertex $v\in e\setminus v_i$} {
			\If{$T_i=\Null$} {
				\CommentSty{/*}\\
				\CommentSty{ * Try to find Euler Tour of $H\circ(V(H)\setminus V(H_i))$ that}\\
				\CommentSty{ * traverses $e$ via $v$ and $v_i$:}\\
				\CommentSty{ */}\\
				$T_i\gets\FindET(collapsedH_i\setminus e + \{v,v_i\})$\;
			}
		}
	}
}
\Return $T_1,T_2$\;
\end{algorithm}

\begin{algorithm}[H]
\SetKwFunction{FindET}{findEulerTour}
\SetKwFunction{BruteForceET}{bruteForceEulerTour}
\SetKwFunction{FindETCollapsed}{findETCollapsed}
\SetKwFunction{init}{// initialization}
\SetKwFunction{Null}{null}
\SetKwFunction{Return}{return}
\SetKwInOut{Input}{inputs}
\SetKwInOut{Output}{output}

\FindETOneComp function (called on line 27 of \FindET): \\*
\Input{
	\begin{description}
		\item[1)] A connected hypergraph $H$ with at least two vertices and no vertices of degree 1;
		\item[2)] An edge cut $F$ of $H$ of cardinality 2;
		\item[3)] $H_1$, the sole nonempty connected component of $H\setminus F$;
		\item[4)] $H_2$, an empty connected component of $H\setminus F$
	\end{description}
}
\Output{An Euler tour of $H$, or \Null if $H$ is not eulerian}
\BlankLine

\CommentSty{/*}\\
\CommentSty{ * An Euler tour of $H$ either doesn't traverse the vertex }\\
\CommentSty{ * of $H_2$, or it does.  We check these possibilities in order.}\\
\CommentSty{ */}\\

\init:\\
Let $F=\{f_1,f_2\}$\;
$f_1\gets f_1\cap V(H_1)$\;
$f_2\gets f_2\cap V(H_1)$\;
Let $v$ be the vertex of $H_2$\;
$T\gets\Null$\;

\If{$|f_1|\geq 2$ and $|f_2|\geq 2$} {
	\CommentSty{/*}\\
	\CommentSty{ * It might be possible to circumvent $v$.  We had to make}\\
	\CommentSty{ * sure that $f_1$ and $f_2$ are edges in $H[V(H_1)]$.}\\
	\CommentSty{ */}\\
	$T\gets\FindET(H[V(H_1)]$)\;
}
\If{$T=\Null$} {
	\CommentSty{/*}\\
	\CommentSty{ * We will try to find a traversal through $v$, but we must }\\
	\CommentSty{ * brute force it, or else we will get stuck in this case. }\\
	\CommentSty{ */}\\
	\ForEach{vertex $a\in f_1$ and vertex $b\in f_2$}{
		$T\gets\BruteForceET((H\setminus F) + \{av,bv\})$\;
		\lIf{$T\neq\Null$}{\Return $T$}
	}
}
\Return $T$\;
\end{algorithm}

\begin{algorithm}[H]
\BlankLine

\SetKwFunction{FindET}{findEulerTour}
\SetKwFunction{BruteForceET}{bruteForceEulerTour}
\SetKwFunction{FindETLargeCut}{findETLargeEdgeCut}
\SetKwFunction{init}{// initialization}
\SetKwFunction{Null}{null}
\SetKwFunction{Construct}{construct}
\SetKwFunction{Return}{return}
\SetKwInOut{Input}{inputs}
\SetKwInOut{Output}{output}

\FindETLargeCut function (called on line 34 of \FindET): \\*
\Input{
	\begin{description}
		\item[1)] A connected hypergraph $H$ with at least two vertices and no vertices of degree 1;
		\item[2)] A minimal edge cut $F$ of $H$ of cardinality at least 3
	\end{description}
}
\Output{An Euler tour of $H$, or \Null if $H$ is not eulerian}
\BlankLine

\CommentSty{/*}\\
\CommentSty{ * Here, we are making use of Corollary~\ref{cor:fix} to choose }\\
\CommentSty{ * how to traverse each edge of the cut.  We can tell right away}\\
\CommentSty{ * that some choice functions won't work.}\\
\CommentSty{ */}\\
\ForEach{function $\alpha:F\rightarrow\{ij:i,j\in I\}$}{
	\BlankLine
	\CommentSty{// Construct a multigraph $G_{\alpha}$ based on $\alpha$}\\
	$G_{\alpha}\gets (I,\alpha(F))$;\CommentSty{        // $\alpha(F)$ is the multiset image of $F$}\\
	\If{$G_{\alpha}$ is even and $G_{\alpha}$ has just one nonempty connected component}{
		\BlankLine
		\CommentSty{// Trim the edges of $F$ according to $\alpha$}\\
		\ForEach{edge $e\in F$} {
			$e'\gets e\cap\big(\bigcup\limits_{i\in\alpha(e)} V(H_i)\big)$\;
		}
		$F'\gets\{e':e\in F\}$\;
		\BlankLine
		\CommentSty{/*}\\
		\CommentSty{ * If $F'=F$, then we have not made any improvements}\\
		\CommentSty{ * to $H$, so we can just brute force a solution now.}\\
		\CommentSty{ */}\\
		\If{$F'=F$} {\Return$\BruteForceET(H)$\;}
		\BlankLine
		\CommentSty{/*}\\
		\CommentSty{ * It is possible that $F'$ now has edges of cardinality}\\
		\CommentSty{ * 1.  If it doesn't, we can continue to try to find}\\
		\CommentSty{ * more reductions.  If it does, just go to the next choice }\\
		\CommentSty{ * function.}\\
		\CommentSty{ */}\\
		\If{$(H\setminus F)+F'$ is connected and has no edges of cardinality 1}{
			$T\gets\FindET((H\setminus F)+F')$\;
			\lIf{$T\neq\Null$}{\Return $T$}
		}
	}
}
\BlankLine
\CommentSty{// If we get to this point, then $H$ has no Euler tour}\\
\Return\Null\;
\end{algorithm}

\newpage
\begin{remark}\label{remark:algo}{\rm
Though {\fontfamily{lmtt}\selectfont findEulerTour} can undoubtedly be improved, we have chosen to present it in a relatively basic form so that it can serve as a proof-of-concept.  We first remark that a {\em minimum} edge cut in a hypergraph, not just a minimal one, can be found efficiently \cite{CC}.  This means that the parts of the algorithm that deal with smaller edge cuts will be used more often, to our benefit, since we have more effective tools for such small edge cuts.

Because minimum edge cuts can be found efficiently, {\fontfamily{lmtt}\selectfont findEulerTour} can perhaps be improved by choosing a random minimum edge cut, or a ``best'' minimum edge cut.  Currently, we assume for now that it is chosen deterministically and indiscriminately.  The deterministic nature of the algorithm does give it some weaknesses, such as the necessity to resort to brute forcing quite early, so redesigning it to take advantage of random edge cuts could pay off.  A best edge cut might be one that includes edges of large cardinality because we can trim those edges earlier in the recursion rather than later or not at all.

Aside from randomness, there are other ways to possibly improve the efficacy of {\fontfamily{lmtt}\selectfont findEulerTour}. We have chosen to treat all edge cuts of cardinality 3 or greater in the same way as each other.  However, we could find an extension of Theorem~\ref{thm:F=2} to edge cuts of cardinality greater than 2 by investigating the possible cases for $G_{\alpha}$: an Euler tour is only possible if $G_{\alpha}$ is even and has just one nonempty connected component.  While the number of non-isomorphic even connected multigraphs of size $|F|$ is only two when $|F|=2$, this number (only) grows to three for $|F|=3$; seven for $|F|=4$; twelve when $|F|=5$; and at least twenty-eight when $|F|=6$: it is certainly within reach to extend these results a little further!  Note that, though there are two cases when $|F|=2$, we investigated the 2-cycle case in Theorem~\ref{thm:F=2}.  The other case is a single vertex with two loops, which, in practice, is treated similarly to the $|F|=1$ case.

Specifically, when $|F|=3$, the only extra case we must account for is when $G_{\alpha}$ is a 3-cycle.  In that case, we would investigate three collapsed hypergraphs --- $H\circ\{V(H_j):j\neq i\}$ for each $i=1,2,3$ --- and reassemble an Euler tour for $H$ from an Euler tour in each collapsed hypergraph that traverses all the collapsed vertices.  In fact, the author believes that this strategy generalizes whenever $G_{\alpha}$ is a cycle or a cycle with additional loops, and this would certainly be the next point of investigation for future research.
}
\end{remark}
\cleardoublepage

\chapter{Concluding Remarks}

\section{On Design Hypergraphs}
We have seen a few results about eulerian properties of design hypergraphs since the turn of the century.  Dewar and Stevens showed that certain kinds of triple systems are eulerian, using the terminology of universal cycles (Theorems~\ref{thm:ds1} and~\ref{thm:ds2} \cite{DS}), while Horan and Hurlbert constructively proved the existence of eulerian Steiner triple systems and Steiner quadruple systems of every admissible order, using the language of overlap cycles (Theorems~\ref{thm:hh1} \cite{HH} and~\ref{thm:hh2} \cite{HH2}).  Bahmanian and \v{S}ajna extended the language for Euler tours and families in particular, and answered the question of whether triple systems (among many other kinds of hypergraphs) are quasi-eulerian in the affirmative (Corollary~\ref{cor:quasieuleriancovering} \cite{BS2}).

We have expanded on these results by (non-constructively) proving that all Steiner triple systems are eulerian in Chapter~\ref{chapter:STS} (Theorem~\ref{thm:STS}), then proving that all covering $k$-hypergraphs, for $k\geq 3$, are eulerian in Chapter~\ref{chapter:covering} (Theorem~\ref{thm:coveringinduction}), which includes Steiner quadruple systems.  

Chapters \ref{chapter:STS}, \ref{chapter:covering}, and \ref{chapter:quasi-eulerian} culminate to yield the following result on design hypergraphs.\\

\begin{thm}
Let $H$ be an $\ell$-covering $k$-hypergraph with $2\leq\ell < k$.  If $H$ has at least two edges, then $H$ is quasi-eulerian; if we additionally have $\ell=k-1$, then $H$ is eulerian.
\end{thm}

While it seems plausible that all $\ell$-covering $k$-hypergraphs are eulerian, there is as of yet no clear technique that would allow us to prove this.  However, due to the highly structured nature of such hypergraphs, we will make the following conjecture.\\

\begin{conjecture}
All $\ell$-covering $k$-hypergraphs with $2\leq\ell <k$ and at least two edges are eulerian.
\end{conjecture}

Some consideration was also given to a sort of ``pairwise-balanced block design,'' in which all $\ell$-tuples of vertices lie in at least one edge together, but edges have cardinalities taken from some permissible set.  Using an appropriate strengthening of Lov\'{a}sz's $(g,f)$-factor Theorem, as we have seen for $\ell$-covering $k$-hypergraphs, one may be able to prove that such hypergraphs are quasi-eulerian.\\

\begin{question}
Let $H$ be a hypergraph of order at least 3 and size at least 2. Let $\ell\geq 2$, and let $K\subseteq\mathds{N}_{\geq 2}$ be a set of the cardinalities of the edges in $E(H)$, each 2 or greater.  Assume that every $\ell$-subset of $V(H)$ lie together in at least one edge of $H$.  Under what conditions on $K$ is $H$ quasi-eulerian?  Under what conditions on $K$ is $H$ eulerian?
\end{question}

\section{On Edge Cuts}
On the side of edge cuts, there have already been results on the topic of using vertex cuts to reduce the problem of existence of spanning Euler family or tour, by Steimle and \v{S}ajna \cite{SS}.  Our work in Chapter~\ref{chapter:edgecuts} is a continuation of that.  Their results showed necessary and sufficient conditions for a spanning Euler family or tour to exist by checking whether certain derived hypergraphs admit a spanning Euler family or tour.

In Chapter~\ref{chapter:edgecuts}, we produced results analogous to theirs that are focused on edge cuts instead of vertex cuts.  Perhaps the most interesting avenue for future research builds off of Theorem~\ref{thm:F=2}, which looks into collapsed hypergraphs for Euler tours when an edge cut of cardinality 2 exists.  We have touched upon this in Remark~\ref{remark:algo}, but we would like to ask how the proof strategy might extend to larger edge cuts.\\

\begin{question}
For each $m\geq 0$, how many connected even multigraphs of size $m$ are there, down to isomorphism?  That is, how many possible configurations are there for $G=G(H,F,\F)$ if $|F|=m$?  For which of these multigraphs $G$ can we (simply) use the collapsed hypergraphs to reduce the Euler tour problem to finding an Euler tour in each collapsed hypergraph?
\end{question}

There is, at time of writing, no sequence registered in the On-line Encyclopedia of Integer Sequences (OEIS) \cite{OEIS} concerning the number of connected even multigraphs of size $m=0,1,2,\dotso$, and, it must be admitted, the author finds this question quite captivating.  If one were able to classify a number of these multigraphs, it could be put to use in Algorithm~\ref{algo:1}.
\cleardoublepage

%
%
%
%
%
%
%
%
%
%
%
\appendix

\include{appendix_A}
\cleardoublepage

\include{appendix_B}
\cleardoublepage

%
%
%
\PrintIndex

\end{document}